\documentclass[draft]{amsart}

\usepackage[Symbol]{upgreek}

\usepackage{mathdots}

\usepackage{amssymb}
\usepackage{epic}
\usepackage{mathrsfs}
\usepackage{accents}
\usepackage{stmaryrd}

\usepackage{pgfplots, pgfplotstable}
\pgfplotsset{compat=1.17}

\input{xy}
\xyoption{matrix}
\xyoption{arrow}

\usepackage{etoolbox} 

\numberwithin{equation}{section}

\def\HLO{\Upl\Upo}

\newcommand{\bbold}{\mathbb}

\def\R { {\bbold R} }
\def\Q { {\bbold Q} }
\def\Z { {\bbold Z} }

\def\N { {\bbold N} }
\def\T { {\bbold T} }

\def \cf{\operatorname{cf}}
\def \ci{\operatorname{ci}}
\def \te{\operatorname{te}}
\def \cH{\mathcal{H}}
\def \sign{\operatorname{sign}}
\def \sq{\operatorname{sq}}
 \def \nx{\operatorname{next}}

\def \ex{\operatorname{e}}

\renewcommand\epsilon{\varepsilon}

\def \d{\operatorname{d}}

\def \<{\langle}
\def \>{\rangle}

\def \rc{\operatorname{rc}}

\def \supp {\operatorname{supp}}

\def \((  {(\!(}
\def \)) {)\!)}

\def \Li{\operatorname{Li}}
\def \res{\operatorname{res}}

\def \k {{{\boldsymbol{k}}}}

\DeclareMathSymbol{\precequ}{\mathrel}{symbols}{"16}
\DeclareMathSymbol{\succequ}{\mathrel}{symbols}{"17}

\newcommand{\claim}[2][\!\!]{\medskip\noindent {\it Claim #1}\/: {\it #2}\medskip}

\newtheorem{theorem}{Theorem}[section]
\newcounter{tmptheorem}
\newtheorem{lemma}[theorem]{Lemma}
\newtheorem{prop}[theorem]{Proposition}
\newtheorem{cor}[theorem]{Corollary}

\theoremstyle{definition}

\theoremstyle{remark}

\newtheorem{remarkNumbered}[theorem]{Remark}

\newcommand{\abs}[1]{\lvert#1\rvert}

\def \No{\text{{\bf No}}}

\def \BM{\operatorname{BM}}

\let\oldi\i
\let\oldj\j
\renewcommand\i{\relax\ifmmode{\boldsymbol{i}}\else\oldi\fi}
\renewcommand\j{\relax\ifmmode{\boldsymbol{j}}\else\oldj\fi}

\renewcommand\leq{\leqslant}
\renewcommand\geq{\geqslant}
\renewcommand\preceq{\preccurlyeq}
\renewcommand\succeq{\succcurlyeq}
\renewcommand\le{\leq}
\renewcommand\ge{\geq}
\renewcommand\frak{\mathfrak}

\DeclareMathAlphabet{\mathbf}{OML}{cmm}{b}{it}

\DeclareFontFamily{U}{fsy}{}
\DeclareFontShape{U}{fsy}{m}{n}{<->s*[.9]psyr}{}
\DeclareSymbolFont{der@m}{U}{fsy}{m}{n}
\DeclareMathSymbol{\der}{\mathord}{der@m}{182}

\DeclareSymbolFont{der@m}{U}{fsy}{m}{n}
\DeclareMathSymbol{\derdelta}{\mathord}{der@m}{100}

\newcommand\ndeg{\operatorname{ndeg}}

\DeclareSymbolFont{imag@m}{OT1}{cmr}{m}{ui}
\DeclareMathSymbol{\imag}{\mathord}{imag@m}{105}

\DeclareFontFamily{OMS}{smallo}{}
\DeclareFontShape{OMS}{smallo}{m}{n}{<->s*[.65]cmsy10}{}
\DeclareSymbolFont{smallo@m}{OMS}{smallo}{m}{n}
\DeclareMathSymbol{\smallo}{\mathord}{smallo@m}{79}

\DeclareFontFamily{OMS}{largerdot}{}
\DeclareFontShape{OMS}{largerdot}{m}{n}{<->s*[.8]cmsy10}{}
\DeclareSymbolFont{largerdot@m}{OMS}{largerdot}{m}{n}
\DeclareMathSymbol{\largerdot}{\mathord}{largerdot@m}{15}

\DeclareMathSymbol{\llambda}{\mathord}{der@m}{108}
\DeclareMathSymbol{\rrho}{\mathord}{der@m}{114}

\def \upg{\upgamma}
\def \Upg{\Upgamma}
\def \upl{\uplambda}
\def \Upl{\Uplambda}
\def \upo{\upomega}
\def \Upo{\Upomega}

\newcommand{\equationqed}[1]{\[\pushQED{\qed}#1 \qedhere\popQED\]\let\qed\relax}
\newcommand{\alignqed}[1]{\begin{align*}\pushQED{\qed} #1 \qedhere\popQED\end{align*}\let\qed\relax}

\makeatletter
\newcommand{\dminus}{\mathbin{\text{\@dminus}}}

\newcommand{\@dminus}{%
  \ooalign{\hidewidth\raise1ex\hbox{\bf.}\hidewidth\cr$\m@th-$\cr}%
}
\makeatother

\def\ddeg{\operatorname{ddeg}}

\def \Cc{\mathcal{C}}
\def \C{\mathcal{C}}

\def \Gi{\mathcal{C}^{<\infty}}

\def \inv{\operatorname{inv}}

\begin{document}

\title{Filling Gaps in Hardy Fields}

\author[Aschenbrenner]{Matthias Aschenbrenner}
\address{Kurt G\"odel Research Center for Mathematical Logic\\
Universit\"at Wien\\
1090 Wien\\ Austria}
\email{matthias.aschenbrenner@univie.ac.at}

\author[van den Dries]{Lou van den Dries}
\address{Department of Mathematics\\
University of Illinois at Urbana-Cham\-paign\\
Urbana, IL 61801\\
U.S.A.}
\email{vddries@illinois.edu}

\author[van der Hoeven]{Joris van der Hoeven}
\address{CNRS, LIX (UMR 7161)\\ 
Campus de l'\'Ecole Polytechnique\\  91120 Palaiseau \\ France}
\email{vdhoeven@lix.polytechnique.fr}

\begin{abstract} We show how to fill ``countable'' gaps in Hardy fields. We use this to prove that
any two maximal Hardy fields are back-and-forth equivalent.
\end{abstract}

\date{May 2024}

\maketitle

\tableofcontents

\begingroup
\setcounter{tmptheorem}{\value{theorem}}
\setcounter{theorem}{0} 
\renewcommand\thetheorem{\Alph{theorem}}

\section*{Introduction}\label{intr}

\noindent
By a ``Hardy field'' we mean in this paper a Hardy field at $+\infty$: 
 a subfield $H$ of the ring of germs at $+\infty$ of real valued differentiable functions  on intervals $(a,+\infty)$
 ($a\in\R$)  such that $H$ is closed under differentiation. For basics on Hardy fields, see
 \cite{Rosenlicht83a}.
Each  Hardy field is  an ordered differential field, the (total) ordering given by
$f\leq g$ iff~$f(t)\leq g(t)$ eventually (that is, for all sufficiently large $t$). Among functions whose germs
at $+\infty$ 
live in Hardy fields are all one-variable rational functions with real coefficients, the   real  exponential and logarithm functions (more generally,  Hardy's logarithmico-exponential functions~\cite{Har12a}),   Euler's $\Gamma$-function and Riemann's $\zeta$-function \cite{Rosenlicht83}, and many other ``regularly growing'' functions arising in mathematical practice. As a case in point, by \cite{vdDries} every o-minimal expansion of the ordered field
of real numbers gives rise to a Hardy field  (of germs of definable functions).  
Our main result is as  follows:

\begin{theorem}\label{mt} Let $H$ be a Hardy field, and let $A$, $B$ be countable subsets of $H$ such that $A < B$. Then $A< f < B$
for some $f$ in a
Hardy field extension of $H$.
\end{theorem}

\noindent
Some of the gaps $A<B$ in this theorem correspond to 
pseudo-cauchy sequences ({\em pc-sequences\/} in our abbreviated terminology). The relevant pc-sequences have length $\omega$, and we can handle them using results from our book [ADH] and from~\cite{ADH2} in an essential way, and various glueing techniques. This is done in  Sections~\ref{psh1} and~\ref{psh2}. This dependence on [ADH] and~\cite{ADH2} makes this the deepest part of the present paper, but most of our work here deals with other gaps. 

Sj\"odin~\cite{S} deals with the case $B=\emptyset$ for $\Cc^{\infty}$-Hardy fields (whose elements are germs of $\Cc^{\infty}$-functions). This provides an important clue for other kinds of gaps:  Sj\"odin's construction of a suitable
$f$ can be varied in several ways, and that gives us a handle
on the relevant remaining cases. In Section~\ref{bh} we treat $B=\emptyset$, basically as in \cite{S}, and organized so that it helps in Section~\ref{fwg} where we deal with ``wide'' gaps. For the remaining gaps we use  results about asymptotic couples from \cite{ADH3} and an elaboration of  the ``reverse engineering'' in \cite{S};  see Sections~\ref{vg} and~\ref{fbg}.   (Sections~\ref{prel} and~\ref{sec:analytic} contain mainly  analytic preliminaries, and Section~\ref{secnmh} applies material in Sections~\ref{bh} and~\ref{fwg} to show that there are $2^{\mathfrak{c}}$ many maximal Hardy fields where $\mathfrak{c}=2^{\aleph_0}$ is the cardinality of the continuum. Here and below, ``maximal'' means ``maximal under inclusion''.) 

Most of \cite{ADH2} concerns {\em differentially algebraic\/} extensions of Hardy fields. The present paper complements this with a ``good enough'' overview of {\em differentially transcendental\/} Hardy field extensions $H\<y\>$ of a Liouville closed Hardy field $H\supseteq \R$.

 An equivalent formulation of Theorem~\ref{mt} is that every maximal Hardy field is~$\eta_1$.
The property $\eta_1$ (Hausdorff~\cite{H}) is defined at the end of the introduction. The main result of \cite{ADH2} is that all
maximal Hardy fields, as ordered differential fields, are $\upo$-free newtonian Liouville closed $H$-fields, and thus
by  [ADH, 15.0.2,~16.6.3]
elementarily equivalent to $\T$, the ordered differential
field of transseries. (On $\T$, see [ADH, Appendix~A] or \cite{ADH-ICM}.) Combining this fact with Theorem~\ref{mt}
and a result from \cite{ADH1+} we shall derive in Section~\ref{sec:iso}: 

\begin{cor}\label{cormt} Assuming $\operatorname{CH}$ \textup{(}the Continuum Hypothesis\textup{)}, every maximal Hardy field is isomorphic as an ordered differential field to the ordered field $\No(\omega_1)$ of surreal numbers of countable length equipped with the derivation $\der_{\operatorname{BM}}$ of \cite{BM}. 
\end{cor}

\noindent
Thus with $\operatorname{CH}$, all maximal Hardy fields are isomorphic as ordered differential fields. Without $\operatorname{CH}$, the proof yields a nonempty back-and-forth system between any  
maximal Hardy field and the ordered differential field $\No(\omega_1)$. 
 (See [ADH, B.5] for  ``back-and-forth system''.) Then by Karp~\cite{Karp}, cf.~\cite[Theorem~3]{Barwise}, any maximal Hardy field and the ordered differential field $\No(\omega_1)$ are
 $\infty\omega$-equivalent. This is a strengthening of  \cite[Corollary~1]{ADH2}. 
 
 Key ingredients for proving Theorem~\ref{mt} include Lemma~\ref{unad}, the construction of a partition of unity in Section~\ref{psh2}, the reduction to Case \textup{(b)}  stated in Lemma~\ref{caseb}, and the elaborated reverse engineering in Section~\ref{fbg} that culminates in a diagonal argument. (The idea behind the original reverse engineering from \cite{S} is sketched in the remarks that follow the statement of Theorem~\ref{ha8}.)

Theorem~\ref{mt} answers a question of Ehrlich~\cite{Ehrlich} and establishes Conjecture~B from~\cite{ADH-ICM}.
(For Conjecture~A,  see \cite[Theorem~A]{ADH2}.) 

 In this paper our Hardy fields are not assumed to be $\Cc^{\infty}$-Hardy fields, and we do not know if maximal $\Cc^{\infty}$-Hardy fields are necessarily maximal Hardy fields (even under $\operatorname{CH}$). So the question arises if our main results go through for maximal $\Cc^{\infty}$-Hardy fields. This is indeed the case, and it is not hard to refine some of our proofs to that effect. The same question arises for the still more special $\Cc^{\omega}$-Hardy fields (analytic Hardy fields). Our main results also go through in that setting, but this is more delicate. We shall treat these refinements in a follow-up paper. 

We thank the referee for helpful comments.

\subsection*{Notations and conventions} We let $i$, $j$, $k$, $l$, $m$, $n$ range over $\N=\{0,1,2,\dots\}$. As in [ADH] the convention is that the ordering of an
ordered set, ordered  abelian group, or ordered field
is a {\it total}\/ ordering. Let $S$ be an ordered set. For any element~$b$ in an ordered set extending $S$ we set
$$S^{<b}\ :=\ \{s\in S:\, s<b\}, \qquad S^{>b}\ :=\ \{s\in S:\, s>b\}.$$
We have the usual notion of a set $P\subseteq S$ being cofinal in $S$ (respectively, coinitial in~$S$). In addition, sets $P, Q\subseteq S$ are said to be {\em cofinal\/} if for every $p\in P$ there exists~$q\in Q$ with $p\le q$ and for every $q\in Q$ there exists $p\in P$ with $q\le p$; replacing here $\le$ by $\ge$, we obtain the notion of $P$ and $Q$ being {\it coinitial.}\/ 
Thus~$P$ and $P^\downarrow=\{s\in S:\, \text{$s\leq p$ for some $p\in P$}\}$ are cofinal, hence
$P$, $Q$ are cofinal iff~$P^\downarrow=Q^\downarrow$.
Likewise,
$P$ and $P^\uparrow=\{s\in S:\, \text{$s\geq p$ for some $p\in P$}\}$ are coinitial, and 
$P$, $Q$ are coinitial iff $P^\uparrow=Q^\uparrow$.  We let~$\cf(S)$
and~$\ci(S)$  denote the cofinality and coinitiality of $S$; see [ADH, 2.1].
We say that $S$ is $\eta_1$ if for all countable~$P,Q\subseteq S$ with
$P < Q$ there exists an $s\in S$ with $P<s<Q$; in particular, such
$S$ is uncountable (cf.~Lemma~\ref{lmif} below), has no least element, no largest element, and
is dense in the sense that for all $p,q\in S$ with $p< q$ there exists $s\in S$ with $p<s<q$. We say that an ordered abelian group
(ordered field) is $\eta_1$ if its underlying ordered set is 
$\eta_1$.    
For basic facts about various $\eta_1$-structures, see \cite[Ka\-pi\-tel~IV]{PC}.

Let $(a_{\rho})$ be a well-indexed sequence. Its {\em length\/} is the (infinite limit) ordinal that
is the order type of its well-ordered set of indices $\rho$ (cf.~[ADH, p.~73]). Note that if $(a_{\rho})$ has countable length, then its length has cofinality $\omega$, and thus $(a_{\rho})$ has a cofinal subsequence~$(a_{\rho_n})$
of length $\omega$. 

Let $(\Gamma,\psi)$ be an asymptotic couple. As in [ADH, 6.5] we set $\Gamma_{\infty}:= \Gamma\cup \{\infty\}$, and
adopt  the convention that $\psi(0)=\psi(\infty)=\infty> \Gamma$. 
For $\alpha\in \Gamma_{\infty}$ we use~$\alpha^\dagger$ as an alternative notation for $\psi(\alpha)$ and define~$\alpha^{\<n\>}\in \Gamma_{\infty}$ by recursion on $n$ by~$\alpha^{\<0\>}:=\alpha$ and~$\alpha^{\<n+1\>}:=(\alpha^{\<n\>})^\dagger$.
We simplify terminology by calling an $H$-field {\em closed\/} (``$H$-closed'' in \cite{ADH-ICM,ADH2}) if it is $\upo$-free, newtonian, and Liouville closed. 

As in \cite{ADH5,ADH2}, $\Cc$ is the ring of germs at $+\infty$ of continuous functions~${[a,+\infty)\to \R}$, ${a\in\R}$, and
$\Cc^\times:=\{f\in \C:\, \text{$fg=1$  for some $g\in \Cc$}\}$, its multiplicative group of units. We often use the same notation for a real-valued function 
on a subset of~$\R$ containing a halfline $[a, +\infty)$, $a\in \R$,  as for its germ (at $+\infty$) if the resulting ambiguity is harmless.
With this convention, given a property~$P$ of real numbers
and $g\in \Cc$ we say that {\em $P\big(g(t)\big)$ holds eventually\/} if~$P\big(g(t)\big)$ holds for all sufficiently large real $t$.
We equip $\Cc$ with the partial ordering given by~$f\leq g:\Leftrightarrow f(t)\leq g(t)$, eventually. We define the asymptotic relations $\preceq$,~$\prec$,~$\sim$ on $\Cc$ as follows: for $f,g\in \Cc$,
\begin{align*} f\preceq g\quad &:\Longleftrightarrow\quad \text{there exists $c\in \R^{>}$ such that $|f|\le c|g|$,}\\
f\prec g\quad &:\Longleftrightarrow\quad \text{$g\in \Cc^\times$ and $\lim_{t\to \infty} f(t)/g(t)=0$} \\
 &\phantom{:} \Longleftrightarrow\quad \text{$g\in \Cc^\times$ and $\abs{f}\leq c\abs{g}$ for all $c\in\R^>$},\\
f\sim g\quad &:\Longleftrightarrow\quad \text{$g\in \Cc^\times$ and
$\lim_{t\to \infty} f(t)/g(t)=1$}\\ 
\quad&\phantom{:} \Longleftrightarrow\quad f-g\prec g.
\end{align*}
For $r\in \N\cup\{\infty\}$ we let
$\Cc^r$ be the subring of $\Cc$ consisting of the germs of $r$ times continuously differentiable functions $[a,+\infty)\to \R$, $a\in \R$. Thus $\Cc^{<\infty} :=  \bigcap_{n}\Cc^n$ is a differential ring with the obvious derivation, 
and has $\Cc^{\infty}$ as a differential subring.

\endgroup

\section{Preliminaries on Hausdorff Fields}\label{prel} 

\noindent
This section contains basic facts about Hausdorff fields.
 After a subsection on pc-sequences of length $\omega$ in an ordered field we 
 construct pseudolimits of such pc-sequences in the setting of Hausdorff fields, and show how to extend the value group of a Hausdorff field.

\subsection*{Ordered fields} Let $K$ be an ordered field.
We view $\Q$ as a subfield of $K$ in the natural way,
and consider $K$ also as a valued field with respect to the
standard valuation given by the valuation ring
$$\mathcal{O}\ =\ \big\{a\in K:\,\text{$|a|\le n$ for some $n$}\big\},$$
the smallest convex subring of $K$; see [ADH, p.~175].

\begin{lemma}[{Alling~\cite{Alling1,Alling2}}]\label{eta} The following two conditions on $K$ are equivalent:
{\samepage\begin{enumerate}
\item[\textup{(i)}] $K$ is $\eta_1$;
\item[\textup{(ii)}] the residue field of $K$ is isomorphic to $\R$, every pc-sequence of length $\omega$
in~$K$ has a pseudolimit in~$K$, and the value group of $K$ is 
$\eta_1$.
\end{enumerate}}
\end{lemma}

\noindent
This is well-known, see 
\cite[1.4]{Moresco} or   \cite[p.~160]{PC}. 
For a maximal Hardy field $H$ we have~${\R\subseteq H}$, and so the residue field
of $H$ is indeed isomorphic to $\R$. Thus in order to show that $H$ is $\eta_1$ it remains to show that all pc-sequences in $H$ of length~$\omega$
have a pseudolimit in $H$ and that the value group of $H$ is $\eta_1$. The former will be taken care of
in Sections~\ref{psh1},~\ref{psh2}, and the latter will be handled in Sections~\ref{bh}--\ref{fbg}.

\medskip\noindent
We continue with generalities on
pc-sequences of length $\omega$ in our ordered field $K$. 

Let $(a_n)$ be a pc-sequence in $K$ of length $\omega$. When does $(a_n)$ have a pseudolimit in $K$?
We
indicate below a reduction of this question to something that
turns out to be more manageable. First, $(a_n)$ and any
infinite subsequence have the same pseudolimits in $K$, and so by passing to such a
subsequence we can arrange that~$(a_n)$ is either strictly increasing or strictly decreasing. Replacing $(a_n)$ by~$(-a_n)$, the strictly decreasing case reduces to the strictly increasing case. Replacing $(a_n)$ by~$(a+a_n)$ for a suitable
$a\in K$, the strictly increasing case reduces to the strictly increasing case where in addition all terms are positive.
Next, assume~$(a_n)$ is strictly increasing and all terms are
positive. Dropping some initial terms, if necessary, we arrange in
addition that $a_{n}-a_{n-1}\succ a_{n+1}-a_n$ for all $n\ge 1$.
Then we define $b_n$ by $b_0:=a_0$ and
$b_n:=a_n-a_{n-1}$ for $n\ge 1$, so that $b_n>0$, $b_n\succ b_{n+1}$,
and $a_n=b_0+\cdots + b_n$, for all $n$. 
 
 Reversing this last step, starting with a sequence $(b_n)$ in $K$ such that $b_n>0$ and~${b_n\succ b_{n+1}}$ for all $n$, we obtain
 a strictly increasing pc-sequence $(a_n)$ of positive
 terms $a_n$ by $a_n=b_0+\cdots + b_n$. This leads to:
 
\begin{lemma}\label{pcsums} The following are equivalent for $K$:
\begin{enumerate}
\item[\textup{(i)}] all pc-sequences in $K$ of length $\omega$ have a pseudolimit in $K$;
\item[\textup{(ii)}] for every sequence $(b_n)$ in $K$ with $b_n>0$ and $b_n\succ b_{n+1}$ for all $n$, the pc-sequence $(a_n)$ with $a_n=b_0+\cdots + b_n$ for all $n$ has a pseudolimit in $K$.
\end{enumerate}
\end{lemma}

\subsection*{Hausdorff fields} As in \cite{ADH5} we define a {\em Hausdorff field\/} to be a
subfield of $\Cc$, that is, a subring of $\Cc$ that happens to be a field. Let $H$ be a Hausdorff field. Then $$\big\{f\in H:\, \text{$f(t)>0$,   eventually}\big\}$$ is the strictly positive cone for a (total) ordering on $K$ that makes $H$ an ordered field, and below we consider $H$ as an ordered field in this way. This yields the convex subring
$$\mathcal{O}\ :=\ \big\{f\in H:\,\text{$|f|\le n$ for some $n$}\big\},$$
which is a valuation ring of $H$, and we consider $H$ accordingly as a valued field as well. Restricting the relations $\preceq$, $\prec$, $\sim$ on $\Cc$ to $H$ gives exactly the asymptotic relations $\preceq$, $\prec$, $\sim$ on $H$ that it comes equipped with as a valued field.

\subsection*{Extending Hausdorff fields with pseudolimits}
Let $H$ be a Hausdorff field, and let a sequence  
$$f_0 \succ f_1 \succ f_2 \succ \cdots$$ in $H^{>}$
be given. Then $(f_0 + \cdots +f_n)$ is a pc-sequence
in $H$. We shall construct a pseudolimit of this pc-sequence in some 
Hausdorff field extension of $H$ (possibly $H$ itself).
To conform with some later parts we let $t$ range over real numbers $\ge 1$ in this subsection. We take for each $n$ a continuous function $\R^{\ge 1}\to \R$ that represents the germ $f_n$, to be denoted also by $f_n$, such
that $f_n(t)\ge 0$ and $f_{n+1}(t)\le f_n(t)/2$ for all
$t$. Now the sequence $(f_0+\cdots +f_n)$ of partial sums converges pointwise to a function $f=\sum_{n=0}^\infty f_n\colon \R^{\ge 1} \to \R$,
with the convergence being uniform on each compact subset of $\R^{\ge 1}$, so $f$ is continuous. We claim that for all $n$,
$$f-(f_0 + \cdots + f_n)\ \prec\ f_n\ \text{ in $\mathcal{C}$.}$$
Let $\epsilon >0$, and take $t_n\in \R^{\ge 1}$ with
$f_{n+1}(t)\le \epsilon f_n(t)$ for all $t\ge t_n$. Then for such $t$,
\begin{align*} f(t)-\big(f_0(t) + \cdots + f_n(t)\big)\ &=\ f_{n+1}(t) + f_{n+2}(t) + f_{n+3}(t) + \cdots \\
&\le\ f_{n+1}(t) + f_{n+1}(t)/2 + f_{n+1}(t)/4 + \cdots\\
&=\ 2f_{n+1}(t)\  \le\ 2\epsilon f_n(t),
\end{align*} 
which proves the claim. As usual we denote the germ of $f$ at $+\infty$ also by $f$, so that~$f\in \mathcal{C}$. Let $g,h\in \Cc$. Then (as defined earlier) $g\le h$ means $g(t)\le h(t)$, eventually,  and by  $g < h$  we mean $g\le h$ and $g\ne h$. Also
$$g<_{\ex} h\quad   :\Longleftrightarrow\quad g(t) < h(t), \text{ eventually,}$$
so $g<_{\ex} h\Rightarrow g< h$, and if $g,h\in H$, then $g<_{\ex} h\Leftrightarrow g< h$.   

\begin{lemma} Suppose $(f_0 + \cdots + f_n)$ has no pseudolimit in $H$. Let $g\in H$ be such that
$g> f_0 + \cdots + f_n$ in $\C$, for all $n$. Then for all $n$ we have
$$f_0 + \cdots + f_n <_{\ex} f <_{\ex} g \quad\text{ in $\mathcal{C}$}.$$ 
\end{lemma}
\begin{proof} As $g$ is not a pseudolimit of $(f_0+ \cdots + f_n)$, we have $v\big(g-(f_0+\cdots + f_n)\big)< v(f_{n+1})$ for some $n$.
For such $n$ we have, eventually,
$g(t)-\big(f_0(t) + \cdots + f_n(t)\big) > 2f_{n+1}(t)$, and thus, eventually,
$g(t) > f_0(t) + \cdots + f_n(t)+ 2f_{n+1}(t) \ge f(t)$.  
\end{proof}

\noindent
In view of \cite[Lemma~2.11]{ADH5} this yields:

\begin{cor} If $H$ is real closed and $(f_0 + \cdots + f_n)$ has no pseudolimit in~$H$, then~$f$ generates over $H$ an immediate Hausdorff field extension $H(f)$ of $H$ such that~$f_0 + \cdots + f_n \leadsto f$.  
\end{cor}

\noindent
Even if $H$ is not real closed, 
$(f_0+ \cdots + f_n)$ pseudoconverges in some Hausdorff field extension of $H$, since we can pass to the real closure of $H$ by \cite[Proposition~2.4]{ADH5}. 

\subsection*{Extending the value group of a Hausdorff field} 
This is closely connected to filling   {\it additive}\/ gaps in Hausdorff fields: see Remark~\ref{rem:add gap} and Lemma~\ref{lem:5.1.18} below. 
For now, $H$ is just an ordered field and $v\colon H^\times\to\Gamma$ is its standard valuation.

\begin{lemma}\label{lem:add gap 1}
Let $A\subseteq H$. Then $A+A$, $2A$ are cofinal. Also,
 $A$, $2A$ are cofinal iff~$A$,~$\frac{1}{2}A$ are cofinal. Likewise with ``coinitial'' in place of ``cofinal''.
\end{lemma}
\begin{proof}
From $2A\subseteq A+A$ and $a+b\leq 2\max(a,b)$ for $a,b\in A$ it follows that $A+A$ and $2A$ are cofinal.
The rest is clear.
\end{proof}

\begin{cor}\label{cor:add gap}
Let $A,B\subseteq H^{>}$ be such that $A<B$ and there is no $h\in H$ with~$A<h<B$. Then the following are equivalent:
\begin{enumerate}
\item[(i)] $A$, $A+A$  are cofinal;
\item[(ii)] $A$, $2A$  are cofinal;
\item[(iii)] $B$, $B+B$ are coinitial;
\item[(iv)] $B$, $\frac{1}{2}B$  are coinitial.
\end{enumerate}
\end{cor}
\begin{proof}
The equivalence of (i) and (ii) follows from Lemma~\ref{lem:add gap 1}; likewise with (iii) and (iv),  
The equivalence of (ii) and (iv) is a consequence of $B^\uparrow=H^{>}\setminus A^\downarrow$.
\end{proof}

\noindent
An {\bf additive gap} in $H$ is a pair $A$, $B$ of  subsets of $H^{>}$ with $A<B$ such that there is no $h\in H$ with $A<h<B$, and
one of the equivalent  conditions (i)--(iv) in Corollary~\ref{cor:add gap} holds.

\begin{remarkNumbered}\label{rem:add gap}
As in [ADH], a {\it cut}\/ in an ordered set $S$ is a downward closed subset of $S$.
Call a   cut $A$ in the ordered set $H^{>}$  {\it additive}\/ if $A$, $B:=H^{>}\setminus A$ is an additive gap in~$H$.
Then $A\mapsto A\cup(-A)\cup\{0\}$ defines an inclusion-preserving bijection 
$$\big\{\text{additive cuts in $H^{>}$}\big\}\to \{ \text{convex subgroups of $H$}\},$$ 
with inverse~$D\mapsto D^{>}$. (In some places additive cuts in $H^{>}$ are therefore called   ``group cuts'' in $H$; cf.~\cite{Kuhlmann-cuts}.)
Note: $D\mapsto v(D^>)$   is an inclusion-preserving bijection 
$$\{ \text{convex subgroups of $H$}\} \to \{ \text{upward closed subsets of $\Gamma$} \},$$
with inverse~$P\mapsto v^{-1}(P)\cup\{0\}$.
\end{remarkNumbered}

\noindent
{\it In~\ref{lem:mult gap}--\ref{lem:mult->add} below we assume that $H$ is real closed.}\/
We have multiplicative versions of Lemma~\ref{lem:add gap 1} and Corollary~\ref{cor:add gap},  obtained in the same way:

\begin{lemma}\label{lem:mult gap}
Let $A\subseteq H^>$. Then $A\cdot A$ and $\operatorname{sq}(A):=\{a^2:\,a\in A\}$ are cofinal.
Moreover,
 $A$ and $\operatorname{sq}(A)$ are cofinal iff~$A$ and~$\sqrt{A}:={\{b\in H^{>}:\, b^2\in A\}}$ are cofinal. Likewise with ``coinitial'' in place of ``cofinal''.
\end{lemma}

\begin{cor}\label{cor:mult gap}
Let $A,B\subseteq H^{>}$ be such that $A<B$ and there is no $h\in H$ with~$A<h<B$. Then the following are equivalent:  
\begin{enumerate}
\item[(i)] $A$, $A\cdot A$  are cofinal;
\item[(ii)] $A$, $\operatorname{sq}(A)$  are cofinal;
\item[(iii)] $B$, $B\cdot B$ are coinitial;
\item[(iv)] $B$, $\sqrt{B}$  are coinitial.
\end{enumerate}
\end{cor}

\begin{lemma}\label{lem:mult->add}
Let $A\subseteq H^{>\Q}$. If $A$, $\operatorname{sq}(A)$ are cofinal, then so are $A$, $2A$, and if~$A$,~$\sqrt{A}$ are coinitial, then so are $A$, $\frac{1}{2}A$.
\end{lemma}
\begin{proof}
Let $a\in A$. For the first part, use $2a < a^2$; for the second, use $\sqrt{a} < \frac{a}{2}$.
\end{proof}

\noindent
Now suppose $H$ is a Hausdorff field, turned into an ordered field as described earlier in this section.
The following is \cite[Lemma 2.12]{ADH5}:

\begin{lemma}\label{lem:5.1.18}
Suppose $\Gamma=v(H^\times)$ is divisible.
Let $P$ be a nonempty upward closed subset of $\Gamma$, and let $f\in \mathcal{C}$ be such that 
$a < f$ for all~$a\in H^{>}$ with 
$va\in P$, and~$f<b$ for all $b\in H^{>}$ with $vb < P$.
Then $f$ generates a Hausdorff field~$H(f)$ with 
$P >vf >  \Gamma\setminus P$.
\end{lemma}

 

\noindent
A {\em Hardy field\/}  is a differential subfield of the differential ring $\Cc^{<\infty}$. Given a Hardy field $F\supseteq \R$ we let $\Li(F)$ be the {\it Hardy-Liouville closure}\/ of $F$, that is, the smallest real closed Hardy field extension of $F$ that contains with any~$f$ also $\exp(f)$, and contains any $g\in \Cc^1$ whenever it contains $g'$; see \cite[Section~2]{ADH2}.

We now   specialize $H$ even further: {\it
in the rest of this subsection we assume that~$H$ is a Liouville closed Hardy field and $H\supseteq\R$.}\/

\begin{lemma}\label{in2} Let $A\subseteq H^{>\R}$. Then: \begin{enumerate}
\item[(i)] if $A$ and $\exp(A)$ are cofinal, then so are $A$ and 
$\operatorname{sq}(A)$;
\end{enumerate}
Next, assume also that $\ex^x\in A$, and that $A$ and $\operatorname{sq}(A)$ are cofinal. Then:\begin{enumerate}
\item[(ii)] $A$ and $A':=\{a':\,a\in A\}$ are cofinal;
\item[(iii)]  $A$ and $\int\! A:=\{b\in H:\,b'\in A\}$ are cofinal, and $\int\! A\subseteq H^{>\R}$.
\end{enumerate}
\end{lemma}
\begin{proof} Item  (i) follows from $a^2\le \exp a$ for $a\in A$. 

Next, assume $\ex^x\in A$, and   $A$, $\operatorname{sq}(A)$ are cofinal.
Then $\ex^{nx}\in A^\downarrow$ if $n\geq 1$. 
Now for~(ii),
let $a\in A$. Then $1/a\prec 1$, so $-a'/a^2=(1/a)'\prec 1$,  
and thus $0 < a' < a^2$. This yields $(A')^\downarrow\subseteq \operatorname{sq}(A)^\downarrow=A^\downarrow$.
Suppose in addition~$a\ge\ex^x$, so~$a^\dagger \succeq 1$.  
If~$a^\dagger\succ 1$, then $a<a'$, and
if $a^\dagger\asymp 1$,
then~[ADH, 9.1.11] yields $n\geq 1$ with $a \leq \ex^{nx}$, and
taking~$b\in A$ with $b\geq \ex^{(n+1)x}\succ\ex^{nx}$ we get~$b' > (\ex^{nx})'\geq \ex^{nx}\geq a$. Thus~$A^\downarrow\subseteq (A')^\downarrow$. 

As to~(iii), let $a\in A$, $b\in H$, and $b'=a$. Then $b> \R$, even~$b\succ x$. 
Moreover, $0<b'=a<b^2$, so $\sqrt{a}<b$.
Thus $A^\downarrow=(\sqrt{A})^\downarrow\subseteq (\int\!A)^\downarrow$.  Next, assume
also~$a=b'\succ \ex^x$. Then~$b\succ \ex^x$, since $H$ is asymptotic, so
$a/b=b^\dagger \succeq 1$, hence  $b\preceq a\prec a^2$ and thus $b<a^2$.  This yields $(\int\!A)^\downarrow\subseteq \operatorname{sq}(A)^\downarrow=A^\downarrow$. 
\end{proof}

\begin{lemma}\label{in3} Let $B\subseteq H$, $B>\ex^x$, and assume $B$, $\sqrt{B}$ are coinitial. Then: 
\begin{enumerate}
\item[(i)] $B$ and $B':=\{b': b\in B\}$ are coinitial;
\item[(ii)] $B$ and $\int B:=\{a\in H:\,a'\in B\}$ are coinitial;
\item[(iii)] $B^{-1}$ and 
${-\int B^{-1}}:=\big\{{-g}:\,g\in H^{\prec 1},\, g'\in B^{-1}\big\}$ are cofinal.
\end{enumerate}
\end{lemma}
\begin{proof} 
Since $B$, $\sqrt{B}$ are coinitial, so are~$B$,~$\frac{1}{2}B$, by Lemma~\ref{lem:mult->add}. Thus ~$B$,~$\R^>B$ are coinitial. Let $b\in B$. Then $b\succ \sqrt b > \ex^x$, so
$\beta:= vb <0$ gives~$\beta^\dagger \le 0$, hence~$\beta'\le \beta$, and thus $b'\ge b$. Also $\beta^\dagger=o(\beta)$ by [ADH, 9.2.10], so $\beta < \textstyle\frac{1}{2}\beta + \beta^\dagger = (\frac{1}{2}\beta)'$ 
and thus $ b\succ (\sqrt b)'\succeq d'$ for some $d\in B$. This proves (i).  

For (ii), let $a\in H$ and $a'=b\in B$. Then $a>\R$, and also $a\succ \ex^x$, since $a\preceq \ex^x$ gives~$b=a'\preceq \ex^x$, a
contradiction. Hence~$\alpha^\dagger\le 0$ for $\alpha:= va$, so
$\alpha\ge \alpha'=\beta:=vb$,  which gives 
$a\preceq b$, and thus $a\le b^2$. Since~$b\in B$ was arbitrary and
$B$ and $\operatorname{sq}(B)$ are coinitial by Lemma~\ref{lem:mult gap}, this shows that every element of $B$ is $\ge a$ for some~$a\in \int B$. With $a$ and~$b$ as above we also have~$\alpha < \beta/2$, so $a > \sqrt{b}$. This
proves (ii). 

As to (iii), let $b\in B$, $\beta:= vb$, and $g\in H^{\prec 1}$ with~$g'=b^{-1}$, so $g<0$, and for~$\gamma:=vg$ we have $-\beta=\gamma+ \gamma^{\dagger}$.
We have $\sqrt{b} > \ex^x$, so $b^{-1}< \ex^{-2x}$. Claim: $\gamma^\dagger \le 0$. If this claim does not hold, then $0 <\gamma^\dagger <v(x^{-2})$, so~$g\succ\ex^{-x}$, and~$g^\dagger \succ x^{-2}\succ \ex^{-x}$,
and  thus $b^{-1}=g'\succ \ex^{-2x}$, a contradiction.  Now~$\gamma^{\dagger}\le 0$ gives $-\beta\le \gamma$, hence~$b^{-1}=|b^{-1}|\succeq |g|=-g$, and thus $-g\le d$ for some $d\in B^{-1}$. From 
$\gamma^\dagger=(-\gamma)^\dagger\le 0$ we get~$\gamma^\dagger=o(\gamma)$ by    [ADH, 9.2.10], hence~$-2\beta > \gamma$, and thus~$b^{-2} < -g$. It remains to use that $B^{-1}$, $\operatorname{sq}(B^{-1})$ are cofinal.  
\end{proof}

\section{Analytic Preliminaries}\label{sec:analytic}

\noindent
In this  section  $a$, $b$, $c$, $s$, $t$ range over $\R$. 

\subsection*{Constructing smooth functions}  We prove here some facts about smooth functions needed later. 
Let $\rho\colon \R\to \R$ be the $\mathcal{C}^\infty$-function of \cite[(8.12), Exercise~2(a)]{D}. It is defined by
$$\rho(t)\ :=\ \exp\!\left(-\frac{1}{(1+t)^2}-\frac{1}{(1-t)^2}\right)\ \text{if $-1< t < 1$,}  \quad \rho(t):= 0\  \text{if $t\le -1$ or $t\ge 1$.}$$
Thus $\rho(t) >0$ for $-1<t<1$, $\rho$ is even, and $\rho(0)=\ex^{-2}$.  (See Figure~\ref{fig:rho}.)

\begin{figure}[ht]
 \begin{tikzpicture}
  \begin{axis} [axis equal image, axis lines=center, xmin=-1.25, xmax=1.25, ymin = -0.1, ymax =0.4, width=0.75\textwidth,  xlabel={$t$}, xtick={-1,0,1}, ytick=\empty]
    \addplot [domain=-0.9:0.9, smooth, very thick] { exp(-(1+x)^(-2)-(1-x)^(-2)) };
    \addplot [domain=-1.1:-0.9, smooth, very thick] {0}; 
    \addplot [domain=0.9:1.1, smooth, very thick] {0};  
    \node[right] at (axis cs:-0.9,0.1) {$\rho$};
  \end{axis}
\end{tikzpicture}
  \caption{Sketch of $\rho$}\label{fig:rho}
\end{figure}
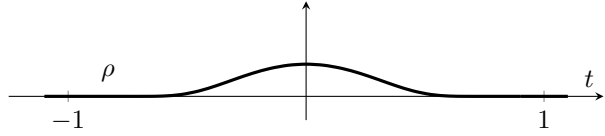

\noindent
For any subset $I$ of $\R$ and $r\in \N\cup\{\infty,\omega\}$ we define
$\Cc^r(I)$ to be the set of $f\colon I \to \R$ for which $f=g|_{I}$ for some $\Cc^r$-function $g\colon U\to \R$ with $U$ an open neighborhood of $I$ in $\R$; instead of ``$f\in \Cc^r(I)$'' we also write ``$f\colon I \to \R$ is a $\Cc^r$-function'' or ``$f\colon I \to \R$ is of class $\Cc^r$''. 
We use this mainly for sets $I=[a,b]$ with $a<b$ and sets~$I=[a,\infty)$. 
As in \cite{ADH5,ADH2} we  denote $\Cc^r[a,+\infty)$ by $\Cc_a^r$, and $\Cc_a:=\Cc_a^0$. 

\begin{lemma}\label{bump} There is a $\Cc^{\infty}$-function $\alpha\colon \R \to \R$ such that $\alpha=0$ on $(-\infty,0]$, $\alpha$ is strictly increasing on $[0,1]$, and $\alpha=1$ on $[1,+\infty)$.
\end{lemma}
\begin{proof} One can take $\alpha(t):= c^{-1}\int_{-\infty}^t \rho(2s-1)ds$ where $c:=\int_{-\infty}^\infty \rho(2s-1)ds$.
\end{proof}

\begin{lemma}\label{thetazeta} Let $\theta\colon [a,\infty) \to \R^{>}$ be continuous. Then there exists a decreasing $\Cc^{\infty}$-function $\zeta\colon [a,\infty) \to \R^{>}$ such that $\theta(t) > \zeta(t)$ and $\zeta'(t) > -1$ for all $t\ge a$.
\end{lemma} 
\begin{proof} Replacing $\theta$ by the function $t\mapsto \min_{a\le s\le t}\min\!\big(\theta(s),1\big) \colon [a,\infty)\to \R^{>}$  we arrange that $\theta$ is decreasing and $0\le \theta\le 1$ on $[a,\infty)$.  Next we follow Exercise~2 of~\cite[(8.12)]{D}, taking the convolution with $\rho$; in other words, we extend $\theta$ to all of~$\R$ by setting $\theta(t)=0$ for $t<a$, and then define $f\colon \R \to \R$ by $$f(t):=\int_{-\infty}^\infty \theta(s)\rho(t-s)\,ds=\int_{-\infty}^\infty\theta(t-s)\rho(s)\,ds.$$
Instead of $-\infty$, $\infty$ we can take in the left integral any real bounds  $c \le t-1$, $d\ge t+1$, and in the right integral  any real bounds
$c\le -1$, $d\ge 1$. As in that exercise one shows that $f$ is of class $\Cc^\infty$ (in fact, $f^{(p)}(t)=\int_{-\infty}^\infty\theta(s)\rho^{(p)}(t-s)ds$ for all $p\in \N$ and  all $t$) and decreasing on $[a+1, \infty)$.  For $t\ge a+1$ we have
$$0 < f(t)\ =\ \int_{-1}^1\theta(t-s)\rho(s)\,ds\ \le\ 2\ex^{-2}\theta(t-1)\  <\ \theta(t-1).$$
Using
$\rho'(s)\ge 0$ for $-1\le s\le 0$ and $\rho'(s)\le 0$ for $0\le s\le 1$, we obtain for all $t$,
$$f'(t)\ =\ \int_{-1}^1 \theta(t-s)\rho'(s)\,ds\ \ge\ \int_0^1\theta(t-s)\rho'(s)\,ds\ \ge\  \int_0^1\rho'(s)\,ds\ =\ -\ex^{-2}\ >\ -1.$$ 
Thus $\zeta\colon [a,\infty)\to \R^{>}$ defined by $\zeta(t):= f(t+1)$ has the desired property.
\end{proof}

\begin{lemma}\label{phizeta} Let $a< b$ and $\phi,\zeta\in\Cc^{\infty}[a,b]$
be such that $\phi(a)=\zeta(a)$ and $\phi < \zeta$ on~$(a,b]$,  
and let real numbers $c_n$ be given with
$\phi(b) < c_0 < \zeta(b)$. Then there exists a function $\theta\in \Cc^{\infty}[a,b]$ such that $\theta^{(n)}(a)=\phi^{(n)}(a)$ for all $n$, $\phi < \theta < \zeta$
on $(a, b]$, and $\theta^{(n)}(b)=c_n$ for all $n$. \textup{(See Figure~\ref{fig:phizeta}.)}
\end{lemma}

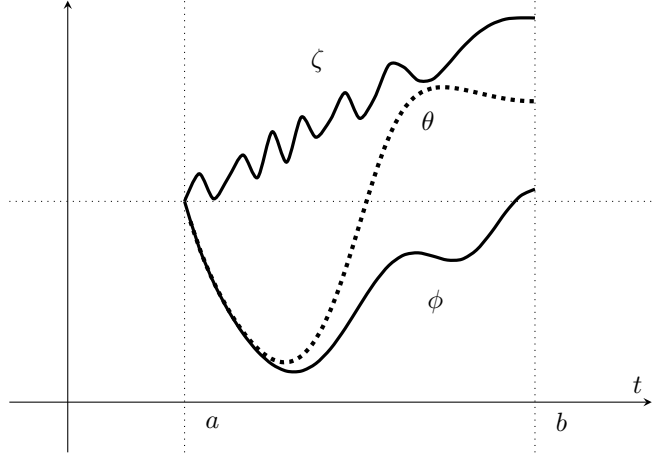
\begin{figure}[ht]
\begin{tikzpicture}
  \begin{axis} [axis lines=center, xmin=-0.5, xmax=5, ymin = -0.25, ymax = 2, width=0.8\textwidth, height = 0.6\textwidth, xlabel={$t$}, xtick=\empty, ytick=\empty]
  \draw[dotted] (axis cs: 1,-0.25) -- (axis cs: 1,2);
  \draw[dotted] (axis cs: 4,-0.25) -- (axis cs: 4,2);
  \node[right] at (axis cs: 1.1,-0.1) {$a$};
  \node[right] at (axis cs: 4.1,-0.1) {$b$};
  \node[right] at (axis cs:3,0.5) {$\phi$};
  \node[right] at (axis cs:2,1.7) {$\zeta$};
  \node[right] at (axis cs:2.95,1.4) {$\theta$};

  \def\deltaval{1.7};
  \def\epsilonval{0.0005};
  \def\alphafn{(1/(1+exp(-2*( (4/\deltaval)*(x-4)+3))))};
  \def\phifn{((x-1)/3 +  exp((x-1)/2)*(1/x^3*cos((x^2-1)*50)))};
  \def\zetafn{1 + (x-1)/3 +   sin((x^2-1)*400)/10};
  \def\betafn{(1.5+(1/4)*(x-4)^2)};
  \def\thetafnsecond{(\phifn)+(((\betafn)-(\phifn))*(\alphafn))};
  \def\transfn{(-1)+(2/(3-0.5*\deltaval))*(x-1)}; 
  \def\rhofn{ exp(-1/(1+(\transfn))^2 - 1/(1-(\transfn))^2) };
  \def\thetafnfirst{((\phifn)+((\betafn)-(\phifn))*(\alphafn))+((\epsilonval)*(\rhofn)*((\zetafn)-(\phifn)))};
  
  \draw[dotted] (axis cs: -0.5,1) -- (axis cs:5,1);
  \addplot [domain=1:4, smooth, very thick] { \zetafn};
  \addplot [domain=1:4, smooth, very thick] { \phifn };
    \addplot [domain=(4-(1/2)*\deltaval):4, smooth, ultra thick, dotted] { \thetafnsecond };
    \addplot [domain=1.05:(3.985-(1/2)*\deltaval), smooth, ultra thick, dotted] { \thetafnfirst };

   \end{axis}
\end{tikzpicture}
\label{fig:phizeta}\caption{Sketch of $\phi$, $\theta$, $\zeta$ in Lemma~\ref{phizeta}}
\end{figure}

\begin{proof} By subtracting $\phi$ throughout we arrange $\phi=0$. A result due to E.~Bo\-rel~\cite[Ex\-er\-cise~4(a), p.~192]{D} yields a
function $\beta\in \Cc^{\infty}[a,b]$ such that $\beta^{(n)}(b)=c_n$ for all~$n$. Take $\delta \in (0,b-a)$ with
$\delta < \beta < \zeta-\delta$ on $[b-\delta, b]$, and then~$\alpha\in \Cc^{\infty}[a,b]$ such that
\begin{itemize}
\item  $\alpha=0$ on $[a,b-\delta]$, 
\item
$\alpha$ is strictly increasing on $[b-\delta, b-\frac{1}{2}\delta]$, and
\item  $\alpha=1$ on $[b-\frac{1}{2}\delta,b]$.
\end{itemize}
Take $\epsilon\in (0,\frac{1}{2})$ such that $\epsilon\zeta< \delta$ on $[b-\delta,b]$, and take $\gamma\in \Cc^{\infty}[a,b]$ such that
\begin{itemize}
\item $\gamma^{(n)}(a)=0$ for all $n$,
\item $0 <\gamma< \epsilon$ on $(a,b-\frac{1}{2}\delta)$,
and
\item $\gamma=0$ on $[b-\frac{1}{2}\delta,b]$. 
\end{itemize}
Then the function
$\theta:= \alpha\beta+\gamma\zeta$ has the desired properties.  \end{proof}

\begin{lemma}\label{lem:pwa intermed} 
Let $a<b$, $f,g\in\Cc[a,b]$,  and $f<g$ on $[a,b]$.
Then there are~$a_0<a_1<\cdots<a_n$ with $a_0=a,\ a_n=b$, and a function $\phi\colon [a,b]\to\R$ such
that {\samepage
\begin{enumerate}
\item[(i)]
$f<\phi<g$ on $[a,b]$,
\item[(ii)] $\phi(a)=\frac{1}{2}\big(f(a)+g(a)\big)$ and $\phi(b)=\frac{1}{2}\big(f(b)+g(b)\big)$,
and
\item[(iii)]
for $i=0,\dots,n-1$, the restriction of $\phi$ to $[a_i,a_{i+1}]$ is the restriction of
an affine function $\R\to\R$.
\end{enumerate}}
\end{lemma}
\begin{proof}
Let $\varepsilon:=\frac{1}{2}\min\big\{g(t)-f(t):t\in [a,b]\big\}$, so $\varepsilon>0$. Choose $n\geq 1$ such that
for all $s,t\in [a,b]$ with $\abs{s-t}\leq \delta:=\frac{b-a}{n}$ we have $\abs{f(s)-f(t)},\abs{g(s)-g(t)}<\varepsilon$.
For~$i=0,\dots,n$ set $a_i:=a+i\delta$, and for $i=0,\dots,n-1$ take affine~$\phi_i\colon\R\to\R$ with~$\phi_i(a_i)=\frac{1}{2}\big(f(a_i)+g(a_i)\big)$ and $\phi_i(a_{i+1})=\frac{1}{2}\big(f(a_{i+1})+g(a_{i+1})\big)$.
It suffices  to show that $f<\phi_i<g$ on $[a_i, a_{i+1}]$ for $i=0,\dots, n-1$. For such $i$ and
 $s,t \in [a_i,a_{i+1}]$,
$$f(t)<\varepsilon+f(s)\leq\textstyle\frac{1}{2}\big(f(s)+g(s)\big)  \leq -\varepsilon+g(s)<g(t), 
$$
in particular, $f(t)<\phi_i(a_i),\phi_i(a_{i+1})<g(t)$. Since $\phi_i(a_i)\leq\phi_i(t)\leq\phi_i(a_{i+1})$ or~$\phi_i(a_{i+1})\leq\phi_i(t)\leq\phi_i(a_{i})$, we are done.
\end{proof}

\begin{lemma}\label{smooth} Let $f, g\in\Cc_a$  be such that $f < g$ on $[a,+\infty)$. Then there exists a function~$y\in\Cc^\infty_a$   such that $f < y < g$ on $[a,+\infty)$.  
\end{lemma}
\begin{proof} Lemma~\ref{lem:pwa intermed} yields a piecewise affine intermediary $\phi$, more precisely,   a
strictly increasing sequence $(a_n)$ in $\R$ with $a_0=a$ and $a_n\to +\infty$ as $n\to \infty $  and a $\phi\in \Cc_a$ such that for each $n$ the restriction of
$\phi$ to $[a_n, a_{n+1}]$ is the restriction of an affine function $\R\to \R$, and such that $f< \phi< g$ on $[a,+\infty)$. This reduces the problem of constructing
$y$ to  proving the next lemma.
\end{proof}

\begin{lemma} Let $a<b<c$ and $\phi,\theta\in \Cc^{\infty}[a,c]$ be such that $\phi(b)=\theta(b)$, and let~$0 < \epsilon< b-a, c-b$.
Then there exists $y\in \Cc^{\infty}[a,c]$ such that \begin{align*}
y(t)\ &=\ \phi(t)\ \text{ for $a\le t\le b-\epsilon$,} &  \abs{y(t)-\phi(t)} &\ <\  \epsilon\ \text{ for $b-\epsilon \le t\le b$,}\\
y(t)\ &=\ \theta(t)\ \text{ for $b+\epsilon \le t\le c$,} &  \abs{y(t)-\theta(t)} &\ <\ \epsilon\ \text{ for $b\le t \le b+\epsilon$.}
\end{align*}
\end{lemma}
\begin{proof} Take $0<\delta < \epsilon$ such that $|\phi-\theta|\le \epsilon/2$ on $[b-\delta, b+\delta]$. Next, take
$\beta\in \Cc^{\infty}[a,c]$ such that 
\begin{itemize}
\item $\beta=0$ on $[a,b-\delta]$, 
\item $0\le \beta\le 1$ on $[b-\delta, b+\delta]$, and
\item $\beta=1$ on $[b+\delta, c]$. 
\end{itemize}
Then $y:=(1-\beta)\phi+\beta\theta$
has the desired property. 
\end{proof}

\begin{lemma}\label{lem:smooth} 
For each $n$, let $f_n,g_n\in\Cc$ be such that $f_n\leq f_{n+1}$, $g_{n+1}\leq g_n$, and~$f_n <_{\ex} g_n$. Then there exists~$\phi\in \Cc^\infty$ such that $f_n <_{\ex} \phi <_{\ex} g_n$ for each $n$.
\end{lemma}
\begin{proof}
Take for each $n$ representatives
of $f_n$ and $g_n$ in $\mathcal{C}_0$, denoted also by $f_n$ and~$g_n$, such that $f_n < g_n$ on $[0,\infty)$. 
 Next, take a strictly increasing sequence
$(a_n)$ of real numbers  $\ge 0$ with $a_n\to \infty$ as $n\to \infty$, such that $f_n \le f_{n+1}$ and~$g_n\ge g_{n+1}$ on~$[a_n,\infty)$, and  take continuous functions~$\alpha_n, \beta_n\colon [0,\infty) \to [0,1]$ with $\alpha_n(a_n)=1$, $\alpha_n(a_{n+1})=0$ and $\alpha_n+\beta_n=1$ on $[a_n, a_{n+1}]$. Let~$f, g\colon [0,\infty)\to \R$ be  given by
\begin{align*} f&=f_0\text{ on }[0,a_0],\quad  f=\alpha_{n} f_n + \beta_{n} f_{n+1} \text{ on }[a_n, a_{n+1}],\\
g&=g_0\text{ on }[0,a_0],\quad  g=\alpha_{n} g_n + \beta_{n} g_{n+1} \text{ on }[a_n, a_{n+1}],
\end{align*} 
so $f$, $g$ are continuous, $f_n\le f$  and $g\le g_n$ on $[a_n,\infty)$, and $f< g$ on $[0,\infty)$. 
(See Figure~\ref{fig:lemsmooth}.)
Now Lemma~\ref{smooth}  gives $\phi\in\Cc_0^\infty$ such that $f < \phi < g$ on $[0,\infty)$, and then its germ at $+\infty$, denoted also by $\phi$,
satisfies~$f_n <_{\ex} \phi <_{\ex} g_n$ for all~$n$.
\end{proof}

\begin{figure}[ht]
\begin{tikzpicture}
  \begin{axis} [axis lines=center, xmin=-0.5, xmax=5, ymin = -0.25, ymax = 2,  height = 0.5\textwidth, xlabel={$t$}, xtick=\empty, ytick=\empty]
  \def\anval{1}; \def\anploneval{4}; 
  \def\alphanfn{1+(x-\anval)/(\anval-\anploneval)};
  \def\betanfn{1-(\alphanfn)};
  \draw[dotted] (axis cs: \anval,-0.25) -- (axis cs: \anval,2);
  \draw[dotted] (axis cs: \anploneval,-0.25) -- (axis cs: \anploneval,2);
  \node[right] at (axis cs: \anval,-0.1) {$a_n$};
  \node[right] at (axis cs: \anploneval,-0.1) {$a_{n+1}$};
  \addplot [domain=\anval:\anploneval, smooth, very thick] {\alphanfn};
  \addplot [domain=\anval:\anploneval, smooth, very thick] {\betanfn};
  \node[left] at (axis cs: \anploneval-0.6,0.9) {$\beta_n$};
    \node[right] at (axis cs: \anval+0.5,0.9) {$\alpha_n$};
  \end{axis}
\end{tikzpicture}\quad
\begin{tikzpicture}
  \begin{axis} [axis lines=center, xmin=-0.5, xmax=5, ymin = -0.25, ymax = 2,  height = 0.5\textwidth, xlabel={$t$}, xtick=\empty, ytick=\empty]
  \def\anval{1}; \def\anploneval{4}; 
  \def\alphanfn{1+(x-\anval)/(\anval-\anploneval)};
  \def\betanfn{1-(\alphanfn)};
  \def\fnfn{0.04*x*x};
  \def\fnplonefn{0.1+0.04*(x*x+x)};
  \def\gnfn{2*x*x};
  \def\gnfn{1.3+0.03*(x*x+(1.5)*x)};
  \def\gnplonefn{1.2+0.03*(x*x-x)};
  \def\ffn{((\alphanfn)*(\fnfn))+((\betanfn)*(\fnplonefn))};
  \def\gfn{((\alphanfn)*(\gnfn))+((\betanfn)*(\gnplonefn))};
  \def\phifn{((\ffn)+(\gfn))/2};
  
  \draw[dotted] (axis cs: \anval,-0.25) -- (axis cs: \anval,2);
  \draw[dotted] (axis cs: \anploneval,-0.25) -- (axis cs: \anploneval,2);
  \node[right] at (axis cs: \anval,-0.1) {$a_n$};
  \node[right] at (axis cs: \anploneval,-0.1) {$a_{n+1}$};
  
    \addplot [domain=\anval:\anploneval, smooth, very thick] {\fnfn};
    \addplot [domain=\anval:\anploneval, smooth, very thick] {\fnplonefn};
    \addplot [domain=\anval:\anploneval, smooth, very thick] {\gnplonefn};
    \addplot [domain=\anval:\anploneval, smooth, very thick] {\gnfn};
    \addplot [domain=\anval:\anploneval, smooth, very thick, dotted] {\ffn};
    \addplot [domain=\anval:\anploneval, smooth, very thick, dotted] {\gfn};
    \addplot [domain=\anval:\anploneval, smooth, very thick, dashed] {\phifn};
 \node[right] at (axis cs:4,1.25) {$\phi$};
 \node[right] at (axis cs:4,1.55) {$g_{n+1}$};
 \node[right] at (axis cs:4,1.95) {$g_n$};
 \node[right] at (axis cs:4,0.9) {$f_{n+1}$};
 \node[right] at (axis cs:4,0.6) {$f_n$};
  \end{axis}
\end{tikzpicture}
\caption{Constructing $\phi$ in the proof of Lemma~\ref{lem:smooth}}\label{fig:lemsmooth}
\end{figure}
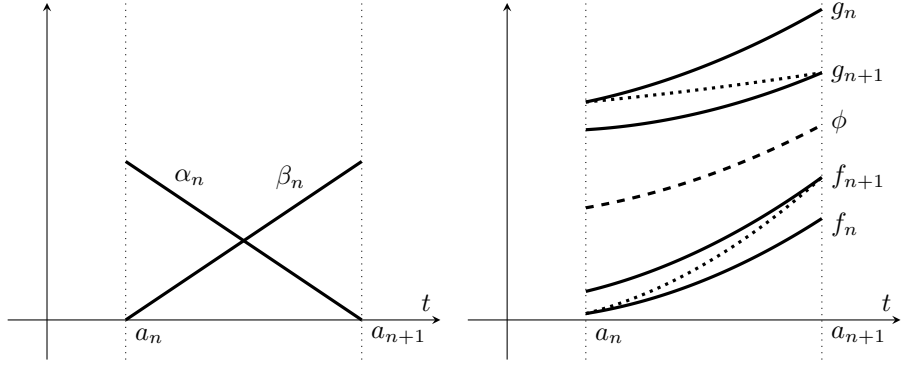

\begin{cor}\label{corsmooth} Let $H$ be a Hausdorff field and $A$, $B$ nonempty countable subsets of~$H$ with  $A < B$. Then there exists
$\phi\in \Cc^\infty$ such that $A<_{\ex} \phi <_{\ex} B$.
\end{cor}
\begin{proof}
Take an increasing and cofinal sequence $(f_n)$ in $A$ and a decreasing  coinitial sequence $(g_n)$ in $B$, 
and apply the previous lemma.
\end{proof}

\noindent
We shall also use the following variant of Lemma~\ref{phizeta}:

\begin{lemma} \label{phizetavar} Let $a< b$, $f, g\in \Cc[a,b]$, and $c_n, d_n\in \R$ for $n=0,1,2,\dots$ be such that $f(a) < c_0 < g(a)$, $f<g$ on  $[a,b]$, and $f(b) < d_0 < g(b)$. Then there exists~$y\in \Cc^{\infty}[a,b]$ with $f< y < g$ on $[a,b]$ and $y^{(n)}(a)=c_n$, $y^{(n)}(b)=d_n$ for all $n$. 
\end{lemma} 
\begin{proof}  Take $\epsilon>0$ such that $f(a)+\epsilon< c_0$, $f+\epsilon < g$ on
$[a,b]$, and~$f(b)+\epsilon< d_0$. Lemma~\ref{smooth} gives $\phi\in \Cc^{\infty}[a,b]$ with
$f < \phi < f+\epsilon$ on $[a,b]$, and so replacing $f$ by $\phi$ 
 and then subtracting $\phi$ throughout (replacing~$g$ by $g-\phi$ and $c_n$, $d_n$ by $c_n-\phi^{(n)}(a)$, $d_n-\phi^{(n)}(b)$, respectively) we arrange $f=0$. 
 
 Borel's result gives $\alpha, \beta\in \Cc^{\infty}[a,b]$ with
 $\alpha^{(n)}(a)=c_n$ and $\beta^{(n)}(b)=d_n$ for all~$n$.  Take a real number $M>0$ such that
 $|\alpha|, |\beta|\le M$ on $[a,b]$.  Take ``small'' real numbers $\eta_1, \eta_2>0$ such that
 $a+2\eta_1 < b-2\eta_1$, $2M\eta_1 < \eta_2$, and~$2M\eta_1+\eta_2 < g$ on~$[a,b]$. Take $\gamma,\delta\in \Cc^{\infty}[a,b]$ such that
\begin{itemize} 
\item $\gamma=1$ on~$[a, a+\eta_1]$, 
\item $\gamma$ is decreasing on $[a+\eta_1, a+2\eta_1]$,
\item $\gamma=\eta_1$ on $[a+2\eta_1, b-2\eta_1]$, 
\item $\gamma$ is decreasing on~$[b-2\eta_1, b-\eta_1]$, and
\item $\gamma=0$ on $[b-\eta_1,b]$,
\end{itemize} and $\delta$  behaves similarly in the opposite direction: 
\begin{itemize}
\item $\delta=0$ on $[a, a+\eta_1]$, 
\item $\delta$ is increasing on $[a+\eta_1, a+2\eta_1]$,
\item $\delta=\eta_1$ on~$[a+2\eta_1, b-2\eta_1]$, 
\item $\delta$ is increasing on $[b-2\eta_1, b-\eta_1]$, and 
\item $\delta=1$ on $[b-\eta_1,b]$. 
\end{itemize}
Finally, take $\theta\in \Cc^{\infty}[a,b]$ such that
\begin{itemize}
\item $\theta=0$ on $[a, a+\eta_1]$, 
\item $\theta$ is increasing on $[a+\eta_1, a+2\eta_1]$, 
\item $\theta=\eta_2$ on $[a+2\eta_1, b-2\eta_1]$, 
\item $\theta$ is decreasing on $[b-2\eta_1, b-\eta_1]$, and 
\item $\theta=0$ on $[b-\eta_1, b]$. 
\end{itemize}
(See Figure~\ref{fig:phizetavar}.) Then $y:=\gamma\alpha+ \delta\beta + \theta$ has the desired property, provided~$\eta_1$,~$\eta_2$ are sufficiently small.  
 \end{proof}

\begin{figure}[ht]
\begin{tikzpicture}
\def\an{0.5};\def\bn{4.75};\def\etaone{0.6};\def\etatwo{0.85}; 
  \begin{axis} [axis lines=center, xmin=-0.2, xmax=5, ymin = -0.01, ymax = 1.25, width=0.67\textwidth, height = 0.275\textwidth, xlabel={$t$}, xtick={\an,\an+\etaone,\an+2*\etaone,\bn-2*\etaone,\bn-\etaone,\bn}, ytick={\etaone, \etatwo, 1}, xticklabels={\strut $a$, \strut $a+\eta_1$, \strut $\hskip2.25em a+2\eta_1$, \strut $\hskip-1.25em b-2\eta_1$, \strut $\hskip0.5em b-\eta_1$, \strut $b$}, yticklabels={\strut $\eta_1$, \strut $\eta_2$, \strut $1$},
  legend style={font=\small},
 legend cell align=left, legend style={at={(1.05,0.7)},anchor=west}]

   \addplot [domain=\an:\an+\etaone, smooth, thick] {1};
   
    \addlegendentry{$\gamma$};

    \addplot [domain=\an:\an+\etaone, smooth,     thick, dotted] {0};

    \addlegendentry{$\delta$};

      \addplot [domain=\an:\an+\etaone, smooth,   thick, dashed] {0};

    \addlegendentry{$\theta$};

   \addplot [domain=\an+\etaone:\an+2*\etaone, smooth, thick] { ((1-\etaone)/2)*(-tanh( 10*(x-\an-1.5*\etaone) )) + 1-(1-\etaone)/2};
   \addplot [domain=\an+2*\etaone:\bn-2*\etaone, smooth, thick] {\etaone};
  \addplot [domain=\bn-2*\etaone:\bn-\etaone, smooth, thick] { (\etaone/2)+(\etaone/2)*(-tanh( 10*(x-\bn+1.5*\etaone) ) )};
   \addplot [domain=\bn-\etaone:\bn, smooth,  thick] {0};

   \addplot [domain=\an:\an+\etaone, smooth,      thick, dotted] {0};
   \addplot [domain=\an+\etaone:\an+2*\etaone, smooth,      thick, dotted] { (\etaone/2)*(tanh( 9*(x-\an-1.5*\etaone) ) + 1)};
   \addplot [domain=\an+2*\etaone:\bn-2*\etaone, smooth,      thick, dotted] {\etaone};
   \addplot [domain=\bn-2*\etaone:\bn-\etaone, smooth,      thick, dotted] { \etaone+((1-\etaone)/2)*(tanh( 9*(x-\bn+1.5*\etaone) ) + 1)};
   \addplot [domain=\bn-\etaone:\bn, smooth,      thick, dotted] {1};
    
   \addplot [domain=\an+2*\etaone:\bn-2*\etaone, smooth,   thick, dashed] {\etatwo};
   \addplot [domain=\bn-\etaone:\bn, smooth,    thick, dashed] {0};
   \addplot [domain=\an+\etaone:\an+2*\etaone, smooth,    thick, dashed] { (\etatwo/2)*(tanh( 9*(x-\an-1.5*\etaone) ) + 1)};
   \addplot [domain=\bn-2*\etaone:\bn-\etaone, smooth,   thick, dashed] { (\etatwo/2)+(\etatwo/2)*(-tanh( 10*(x-\bn+1.5*\etaone) ) )};
  \end{axis}
\end{tikzpicture}
\caption{The functions $\gamma$, $\delta$, $\theta$}\label{fig:phizetavar}
\end{figure}
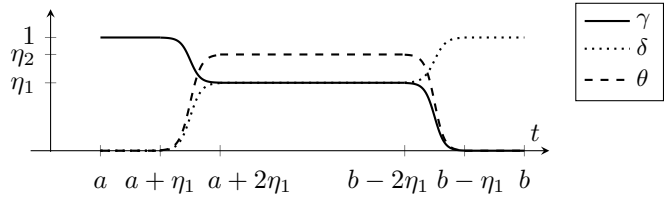

\subsection*{Constructing infinite sums}   
The next lemma follows from \cite[(8.6.4)]{D}:

\begin{lemma}  Let $(f_n)$ be a sequence of functions in $\mathcal{C}_a^1$ such that $f_n(a)\to c $ as~${n\to \infty}$,  for some $c\in \R$. Suppose also that $(f_n')$ converges to $g\in \mathcal{C}_a$, uniformly on $[a,b]$ for every $b>a$. Then $(f_n)$ converges to a function
$f\in \mathcal{C}_a^1$, uniformly on~$[a,b]$ for every $b>a$, with $f'=g$. 
\end{lemma} 

\noindent
We use this for infinite series where the $f_n$ are the partial sums,  and with higher derivatives where the assumptions allow us to apply the lemma
inductively. We shall also need a slight twist, where instead of the derivation $\der\colon \mathcal{C}_a^1\to \mathcal{C}_a$ we use~${\derdelta\colon \mathcal{C}_a^1\to \mathcal{C}_a}$, with $\derdelta:= \phi^{-1}\der$, $\phi\in (\mathcal{C}_a)^{\times}$:

\begin{lemma} Let $(f_n)$ be a sequence of functions in $\mathcal{C}_a^1$ such that $f_n(a)\to c $ as~${n\to \infty}$,  for some $c$. Suppose also that $(\derdelta f_n)$ converges to $g\in \mathcal{C}_a$, uniformly on~$[a,b]$ for every $b>a$. Then $(f_n)$ converges to a function
$f\in \mathcal{C}_a^1$, uniformly on~$[a,b]$ for every $b>a$, with $\derdelta f=g$. 
\end{lemma}

\noindent
This follows from the previous lemma in view of $\der=\phi\derdelta$. Induction on $m$ yields:

\begin{cor}\label{corsumdif} Let $m\ge 1$, $\phi\in (\mathcal{C}_a^{m-1})^\times$, and $\derdelta :=  \phi^{-1}\der \colon \mathcal{C}_a^1\to \mathcal{C}_a$.
Then $\derdelta$ maps~$\mathcal{C}_a^j$ into $\mathcal{C}_a^{j-1}$ for $j=1,\dots, m$.
 Let $(f_n)$ be a sequence of functions in $\mathcal{C}_a^m$ such that  for $k=0,\dots, m$ the series $\sum_{n=0}^\infty\derdelta^kf_n$ converges to $g_k\in \mathcal{C}_a$, uniformly on~$[a,b]$ for every $b>a$. Then for $k=0,\dots,m$ we have $g_k\in \mathcal{C}_a^{m-k}$ and $f:= g_0\in \mathcal{C}_a^m$ satisfies~$\derdelta^k f=g_k$. 
\end{cor}

\noindent
This corollary and the following results on infinite sums will be used in the next two sections. 
Let $a\in \R$, and for $i=0,1,2,\dots$, let a continuous function $f_i\colon [a,\infty) \to \R$ be given, and set $M_i^n:=\max\limits_{a\le t\le a+n}|f_i(t)|$, so $$0\ \le\  M_i^0\ \le\ M_i^1\ \le\ M_i^2\ \le\ \cdots.$$ Suppose the real numbers
$\epsilon_i> 0$ are such that $\sum_i\epsilon_iM^n_i<\infty$ for every $n$. Then~$\sum_i\epsilon_i f_i$ converges uniformly on each set $[a,a+n]$, and so this sum defines a continuous function on $[a,\infty)$. We can certainly take real numbers $\epsilon_i> 0$ such that~$\sum_i \epsilon_i M_i^i< \infty$, and then we do indeed have for every $n$ that $\sum_i\epsilon_iM^n_i<\infty$, since 
$$\sum_i\epsilon_iM^n_i\ =\ \sum_{i=0}^n\epsilon_i M^n_i + \sum_{i>n}\epsilon_i M^n_i\ \le\ \sum_{i=0}^n\epsilon_i M^n_i + \sum_{i>n}\epsilon_i M^i_i.$$
Thus there exist $\epsilon_i$ as in the hypothesis of the next lemma. {\em In the rest of this subsection we assume that for every $i$
we have $f_i\ge 0$ on $[a,\infty)$ and 
$f_i\prec f_{i+1}$ in $\Cc$}. 

\begin{lemma}\label{sum1}  Let the reals $\epsilon_i>0$
be such that $\sum_i \epsilon_i f_i$ converges to a function~$f\colon [a,\infty)\to \R$, uniformly on each 
compact subset of $[a,\infty)$. Then $f\succ f_n$ \textup{(}in $\Cc$\textup{)} for all $n$. If all $f_i$ are increasing, then so is $f=\sum_i \epsilon_i f_i$. 
\end{lemma}
\begin{proof} Note that
$\sum_i \epsilon_i f_i\ge \epsilon_{n+1}f_{n+1} \succ f_n$.
\end{proof}

\begin{lemma}\label{sum2} Let for each $n$ a continuous function $g_n\colon [a,\infty) \to \R^{>}$ be given such that $f_i\prec g_n$, for all $i$ and $n$.
Then there exist reals $\epsilon_i>0$ for which $\sum_i \epsilon_i f_i$ converges to a function $f\colon [a,\infty)\to \R$, uniformly on each 
compact subset of $[a,\infty)$, such  that $f \le g_n$ in $\Cc$ for all $n$. 
\end{lemma}
\begin{proof} For the moment we just consider one continuous function $g\colon [a,\infty)\to \R^{>}$ with $f_i \prec g$ for all $i$. Then we pick the $\epsilon_i>0$ so small that
$\sum_i \epsilon_iM_i^i< \infty$ and~$\epsilon_if_i\le g/2^{i+1}$. This results in $f:=\sum_i \epsilon_if_i \le g$. Let a second continuous function $h\colon [a,\infty)\to \R^{>}$
be given with $f_i\prec h$ for all $i$. 
Take $b\ge a$ such that~$\epsilon_0f_0\le h/2$ on 
$[b,\infty)$, and next decrease, if necessary, the $\epsilon_i$ with $i\ge 1$ so that $\epsilon_if_i\le h/2^{i+1}$ on $[b,\infty)$ for the new values of $\epsilon_i$.
This results in $f\le h$ on $[b,\infty)$ for the new $f$; note that we did not change $\epsilon_0$. 
Starting with $g=g_0$ we apply this procedure successively to $g_1, g_2,\dots$ in the role of $h$: 
we recursively pick $b_1, b_2,\dots \ge a$, decreasing only the~$\epsilon_i$ for $i\ge n$ when dealing with $g_n$, $n\ge 1$.
Then at the end we have not only $f\le g_0$ on $[a,\infty)$,
but also $f\le g_n$ on $[b_n,\infty)$, for all $n\ge 1$ simultaneously.  
\end{proof}

\noindent
Note that if in Lemma~\ref{sum2} we have $g_0\succ g_1 \succ g_2\succ\cdots$, then $f\prec g_n$ for all $n$. 
Lemmas~\ref{sum1} and~\ref{sum2} are more precise versions of results of du Bois-Reymond~\cite{dBR75} and 
Hadamard~\cite[\S{}19]{Hadamard}, respectively; cf.~\cite[Chapter~II]{Ha}.

\medskip\noindent
Assume next that the $f_i$ are of class $\Cc^\infty$. Then we set
$$M^n_i\ :=\ \max_{j\le n,\ a\le t\le a+n} \big|f_i^{(j)}(t)\big|.$$ 
Again, $0\le M_i^0 \le M_i^1\le M_i^2\le \cdots$. Taking the 
$\epsilon_i> 0$ such that $\sum_i \epsilon_i M_i^i< \infty$, we have
$\sum_i\epsilon_iM^n_i<\infty$ for every $n$, as before.
Hence $\sum_i \epsilon_i f_i$ converges, say to the continuous function $f\colon [a,\infty) \to \R^{\ge}$, uniformly on each $[a, a+n]$. Also $\sum_i \epsilon_i f_i^{(j)}$ converges for every $j$ to a continuous function $f^{(j)}\colon [a,\infty)\to \R$,
uniformly
on each~$[a,a+n]$.  An easy induction on $j$ shows that $f$ is in fact of class $\Cc^j$ with $f^{(j)}$ as its $j$th derivative, as suggested by the notation. Thus $f$ is of class $\Cc^{\infty}$.

\subsection*{Useful inequalities in constructing Hardy fields} The lemmas in this subsection will be used in Sections~\ref{fwg} and~\ref{fbg}.

\begin{lemma}\label{intineq} 
Let   $F,G\in \Cc^1_a$ satisfy $F'(t)\leq G'(t)$ for all $t\geq a$. Then there is a real constant $c$ such that $F<G+c$ on $[a,\infty)$. 
\end{lemma}
\begin{proof} The function $F-G$ is continuous and decreasing, hence on $[a,\infty)$ we have~$F-G\le F(a)-G(a) < c:= F(a)-G(a)+1$.
\end{proof}

\noindent
Here is a useful multiplicative version: 

\begin{lemma}\label{logintineq}  Let   $F,G\colon [a,+\infty) \to \R^{>}$ of class  $\Cc^1$ be such that~$F^\dagger\le  G^\dagger$ on $[a,\infty)$. Then  there is $c\in \R^{>}$ such that  $F< cG$ on $[a,\infty)$. 
\end{lemma}
\begin{proof} We have $F^\dagger=(\log F)'$ and $G^\dagger=(\log G)'$, so Lemma~\ref{intineq} yields $d\in\R$ with~$\log F <d+ \log G$
on $[a,\infty)$,  and thus $F<c\,G$ on $[a,\infty)$ for   $c:=\ex^d\in\R^>$.
\end{proof}

\begin{lemma}\label{in1} Suppose $f\in \Cc$ lies in a Hardy field.
Then the germ $f(x+1)$ satisfies:\begin{enumerate}
\item[(i)] $f(x)-x >_{\ex} \R\ \Longrightarrow\ f(x)+1 <_{\ex} f(x+1)$;
\item[(ii)] $0< f(x)\succ \ex^x\ \Longrightarrow\ f(x) <_{\ex} f(x+1)/2$;
\item[(iii)] $0 < f(x)\prec \ex^{-x}\ \Longrightarrow\ f(x) -f(x+1)>_{\ex} f(x)/2$.
\end{enumerate}
\end{lemma}
\begin{proof} For (i), assume $f=x+g$ with $g\in \Cc$ and $g>\R$. Then $g$
lies in a Hardy field, so $g$ is eventually strictly increasing,
hence $g=g(x) <_{\ex} g(x+1)$, and thus 
$$f(x)+1\  =\ x+1  + g\ <_{\ex}\ x+1+ g(x+1)\ =\ f(x+1).$$
Next, assume $0<f\succ \ex^x$. Then $(\log f)-x>_{\ex} \R$, so
$\big(\log f(x)\big)+1 <_{\ex} \log f(x+1)$ by (i), and thus $\ex f(x) <_{\ex}\ f(x+1)$,
hence $f(x) <_{\ex} f(x+1)/\ex <_{\ex} f(x+1)/2$. As to (iii), assume $0 < f(x) \prec \ex^{-x}$. Applying (ii) to $f^{-1}$ gives $f(x) >_{\ex} 2f(x+1)$, so~$0 <_{\ex} f(x+1)<_{\ex} f(x)/2$, and thus $f(x)-f(x+1) >_{\ex} f(x)/2$. 
\end{proof}

\section{Pseudoconvergence in Hardy Fields}\label{psh1} 

\noindent
Let $H$ be a Hardy field. Let a sequence $$f_0 \succ f_1 \succ f_2 \succ \dots$$ in $H^{>}$
be given. Then we have the pc-sequence $(F_i)$ in $H$, with $F_i:=f_0 + \cdots +f_i$.
Our aim in this section and the next is to show:

\begin{theorem}\label{immom} $(F_i)$ pseudoconverges in some Hardy field extension
of $H$. 
\end{theorem}

\noindent
In view of Lemma~\ref{pcsums} this has the following consequence:

\begin{cor}\label{corimmom} Every pc-sequence of countable length in a maximal Hardy field has a pseudolimit in that Hardy field. 
\end{cor} 

\noindent
Towards the proof of Theorem~\ref{immom} we first recall from~\cite[Sections~3, 4]{ADH5}  the following. 
Let $\ell\in \Cc^{<\infty}$ be such that $\ell>\R$ and $\ell'\in H$. Then $\ell$ lies in a Hardy field extension of $H$,  $\phi:=\ell'\in H^{>}$ is active in $H$, and the compositional inverse~$\ell^{\inv}>\R$ of $\ell$ yields an isomorphism
$f\mapsto f^\circ:=f\circ \ell^{\inv} \colon (\mathcal{C}^{<\infty})^\phi\to \mathcal{C}^{<\infty}$ of differential rings that maps $H$ onto the Hardy field $H^\circ:=H\circ \ell^{\inv}$; moreover,
$f_1\prec f_2\Leftrightarrow f_1\circ \ell\prec f_2\circ \ell$, for all $f_1, f_2\in \mathcal{C}^{<\infty}$. Thus $(F_i^\circ)$ is a pc-sequence in $H^\circ$, and we have:

\begin{lemma}\label{compo} $(F_i)$ pseudoconverges in some Hardy field extension of $H$ if and only if $(F_i^\circ)$ pseudoconverges in some Hardy field extension of $H^\circ$.
\end{lemma}

\noindent
We can also use \cite[Theorem~11.19]{ADH2} to pass to an extension and arrange that $H\supseteq \R$ and $H$ is closed. Then the following lemma is relevant.   
 
\begin{lemma}\label{unad} Let $H\supseteq \R$ be closed. Suppose 
$(F_i)$ has no pseudolimit in $H$, and let any element $F\in \mathcal{C}^{<\infty}$ be given. Then the following are equivalent: \begin{enumerate}
\item[(i)] $F$ generates a Hardy field $H\<F\>$ over $H$ with $F_i\leadsto F$;
\item[(ii)] for all $k$, $m$ with $k < m$ and active $\phi\in H^{>}$ we have
$$\derdelta^k\!\left(\frac{F-F_m}{f_m}\right)\preceq 1 \text{ in }\mathcal{C}^{<\infty}$$
where $\derdelta:= \phi^{-1}\der$ is construed as a derivation of 
$\mathcal{C}^{<\infty}$.
\end{enumerate}  
\end{lemma}
\begin{proof} Assume (i).
Then for all $k$, $m$ and active $\phi\in H^{>}$ we have $(F-F_m)/f_m\prec 1$, and thus $\derdelta^k\big(\frac{F-F_m}{f_m}\big)\prec 1$. This proves 
(i)~$\Rightarrow$~(ii). For (ii)~$\Rightarrow$~(i), assume (ii).
For~$k=0$ we get $F-F_m\preceq f_m$ for all $m\ge 1$. Let
$P\in H\{Y\}^{\ne}$. Now $(F_m)$ is of $\d$-transcendental type
over $H$ by  [ADH,   11.4, 14.0.2], so we have $m_0\ge \text{order}(P)$ in~$\N^{\ge 1}$ such that $\ndeg_{\preceq f_{m+1}} P_{+F_m}=0$ 
for all $m\ge m_0$, by 
[ADH, 11.4.11, 11.4.12]. Using $P_{+F_{m+1}}=(P_{+F_m})_{+f_{m+1}}$ and 
 [ADH, 11.2.7]  we obtain
$\ndeg_{\preceq f_{m+1}} P_{+F_{m+1}}=0$ for all $m\ge m_0$.
Thus for $m_1:= m_0+1$ and  $Q:= P_{+F_{m_1}, \times f_{m_1}}$ we have an active~$\phi_0\in H^{>}$ with $\ddeg Q^{\phi}=0$ for all active $\phi\preceq \phi_0$ in $H^{>}$.
This gives $h\in H^\times$ such that, with $\j$ ranging over $\N^{m_0}$ and $Q^{\phi_0}_{\j}:=(Q^{\phi_0})_{\j}$,
$$Q^{\phi_0}(Y)\ =\ h+\sum_{ |\j|\ne 0} Q^{\phi_0}_{\j}Y^{\j}, \qquad Q^{\phi_0}_{\j}\prec h \text{ for } |\j|\ne 0.$$
Thus 
with $G:=(F-F_{m_1})/f_{m_1}$ we have $G\preceq 1$ and $F=F_{m_1}+ f_{m_1}G$, so
$$ P(F)\ =\ Q^{\phi_0}(G)\ =\ 
h+ \sum_{|\j|\ne 0} Q^{\phi_0}_{\j}G^{\j}$$
where the factors $G^{\j}$ are evaluated in $\Cc^{<\infty}$ using the derivation $\derdelta=\phi_0^{-1}\der$, and so~$G^{\j}\preceq 1$ for $ |\j|\ne 0$, by (ii).
Hence $P(F) \,  \sim\, h$. This yields (i). 
\end{proof}

\begin{cor}\label{unadcor} In Lemma~\ref{unad}  we can replace {\rm(ii)} by any of the two variants
below: \begin{enumerate}
\item[(ii)$^{*}$] for all $m > k$ and active $\phi_0\in H$ there is an active $\phi\preceq \phi_0$ in $H^{>}$ such that
$$\derdelta^k\left(\frac{F-F_m}{f_m}\right)\preceq 1 \text{ in }\mathcal{C}^{<\infty}$$
where $\derdelta:= \phi^{-1}\der$ is construed as a derivation of 
$\mathcal{C}^{<\infty}$.
\item[(ii)$^{**}$] for all $m_0\ge 1$ and active $\phi_0\in H$ there is an active 
$\phi\preceq \phi_0$ in $H^{>}$ and an~$m\ge m_0$ such that for $k=0,\dots,m_0$,
$$\derdelta^k\left(\frac{F-F_m}{f_m}\right)\preceq 1 \text{ in $\mathcal{C}^{<\infty}$,}\  \text{with $\derdelta:=\phi^{-1}\der$.}$$
\end{enumerate}
\end{cor}
\begin{proof} 
For (ii)$^{**}$~$\Rightarrow$~(i),  assume (ii)$^{**}$.  As before we have $F-F_m\preceq f_m$ for all~$m\ge 1$. Take $m_0$, $Q$, $\phi_0$ as in the proof of (ii)~$\Rightarrow$~(i), and set
$Q_m:= P_{+F_{m+1}, \times f_{m+1}}$.  For any $m\ge m_0$ and active $\phi\preceq \phi_0$ in $H^{>}$ we have~$\ddeg Q^\phi=0$, so
 $\ddeg Q_m^{\phi}=0$ by  [ADH, 6.6.12].
Now (ii)$^{**}$ gives active $\phi\preceq \phi_0$ in $H^{>}$ and $m\ge m_0$ such that for $k=0,\dots,m_0$ we have
$\derdelta^k\big(\frac{F-F_m}{f_m}\big)\preceq 1 $ in $\mathcal{C}^{<\infty}$, with $\derdelta:=\phi^{-1}\der$. In view of~$\ddeg Q_m^\phi=0$,
the last part of the proof of the lemma with~$\phi_0$,~$Q$ replaced by~$\phi$,~$Q_m$, and $G$ replaced by $G_m:= (F-F_{m+1})/f_{m+1}$, but $\j$ still ranging over~$\N^{m_0}$, goes through, and
yields the desired conclusion. 
\end{proof}

\noindent
Rather than Lemma~\ref{unad} we shall use  in what follows the implications (ii)$^{*}$~$\Rightarrow$~(i) and  (ii)$^{**}$~$\Rightarrow$~(i) that are implicit in the proof of that lemma, as we saw.

\subsection*{Expressing the powers $\derdelta^k$ in terms of $\der$} To facilitate the use of Lemma~\ref{unad} and its variants we shall  
express $\derdelta^k$ in terms of $\der$. Let $R$ be any differential ring with derivation $\der$. Then $f\in R$
gives rise to a derivation $\derdelta:=f\der$ on the underlying ring of $R$. 
For $k\geq 1$, $0\leq j\leq k$, we define 
$G^k_j(Y)\in \mathbb{Q}\{Y\}\subseteq R\{Y\}$ by recursion:
\begin{itemize}
\item $G^k_0 = 0$,
\item $G^k_k = Y^k$,
\item $G^{k+1}_j = Y\left(\der(G^k_j) +G^{k}_{j-1}\right)$ for $1\leq j\leq k$. 
\end{itemize}
(See also [ADH, 5.7].)
For the additive operators $\der$ and $\derdelta$ on the underlying ring $R$ this recursion gives: 
\[\derdelta^k\ =\ \sum_{j=1}^k G^k_j(f)\der^j \qquad(k\ge 1).\]
For $1\le j \le k$ the differential polynomial~$G^k_j(Y)$ is homogeneous of degree $k$ 
and of order $\le k$, 
so we have a differential polynomial $R^k_j(Z)\in \Q\{Z\}$ 
of order $< k$ and
depending only on $j$ and $k$
such that~$G^k_j(f)=f^k R^k_j(f^\dagger)$ for all $f\in R^\times$; see also~[ADH,  5.8]. 
For~$g\in R$, $\phi\in R^\times$, $\derdelta=\phi^{-1}\der$, this gives
\begin{equation}\label{eq:derdelta^k} \derdelta^k(g)\ =\ \phi^{-k}\sum_{j=1}^k R^k_j(-\phi^\dagger)g^{(j)}\ \text{ with } g^{(j)}:= \der^j(g) \qquad(k\ge 1). \end{equation}
Given $a\in \R$, the identity \eqref{eq:derdelta^k} also holds for $g\in \Cc^k_a$ 
and $\phi\in (\Cc^k_a)^{\times}$, where $\derdelta^k$ and
the $\der^j$ for $j\le k$ are construed in the obvious way as maps
  $\Cc^k_a\to\Cc_a$.

\medskip\noindent
For use in the next section we add the following observation:

\begin{lemma}\label{gd} Let $g\in H$ be active and 
$g\preceq h\in H$, and suppose $f\in \Cc^{<\infty}$ satisfies~$(g^{-1}\der)^k(f)\prec 1$ for $k=0,\dots,m$. Then also $(h^{-1}\der)^k(f)\prec 1$ for $k=0,\dots,m$. 
\end{lemma}
\begin{proof} Set $u:= g/h\in H^{\preceq 1}$, $\derdelta_g:= g^{-1}\der$
and $\derdelta_h:= h^{-1}\der$. Then $\derdelta_h=u\derdelta_g$, as derivations on $H$ and on
$\Cc^{<\infty}$. For $k\ge 1$ we have by an earlier identity
 \begin{equation}\label{eq:gd} \derdelta_h^k(f)\ =\ \sum_{j=1}^k G^k_j(u)\derdelta_g^j(f) \end{equation}
 where each $G^k_j(u)$ is evaluated according to the small derivation
 $\derdelta_g$ on the asymptotic field $H$, and thus $G^k_j(u)\preceq 1$. This gives the desired result. 
\end{proof}

\begin{remarkNumbered}\label{rem:gd}
For later use we note that the identity \eqref{eq:gd}  also holds for~${1\le k\le m}$,
 $f, g\in \Cc_a^m$, $h\in (\Cc_a^m)^\times$, with $a\in \R$ and~$u:=g/h$ (an element of $\Cc_a^m$), and where~$\derdelta_g:= g^{-1}\der$, $\derdelta_h:= h^{-1}\der$ are taken as derivations $\Cc_a^j\to  \Cc_a^{j-1}$, for~$j=1,\dots,m$, and
each~$G^k_j$ is evaluated according to~$\derdelta_g$.  
\end{remarkNumbered}

\subsection*{Bump functions} In this subsection $t$ ranges over $\R$.  From Lemma~\ref{bump} we obtain an increasing $\Cc^{\infty}$-function $\alpha\colon \R\to \R$  with $\alpha(t)=0$ for $t\le 0$ and
$\alpha(t)=1$ for~$t\ge 1$, and below we fix such an $\alpha$. (See~Figure~\ref{fig:alpha}.)

\begin{figure}[ht]
\begin{tikzpicture}
  \begin{axis} [axis lines=center, xmin=-0.25, xmax=1.25, ymin = 0, ymax = 0.12, width=0.4\textwidth, height = 0.35\textwidth, xlabel={$t$}, xtick={0,1},  ytick=0.1, yticklabels={$1$},  extra x ticks = 0]
    \addplot [domain=-0.25:1.25, smooth, very thick] { 0.05*tanh(4+8*(x-1)) + 0.05};
     \node[right] at (axis cs:1,0.08) {$\alpha$};
  \end{axis}
\end{tikzpicture}
\caption{The bump function $\alpha$}\label{fig:alpha}
\end{figure}
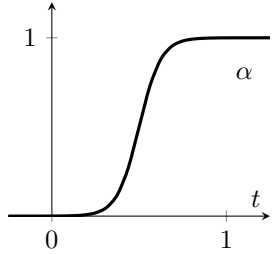

\medskip
\noindent
 For each $n$ we take a real constant $C_n$ such that $1\le C_0\le C_1\le C_2\le \cdots$ and
\begin{equation}\label{eq:alpha(n)}
|\alpha^{(n)}(t)|\ \le\ C_n\ \text{ for all $n$ and $t$.}
\end{equation}
For reals $a< b$ we define the increasing $\Cc^{\infty}$-function 
$\alpha_{a,b}\colon \R \to \R$ by
\begin{equation}\label{eq:alphaab}
\alpha_{a,b}(t)\ :=\  \alpha\!\left(\frac{t-a}{b-a}\right),
\end{equation}
so $\alpha_{a,b}(t)=0$ for $t\le a$ and $\alpha_{a,b}(t)=1$ for $t\ge b$. Also, 
\begin{equation}\label{eq:alphaab(m)}
\big|\alpha_{a,b}^{(m)}(t)\big|\ \le\  \frac{C_m}{(b-a)^m}\ 
\text{ for all $m$ and $t$}.
\end{equation}

\subsection*{Constructing $F^*$} 
We go back to our Hardy field $H$
(not necessarily $\upo$-free or newtonian) and its elements $f_n$ and $F_n:= f_0 + \cdots + f_n$, and in the rest of this section~$t$
ranges over $\R^{\ge 1}$. First, we take for each $n$ a continuous function~${\R^{\ge 1}\to \R}$ that represents the germ $f_n$, to be denoted also by $f_n$, such
that $f_n(t)> 0$ and~$f_{n+1}(t)\le f_n(t)/2$ for all
$t$: first choose the function $f_0$, then $f_1$,
next $f_2$, and so on.  

For each $n$ we fix an $a_n\in \R^{\ge 1}$
such that $f_0,\dots,f_n$ are of class $\Cc^{n}$ on~$[a_n,+\infty)$.
Next, let $c_0 < c_1 < c_2< \cdots$ be real numbers $\ge 1$
with $c_n \to \infty$ as~$n\to \infty$. We define $\alpha_n\colon \R^{\ge 1}\to \R$ by $\alpha_n(t):= \alpha_{c_n, c_{n+1}}(t)$, so
$\alpha_n$ is an increasing
$\Cc^{\infty}$-function with $\alpha_n(t)=0$ for $t\le c_n$ and $\alpha_n(t)=1$ for $t\ge c_{n+1}$, and we set
$$f_n^*\ :=\ \alpha_n f_n\ :\, \R^{\ge 1} \to \R^{\ge 0},$$ so $f_n^*(t)=0$ for $t\le c_n$ and
$f_n^*(t)=f_n(t)$ for $t\ge c_{n+1}$. Thus $f_n$ and $f_n^*$
have the same germ at $+\infty$, and we still have
$f^*_{n+1}(t)\le f_n^*(t)/2$ for all $n$ and $t$. As we saw in the subsection on Hausdorff fields in Section~\ref{prel}, this yields a continuous function
$$F^*\ :=\  \sum_{n=0}^\infty f_n^*\ :\, \R^{\ge 1} \to \R$$
such that $F^*-F_n\prec f_n$ (in $\mathcal{C}$)  for all $n$. 

\begin{lemma}\label{lem26} Assume $c_n > a_0,\dots, a_n$ for all $n$.
Then for all $n$, $f_n^*$ is of class $\Cc^n$, and
$F^*$ is of class $\Cc^n$ on $[c_n,+\infty)$. So the germ of $F^*$ at $+\infty$ belongs to $\mathcal{C}^{<\infty}$. 
\end{lemma} 
\begin{proof} We have $f_n^*=0$ on $[1, c_n]$, and $f_n$ is of class $\Cc^n$
on $[a_n, +\infty)$, so $f_n^*$ is of class $\mathcal{C}^n$.  
For $t\le c_{n+1}$ we have
$F^*(t)=f_0^*(t) + \cdots + f_n^*(t)$, so $F^*$ is of class $\Cc^n$
on~$[a_n, c_{n+1}]$. Likewise,
$F^*$ is of class $\Cc^{n+1}$ on 
$[a_{n+1}, c_{n+2}]$. Continuing this way we obtain that
$F^*$ is of class $\Cc^n$ on $[c_n,+\infty)$. 
\end{proof}

\noindent
We consider the $a_n$ as fixed, with the $c_n>a_0,\dots,a_n$ to be chosen as needed later. 
We set $\varepsilon_m:= f_{m+1}/f_m$, so $0 <\varepsilon_m(t)\le 1/2$ for all $t$ and $\varepsilon_m\prec 1$ in $H$. For any $n> m$ we also set $\varepsilon_{n,m}:=f_n/f_m$, so $\varepsilon_{n,m}$ is of class
$\mathcal{C}^n$ on $[a_n,+\infty)$ and~$0 <\varepsilon_{n,m}(t)\le 2\varepsilon_m(t)/2^{n-m}$ for all $t$.   
Then for $n> m$ we have $\varepsilon_{n,m}\preceq\varepsilon_m\prec 1$ in~$H$, so
$\varepsilon_{n,m}^{(k)}\prec x^{-1}$ for all $k\geq 1$: use [ADH, 9.1.9(iv), 9.1.10], first passing from~$H$ to a Hardy field extension
containing $x$ if necessary. 

\subsection*{Proof in the fluent case} This case of Theorem~\ref{immom} is as follows: 

\begin{prop}\label{flimmom} Suppose $\varepsilon\in H^{\prec 1}$ is such that
$f_{i+1}/f_{i}\prec \varepsilon$ for all $i$. Then $(F_i)$ 
pseudoconverges in some Hardy field extension of $H$.
\end{prop}
\begin{proof} By passing to a suitable extension we first
arrange that $H\supseteq \R$ is  closed. 
Then $\ell:= -\log|\varepsilon|\in H^{>\R}$,  
$|\varepsilon|=\ex^{-\ell}$, so $\abs{\varepsilon}\circ \ell^{\inv}=\ex^{-x}$, and thus~${(f_{i+1}/f_i) \circ \ell^{\inv}}\prec \ex^{-x}$ for all $i$. Replacing $H$ by $H\circ \ell^{\inv}$ and renaming we can arrange in view of Lemma~\ref{compo} that $f_{i+1}/f_i\prec \ex^{-x}$ for all $i$, and this is what we assume below.  Note that then $(f_n/f_m)^{(k)}\prec\ex^{-x}$
for all $n>m$ and all $k$.  We also assume
that $(F_i)$ does not pseudoconverge in $H$.

As in the subsection on constructing $F^*$ we choose for each germ
$f_n$ a continuous representative $\R^{\ge 1} \to \R$, also to be denoted by $f_n$,
and real numbers~$a_0, a_1, a_2,\dots$, $c_0, c_1, c_2,\dots$ with the properties listed there, and with $c_n> a_0,\dots, a_n$ for all $n$: the~$a_n$ are fixed and the $c_n$ are adjustable. As in that subsection this yields an~$F^*=\sum_{n=0}^\infty f_n^*\in \Cc^{<\infty}$ with $f_n^*=\alpha_n f_n$ for all $n$, and
we introduce the 
functions $\varepsilon_m=f_{m+1}/f_m$ and $\varepsilon_{n,m}=f_n/f_m$ for $n>m$. For each $n$ we take 
$b_n\ge a_0,\dots, a_n$ such that for all $k$, $m$ with
$0\le k\le m< n$,  
$$ t\ge b_n\ \Longrightarrow\ 
\big|\varepsilon_{n,m}^{(k)}(t)\big|\ \le\ \frac{\ex^{-t}}{2^{n-m}}.$$
Next, with the $C_n$ from \eqref{eq:alpha(n)},
take $c_n> b_0,\dots, b_n$ such that
$c_{n+1}-c_n\ge C_n$ (so~$c_n\to \infty$).  For $m\le n$ we have $|\alpha_n^{(m)}(t)|\le 1$
for all $t$: this is clear for $m=0$, and for $m\ge 1$ it follows from $C_m\le C_n$ and \eqref{eq:alphaab(m)}.

Let $\phi\in H^{>}$ be active and $\phi\prec 1$,  so
$\phi\succ x^{-2}$ and $\phi^\dagger\preceq x^{-1}$. This gives a derivation $\derdelta:=\phi^{-1}\der$ on $\Cc^{<\infty}$. 
Now we use (ii)$^{*}$~$\Rightarrow$~(i) from Corollary~\ref{unadcor}. It tells us that for $F^*$ to generate a Hardy field over $H$ with
$F_i\leadsto F^*$, it is enough to establish that the present assumptions on $\phi$ imply:

\claim{for all $m>k$ we have
$\derdelta^k\!\big(\frac{F^*-F_m}{f_m}\big)\prec 1$ in $\Cc^{<\infty}$.}

\noindent
Let $m\ge 1$ be given and represent the germ $\phi$ by a  
$\Cc^m$-function $\R^{\ge 1} \to \R^{>}$, to be denoted also by $\phi$. For $1\le j < k$, the  coefficient of $Y^k$ in the homogeneous differential polynomial $G^k_j$ of degree $k$  is $0$, so $G^k_j(1)=R^k_j(0)=0$.  Also $R^k_k=1$ for~$k\ge 1$.  Hence we can
take a real number $c_m^*\ge c_m$ such that for all $t\ge c_m^*$, 
$$\phi(t) \ge\ t^{-2}, \quad \big|R^k_j(-\phi^\dagger)(t)\big|\le 1\ \text{ whenever $1\le j\le k\le m$.}$$ 
Then \eqref{eq:derdelta^k}  yields
$$ \left|\derdelta^k\!\left(\frac{f_n^*}{f_m}\right)(t)\right|\ \le\ t^{2k}\sum_{j=1}^k\left|\left(\frac{f_n^*}{f_m}\right)^{(j)}(t)\right| \qquad(1\le k\le m < n,\ t\ge c_m^*).$$
Here it is relevant that the $f_n^*/f_m$ are of class $\Cc^m$
on $[c_m,+\infty)$ for the derivatives to exist. Next, for
$1\le j \le m < n$ and $t\ge c_m^*$,
\begin{align*} \left|\left(\frac{f_n^*}{f_m}\right)^{(j)}(t)\right|\ &\le\ \sum_{i=0}^j \binom{j}{i} \left|\alpha_n^{(j-i)}(t)\cdot \varepsilon_{n,m}^{(i)}(t)\right|\\
 &\le\ \sum_{i=0}^j \binom{j}{i}\frac{\ex^{-t}}{2^{n-m}}\ =\ 2^{j}\frac{\ex^{-t}}{2^{n-m}}.
 \end{align*}  
Combining this with the previous inequality we get 
$$ \left|\derdelta^k\!\left(\frac{f_n^*}{f_m}\right)(t)\right|\
   \le\ 2^{k+1}t^{2k} \frac{\ex^{-t}}{2^{n-m}}\qquad (1\le k \le m < n,\ t\ge c_m^*).$$
Now $F_m^*:= f_0^* + \cdots + f_m^*$ is of class $\C^m$ on $[c_m^*,\infty)$, so by Lemma~\ref{lem26} the function
$$\frac{F^*-F_m^*}{f_m}\ =\ \sum_{n=m+1}^\infty \frac{f_n^*}{f_m} $$ 
is of class $\C^{m}$ on $[c_m^*,\infty)$. Using also Corollary~\ref{corsumdif} we have for  $t\ge c_m^*$, 
$$ \left|\derdelta^k\!\left(\frac{F^*-F_m^*}{f_m}\right)(t)\right|\ \le\  2^{k+1}t^{2k} \ex^{-t} \qquad (1\le k\le m).$$
Hence
$\derdelta^k\!\left(\frac{F^*-F_m^*}{f_m}\right)
\prec 1 \text{ in }\Cc^{<\infty}$ for $1\le k\le m$.
As $F_m^*$ and $F_m$ are equal as germs in $\Cc^{<\infty}$, this
proves the claim when $k\ge 1$. For $k=0$, use that
$F^*-F_n\prec f_n$ for all $n$.  
\end{proof}

\begin{cor}\label{unccof} If $H^{>\R}$ has uncountable coinitiality, then
$(F_i)$ pseudoconverges in some Hardy field extension of $H$.
\end{cor}

\noindent
Thus to prove Theorem~\ref{immom} it would be enough to show that in every maximal Hardy field its set of positive infinite elements has uncountable coinitiality. However, we were not able to prove the latter directly, and so couldn't exploit this remark. Instead we refine in the next section the previous constructions in the remaining case where $H^{>\R}$ has countable coinitiality.

\subsection*{Remarks on $H^{>\R}$ having countable coninitiality} 
We show that the property of $H^{>\R}$ having countable coinitiality is
fairly robust; this is not used later but has independent interest. More generally, in this subsection $K$ is a pre-$H$-field
with~$\Gamma:= v(K^\times)\ne \{0\}$. Note: 
$K^{>\mathcal O}=K^{>C}$ if $K$ is an $H$-field. If $K$ is ungrounded, then
$\operatorname{ci}(K^{>\mathcal O})=\operatorname{cf}(\Gamma^<)\ge \omega$, and
$\operatorname{cf}(\Gamma^<)=\omega$ iff 
$K$ has a logarithmic sequence (as defined in [ADH, 11.5])
of countable length.  First we refine \cite[Lemma~1.3.20]{ADH4}:

\begin{lemma}\label{uplcof}
Suppose $K$ is not $\upl$-free, and $L$ is a Liouville closed $\d$-algebraic $H$-field extension of $K$.
Then $L$ is $\upo$-free with a logarithmic sequence of length~$\omega$, and $\Gamma^<$ is not cofinal
in $\Gamma_{L}^<$.
\end{lemma}
\begin{proof}
Suppose first that $K$ is grounded. Let $K_\upo$ be the $\upo$-free pre-$H$-field extension of $K$
introduced before [ADH, 11.7.16] (with $K$ in place of $F$ there), identified with a pre-$H$-subfield of $L$ containing $K$ as in the proof of \cite[Lemma~1.3.18]{ADH4}.
The sequence~$(f_n)$ constructed before [ADH, 11.7.16] is a logarithmic sequence in $K_\upo$ with $\Gamma^< < v(f_n)<0$ for all $n\ge 1$.
By \cite[Theorem~1.3.1]{ADH4}, $L$ is $\upo$-free and $\Gamma_{K_{\upo}}^<$ is cofinal in $\Gamma_L^<$, so
$(f_n)$ remains a logarithmic sequence in $L$, and $\Gamma^<$ is not cofinal
in~$\Gamma_{L}^<$.
If $K$ is not grounded we reduce to the grounded case by following the proofs of
\cite[Lemmas~1.3.18--1.3.20]{ADH4}.
\end{proof}

\noindent
Next, let $\mathbf K=(K,I,\Lambda,\Omega)$ be a pre-$\HLO$-field with Newton-Liouville closure $\mathbf K^{\operatorname{nl}}=(K^{\operatorname{nl}},\dots)$;
see [ADH, 16.4]. Recall that $K^{\operatorname{nl}}$ is differentially algebraic over $K$.
The following proposition is analogous to the characterizations  of rational asymptotic integration and of $\upl$-freeness in~\cite[Propositions~1.3.8, 1.3.12]{ADH4}:

\begin{prop}\label{propuplcof} 
The following are equivalent:
\begin{enumerate}
\item[\textup{(i)}] $K$ is $\upo$-free;
\item[\textup{(ii)}] $\Gamma^<$ is cofinal in $\Gamma_L^<$ for every $\d$-algebraic $H$-field extension $L$ of $K$;
\item[\textup{(iii)}] $\Gamma^<$ is cofinal in $\Gamma_{K^{\operatorname{nl}}}^<$.
\end{enumerate}
Moreover, if  $K$ is not $\upo$-free, then $K^{\operatorname{nl}}$ has a logarithmic sequence of length $\omega$. 
\end{prop}
\begin{proof}
The implication (i)~$\Rightarrow$~(ii) holds by  \cite[Theorem~1.3.1]{ADH4}, and
(ii)~$\Rightarrow$~(iii) is clear.
For the rest, note that  if $K$ is not $\upl$-free, then $K^{\operatorname{nl}}$ has a logarithmic sequence of length $\omega$
and $\Gamma^<$ is not cofinal
in $\Gamma_{K^{\operatorname{nl}}}^<$, by Lemma~\ref{uplcof}. Suppose now that $K$ is $\upl$-free but not $\upo$-free.
Then [ADH, 11.8.30] gives $\upo\in K$ with~$\omega\big(\Upl(K)\big)<\upo<\sigma\big(\Upg(K)\big)$.  
By the proof of [ADH, 16.4.6], either $\Omega=\omega(K)^\downarrow$ or
 $\Omega=K\setminus\sigma\big(\Upg(K)\big){}^\uparrow$.
If~$\Omega=\omega(K)^\downarrow$,  then the proof of [ADH, 16.4.6]   yields a   $\upg\in K^{\operatorname{nl}}$
such that~$\upg>0$, $\sigma(\upg)=\upo$, and the pre-$H$-subfield $K_{\upg}:=K\langle\upg\rangle$ of $K^{\operatorname{nl}}$ has a gap.  
  Replacing~$\mathbf K$ by the pre-$\HLO$-subfield $(K_{\upg},\dots)$ of $\mathbf K^{\operatorname{nl}}$
we reduce to the case that $K$ is not $\upl$-free. 
If~$\Omega=K\setminus\sigma\big(\Upg(K)\big){}^\uparrow$, then the proof of [ADH, 16.4.6]
yields $\upl\in K^{\operatorname{nl}}$ such that~$\omega(\upl)=\upo$ and the pre-$H$-subfield $K_\upl:=K\langle\upl\rangle$ of~$K^{\operatorname{nl}}$  
 is not $\upl$-free, so we can argue as before, with $K_\upl$ in place of $K_\upg$. 
\end{proof}

\noindent
Now assume $H\supseteq\R$. Let $M$ be a maximal Hardy field extension of $H$ and 
$$H^{\operatorname{da}}:=\{f\in M:\text{$f$ is $\d$-algebraic over $H$} \}$$
be the $\d$-closure of $H$ in $M$, and let $\mathbf H$, $\mathbf H^{\operatorname{da}}$, $\mathbf M$ be the
canonical $\HLO$-expansions of $H$, $H^{\operatorname{da}}$, $M$, respectively;
see \cite[Sections~12, 13]{ADH2}. Thus $\mathbf H\subseteq\mathbf H^{\operatorname{da}}\subseteq \mathbf M$.
Note that~$\mathbf H^{\operatorname{da}}$ is a Newton-Liouville closure of $\mathbf H$:
  the closed $\HLO$-field   $\mathbf M$ extends~$\mathbf H$ and thus contains  a Newton-Liouville closure  $\mathbf H^{\operatorname{nl}}$
of $\mathbf H$, and $H^{\operatorname{nl}} \subseteq H^{\operatorname{da}}$ since~$H^{\operatorname{nl}}$ is $\d$-algebraic over $H$, so
$H^{\operatorname{nl}} = H^{\operatorname{da}}$ by [ADH, 16.0.3].
If $M^*$ is also a maximal Hardy field extending $H$, then there is
an $H$-field embedding~$H^{\operatorname{da}}\to M^*$   over~$H$ whose image
is the $\d$-closure of~$H$ in $M^*$, by [ADH, 16.4.9].
By Proposition~\ref{propuplcof}:

\begin{cor}
If $H$ is $\upo$-free, then $H^{>\R}$ is coinitial in $(H^{\operatorname{da}})^{>\R}$. If $H$ is not $\upo$-free, then $H^{\operatorname{da}}$ has a logarithmic sequence of length $\omega$. 
In particular, if $H^{>\R}$ has countable coinitiality, then so does 
$(H^{\operatorname{da}})^{>\R}$.
\end{cor}

\section{The Remaining Case}\label{psh2}

\noindent
We keep the assumptions on $H$ and $(f_i)$ from the beginning of Section~\ref{psh1}, and let $t$ range over $\R^{\ge 1}$. For use in the ``remaining case'' we first derive
bounds like those of clause (ii) in Lemma~\ref{unad} for $\phi=1$.

\subsection*{Useful bounds} We adopt the conventions and notations in the subsection
on constructing $F^*$ from the previous section; in particular, the $a_n$ are fixed, and the~$c_n$ will be adjusted so as to get the desired bounds on certain derivatives of the functions $f_n^*/f_m$ with $n>m$.  For each $n$ we take 
$b_n\ge a_0,\dots, a_n$ such that for all~$k$,~$m$ with
$1\le k\le m< n$, 
$$ t\ge b_n\quad \Longrightarrow \quad 
\left|\varepsilon_{n,m}^{(k)}(t)\right|  \le\ \frac{t^{-1}}{2^{n-m}}.$$
Next, with the $C_n$ from \eqref{eq:alpha(n)}, we take
for each $n$ a $c_n > b_0,\dots, b_n$ with~$c_{n+1}-c_n\ge C_n$ (so $c_n\to \infty$ as $n\to \infty$). Then $|\alpha_n^{(m)}(t)|\le 1$
for all $t$ whenever $m\le n$, using that $C_m\le C_n$ for such~$m$,~$n$. Note also that for $m\le n$ the function $f_n^*/f_m$ is of class $\C^n$ on its entire domain $[1,\infty)$ in view of $f_n^*(t)=0$ for $t\le c_n$.

\begin{lemma} For all $k$, $m$ with $k \le m$ we have 
$$\der^k\!\left(\frac{F^*-F_m}{f_m}\right)\prec 1 \quad\text{ in $\mathcal{C}^{<\infty}$.}$$
\end{lemma}
\begin{proof} Let $1\le k \le m < n$. From
$f_n^*/f_m=\alpha_n\varepsilon_{n,m}$ we get for $t\ge c_n$, 
\begin{align*} \left|\left(\frac{f_n^*}{f_m}\right)^{(k)}(t)\right|\ &\le\ \sum_{j=0}^k \binom{k}{j} \left|\alpha_n^{(k-j)}(t)\cdot \varepsilon_{n,m}^{(j)}(t)\right|\\
 &\le\ |\alpha_n^{(k)}(t)| \frac{2\varepsilon_m(t)}{2^{n-m}} + \sum_{j=1}^k \binom{k}{j}\frac{t^{-1}}{2^{n-m}}\ \le \ \frac{2\varepsilon_m(t)}{2^{n-m}}+ \frac{2^k t^{-1}}{2^{n-m}}.
 \end{align*}
This also holds for $t < c_n$, since $(f_n^*/f_m)(t)=0$
for such $t$. Now fix $m\ge 1$ and set~$F_m^*:= f_0^* + \cdots + f_m^*$.  By Corollary~\ref{corsumdif} the function
$$\frac{F^*-F_m^*}{f_m}\ =\ \sum_{n=m+1}^\infty \frac{f_n^*}{f_m} $$ 
is of class $\C^{m+1}$ on its entire domain $[1,\infty)$, and for all $t$, 
$$ \left|\der^k\!\left(\frac{F^*-F_m^*}{f_m}\right)(t)\right|\ \le\  2\varepsilon_m(t)+ 2^k t^{-1} \qquad (k=1,\dots,m).$$
Hence $\der^k\big(\frac{F^*-F_m^*}{f_m}\big)\prec 1$ in $\mathcal{C}^{<\infty}$ for $k=1,\dots,m$. As $F_m^*$ and $F_m$ are equal as germs in $\mathcal{C}^{<\infty}$, this gives the desired result when $k\ge 1$. For $k=0$, use that $F^*-F_n\prec f_n$ for all $n$. 
\end{proof} 

\noindent
For later use we record the following consequence:

\begin{cor}\label{activecase} Let $\phi\in H^{>}$ be active, and $\derdelta:= \phi^{-1}\der$, as a derivation on $\Cc^{<\infty}$. Then there exists $F_{\phi}\in \Cc^{<\infty}$ such that for all 
$k$, $m$ with $k\le m$,
$$\derdelta^k\left(\frac{F_{\phi}-F_m}{f_m}\right)\prec 1\quad \text{ in $\mathcal{C}^{<\infty}$.}$$
\end{cor}
\begin{proof} Take an  $\ell$ in a Hardy field extension of $H$ with $\ell'=\phi$; note that $\ell>\R$. The lemma above applied to
the sequence $(f_i\circ \ell^{\inv})$ in $H\circ \ell^{\inv}$ yields $F^*\in \Cc^{<\infty}$ with~${\der^k\big(\frac{F^*-F_m\circ \ell^{\inv}}{f_m\circ \ell^{\inv}}\big)\prec 1}$ in~$\mathcal{C}^{<\infty}$ for all $m\ge k$. For
$F_{\phi}:= F^*\circ \ell\in \Cc^{<\infty}$ this gives the desired
result.
\end{proof}

\noindent
In view of Lemma~\ref{unad}, the problem is that
$F_{\phi}$ depends on $\phi$. The idea, to be carried out
in the next subsections, is to show that for suitable
$\phi_n$ and a kind of partition of unity $(\beta_n)$
the
infinite sum $\sum_n\beta_n F_{\phi_n}$  has the desired properties. In the previous section we
proved Theorem~\ref{immom} in the so-called fluent case,
which includes the case that $H^{>\R}$ has uncountable coinitiality. The remaining case where $H^{>\R}$ has countable
coinitiality will lead to the suitable $\phi_n$ and the partition of unity $(\beta_n)$ that we alluded to. The $a_n$ and $b_n$
below  are still real numbers but have little to do with the earlier $a_n$ and $b_n$;
reusing these symbols with another meaning simply reflects the limitations of the alphabet. 

\subsection*{Towards constructing a good partition of unity} Until further notice the Hardy field
$H\supseteq \R$ is Liouville closed and $H^{>\R}$ has countable
coinitiality. It follows that there is a sequence $(\phi_n)$ of active elements in $H^{>}$ such that $\big(v(\phi_n)\big)$ is strictly increasing and cofinal in $\Psi_H$. Below we fix such a sequence $(\phi_n)$, and set $\derdelta_n:=\phi_n^{-1}\der$, a derivation on $\Cc^{<\infty}$. Then Corollary~\ref{activecase}  provides for each $n$ a $\Phi_n\in \Cc^{<\infty}$ such that for all $k$, $m$ with $k\le m$,
$$\derdelta_n^k\!\left(\frac{\Phi_n-F_m}{f_m}\right)\ \prec\ 1\quad \text{ in $\Cc^{<\infty}$,}$$
and thus by Lemma~\ref{gd}, for all $k\le m$ and all $i\le n$,
 $$\derdelta_i^k\!\left(\frac{\Phi_n-F_m}{f_m}\right)\ \prec\ 1\quad \text{ in $\Cc^{<\infty}$.}$$
We represent the germs $\phi_n$, $f_n$, and $\Phi_n$ by
$\C^n$-functions $\R^{\ge 1} \to \R^{>}$, denoted also by
$\phi_n$, $f_n$, and $\Phi_n$. These functions $\phi_n$, $f_n$ and
$\Phi_n$ are fixed in the rest of this section, and the notion of
``admissible sequence'' defined below is relative to these given
sequences $(\phi_n)$, $(f_n)$, $(\Phi_n)$. 
Suppose the real numbers $a_n\ge 1$ are such that: \begin{enumerate}
\item[(I)] for each $n$, $f_0,\dots, f_n$ and $\phi_0,\dots, \phi_n$ are of class $\mathcal{C}^n$ on $[a_n,+\infty)$;  
\item[(II)] for all $i$, $k$, $m$, $n$ with $k\le m\le n$, $i\le n$, and all $t\ge a_n$ we have
$$\left|\derdelta_i^k\!\left(\frac{\Phi_n-F_m}{f_m}\right)(t)\right|\ \le\ 1.$$
\end{enumerate}
Note that (II) makes sense in view of (I), and that 
(I) and (II) remain valid upon increasing all $a_n$. We have $\phi_n/\phi_i\prec 1$ in $H$ for $i<n$ and $\phi_n/\phi_i=1$ for $i=n$,
and thus $\derdelta_n^k(\phi_n/\phi_i)\prec 1$ in $H$ for $i\le n$ and $k\ge 1$. Note also that $\derdelta_n^k(\phi_n/\phi_i)(t)$ is defined for $i,k\le n$ and $t\ge a_n$, since $\phi_n/\phi_i$ is of class $\Cc^n$ on $[a_n,+\infty)$ for $i\le n$. Thus by taking the $a_n$ large enough we can arrange in addition to (I) and (II): \begin{enumerate}
\item[(III)] for all $n$ and $i,k\le n$ and all $t\ge a_n$ we have
$$|\derdelta_n^k(\phi_n/\phi_i)(t)|\ \le\ 1.$$
\end{enumerate}
An {\bf admissible sequence\/} is 
a sequence $\big((a_n, b_n, \beta_n)\big)_{n\geq 0}$ of triples $(a_n, b_n, \beta_n)$ such that: \begin{enumerate}
\item[(i)]  $(a_n)$ is a strictly increasing sequence of real numbers $\ge 1$ with $a_n\to \infty$ as~$n\to \infty$ for which (I), (II), (III) hold; 
\end{enumerate}
and such that for all $n$: \begin{enumerate}
\item[(ii)]  $b_n$ is a real number with $a_n < b_n < a_{n+1}$;
\item[(iii)] $\beta_n$ is a function $\R^{\ge 1} \to \R$ of class $\mathcal{C}^n$;
\item[(iv)]
$\beta_n(t)=0$ if $t\le a_n$, $\beta_n$ is increasing on $[a_n, b_n]$, 
$\beta_n(t)=1$ if $b_n \le t\le a_{n+1}$, $\beta_{n}$ is decreasing on $[a_{n+1}, b_{n+1}]$, and $\beta_{n}(t)=0$ for $t\ge b_{n+1}$;
\item[(v)]  $\beta_n+\beta_{n+1}=1$ on $[a_{n+1}, b_{n+1}]$.
\end{enumerate}
(See Figure~\ref{fig:betan}.)

\begin{figure}[ht]
\begin{tikzpicture}
\def\an{0.75};\def\bn{1.5};\def\anplone{2.5};\def\bnplone{3.7};\def\anpltwo{5.2};\def\bnpltwo{6};
  \begin{axis} [axis lines=center, xmin=-0.2, xmax=7, ymin = -0.01, ymax = 0.15, width=0.8\textwidth, height = 0.4\textwidth, xlabel={$t$}, xtick={\an,\bn,\anplone,\bnplone,\anpltwo,\bnpltwo}, ytick=\empty, xticklabels={\strut $a_n$, \strut $b_n$, \strut $a_{n+1}$, \strut $b_{n+1}$, \strut $a_{n+2}$, \strut $b_{n+2}$},
  legend style={font=\small},
 legend cell align=left, legend style={at={(1,0.75)},anchor=west}
  ]
    \addplot [domain=0.25:\an+0.05, smooth, very thick] {0}; 
    \addplot [domain=0.25:\anplone, smooth, ultra thick, dotted] {0}; 
    
    \addlegendentry{$\beta_n$};
    \addlegendentry{$\beta_{n+1}$};

        \addplot [domain=\bn:\anplone, smooth, very thick] {0.1};  
    \addplot [domain=\bnplone:6.5, smooth, very thick] {0}; 

    \addplot [domain=\bnplone:\anpltwo, smooth, ultra thick, dotted] {0.1}; 
    \addplot [domain=\bnpltwo:6.5, smooth, ultra thick, dotted] {0}; 
    \def\transanbn{(6.5*(((2*(x-\bn)/((\bn)-(\an)))+1)))};
    \def\transanplonebnplone{(6.5*(((2*(x-\bnplone)/((\bnplone)-(\anplone)))+1)))};
    \def\transanpltwobnpltwo{(6.5*(((2*(x-\bnpltwo)/((\bnpltwo)-(\anpltwo)))+1)))};

    \addplot [domain=\an:\bn, smooth, very thick] { 0.1/(1+exp(-\transanbn)) };
    \addplot [domain=\anplone:\bnplone, smooth, very thick] { 0.1/(1+exp(\transanplonebnplone)) };
    \addplot [domain=\anplone:\bnplone, smooth, ultra thick, dotted] { 0.1-0.1/(1+exp(\transanplonebnplone)) };
    \addplot [domain=\anpltwo:\bnpltwo, smooth, ultra thick, dotted] { 0.1/(1+exp(\transanpltwobnpltwo)) };
  \end{axis}
\end{tikzpicture}
\caption{The functions $\beta_n$, $\beta_{n+1}$}\label{fig:betan}
\end{figure}
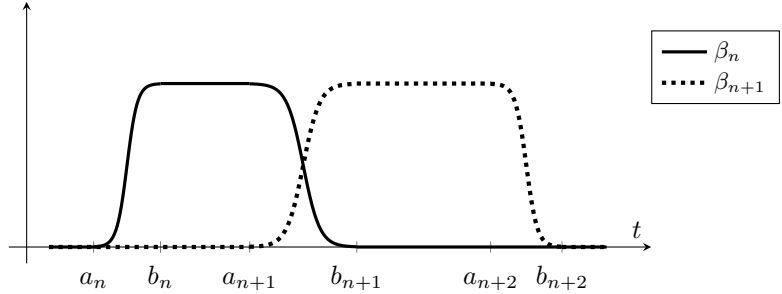

\medskip
\noindent
{\em In the rest of this subsection $\big((a_n, b_n, \beta_n)\big)$ denotes an admissible sequence}. Note that
$\supp \beta_n\subseteq [a_n, b_{n+1}]$ by (iv), and that (v)
expresses the ``partition of unity'' requirement. The series
$\sum_n \beta_n\Phi_n$ converges pointwise on $\R^{\ge 1}$ to
a continuous function $\Phi$ such that on each segment $[b_n, a_{n+2}]$ we have
$$\beta_n + \beta_{n+1}\ =\ 1\ \text{ and }\ \Phi\ =\ \beta_n\Phi_n+ \beta_{n+1}\Phi_{n+1}, $$
so $\Phi$ is of class $\mathcal{C}^n$ on $[b_n, a_{n+2})$. 
Likewise, $\Phi$ is of class $\mathcal{C}^{n+1}$ on the set
$(b_{n+1}, a_{n+3}]$, which overlaps the previous set. Continuing this way we see that
$\Phi$ is of class $\mathcal{C}^n$ on $[b_n, +\infty)$, and thus
the germ of $\Phi$ lies in $\mathcal{C}^{<\infty}$.

\begin{lemma}\label{redcon} Suppose for all $m$ and
$i,k\le m$ there is a positive constant $C=C(i,k,m)$ such that
for all $n\ge m$,
$$ \left|\derdelta_i^k\!\left(\frac{\beta_n\Phi_n + \beta_{n+1}\Phi_{n+1} - F_m}{f_m}\right)\right|\ \le\ C\ \text{ on }\ [a_{n+1}, b_{n+1}].$$
Then for all $m$ and $i,k\le m$ we have
$$\derdelta_i^k\!\left(\frac{\Phi-F_m}{f_m}\right)\ \preceq\ 1\quad \text{ in $\Cc^{<\infty}$.}$$
\end{lemma}
\begin{proof} Let $i,k\le m$, and take a $C\ge 1$ as in the hypothesis.
Then for all $n\ge m$ we have
$|\derdelta_i^k\big(\frac{\Phi - F_m}{f_m}\big)|\le C$ on $[a_{n+1}, b_{n+1}]$, and also by (II) above, 
$$\left|\derdelta_i^k\!\left(\frac{\Phi - F_m}{f_m}\right)\right|\le 1 \text{ on $[b_{n}, a_{n+1}]$,} \qquad
  \left|\derdelta_i^k\!\left(\frac{\Phi - F_m}{f_m}\right)\right|\le 1 \text{ on $[b_{n+1}, a_{n+2}]$,}$$ and thus
$|\derdelta_i^k\big(\frac{\Phi - F_m}{f_m}\big)|\le C \text{ on } [b_{n}, a_{n+2}]$. Taking the union over all $n\ge m$ we obtain
$$\left|\derdelta_i^k\!\left(\frac{\Phi - F_m}{f_m}\right)\right|\le C \text{ on } [b_{m}, +\infty),$$
which gives the desired result.
\end{proof}

\noindent
We didn't use (III) yet, but we need it for a further reduction:

\begin{lemma}\label{redc} Suppose for all $m$ and $k\le m$ there is a positive constant $C(k,m)$ such that for all $n\ge m$,
$$ \left|\derdelta_n^k\left(\frac{\beta_n\Phi_n + \beta_{n+1}\Phi_{n+1} - F_m}{f_m}\right)\right|\ \le\ C(k,m)\ \text{ on }\ [a_{n+1}, b_{n+1}].$$
Then for all $m$ and $i,k\le m$ we have
$$\derdelta_i^k\left(\frac{\Phi-F_m}{f_m}\right)\ \preceq\ 1\quad \text{ in $\Cc^{<\infty}$.}$$
\end{lemma}
\begin{proof} Let $C(m):= \max_{k\le m} C(k,m)$ where the $C(k,m)$ are as in the hypothesis. Let $i,k\le m\le n$, and let $F,g,h$ be the restrictions of $\frac{\beta_n\Phi_n + \beta_{n+1}\Phi_{n+1} - F_m}{f_m}$, $\phi_n$, $\phi_i$ to 
$[a_{n+1},+\infty)$, respectively; these functions are of class
$\Cc^{n}$. For $j=1,\dots,m$ we denote the derivations
$$f\mapsto g^{-1}f'\ :\,\Cc^{j}_{a_{n+1}}\to \Cc^{j-1}_{a_{n+1}},
\quad f\mapsto h^{-1}f'\ :\,\Cc^{j}_{a_{n+1}}\to \Cc^{j-1}_{a_{n+1}}$$ by $\derdelta_g$ and
$\derdelta_h$, suppressing for convenience the dependence on $j$. Let $u:= g/h\in \Cc^m_{a_{n+1}}$. Then Remark~\ref{rem:gd} gives for $f\in \Cc^{m}_{a_{n+1}}$:
$$\derdelta_h^k(f)\ =\ \sum_{j=0}^k G^k_j(u)\derdelta_g^j(f)\quad \text{ (with $G^0_0:=1$ to handle the case $k=0$)},$$
where $G^k_j$ is evaluated according to the derivation 
$\derdelta_g$.  By our hypothesis,
$$|\derdelta_g^j(F)|\le C(m)\ \text{ on }[a_{n+1}, b_{n+1}],\ j=0,\dots,k.$$
Now (III) provides a positive constant $B(m)$ depending on $m$
but not on $n$, such that $|G^k_j(u)|\le B(m)$  on $[a_{n+1}, b_{n+1}]$ for $j=0,\dots,k$.
Hence 
$$|\derdelta_h^k(F)|\ \le\  (k+1)B(m)C(m)\ \text{ on }[a_{n+1}, b_{n+1}], $$
and so the hypothesis of the previous lemma is satisfied. 
\end{proof}

\begin{cor}\label{corredc} If $H$ is closed, $(F_i)$ has no pseudolimit in $H$, and  the hypothesis of Lemma 4.4 is satisfied, then $\Phi$ generates a Hardy field over $H$ and
$F_i\leadsto \Phi$.
\end{cor}
\begin{proof} Use the conclusion of Lemma~\ref{redc}, and the implication (ii)$^{**}$~$\Rightarrow$~(i) (with $\Phi$ in the role of $F$) from Corollary~\ref{unadcor}. \end{proof} 

\noindent
In order to make a further reduction, note that on $[a_{n+1}, b_{n+1}]$ we have $\beta_{n}=1-\beta_{n+1}$, and so, on $[a_{n+1}, b_{n+1}]$,
\begin{align}\tag{$\ast$} \label{eq:betaPhi} \frac{\beta_n\Phi_n + \beta_{n+1}\Phi_{n+1} - F_m}{f_m}\ &=\ \beta_{n+1}\cdot\left(\frac{\Phi_{n+1}-\Phi_n}{f_m}\right) + \frac{\Phi_n-F_m}{f_m}\\
&=\ \beta_{n+1}\cdot\left(\frac{\Phi_{n+1}-F_m}{f_m}-\frac{\Phi_n-F_m}{f_m}\right) + \frac{\Phi_n-F_m}{f_m}. \notag
\end{align}
This leads to a further simplification: 

\begin{lemma}\label{redcc} If for all $k$ there is a constant $B(k)>0$ such that for all $n\ge k$, $|\derdelta_n^k(\beta_{n+1})|\le B(k)$ on $[a_{n+1}, b_{n+1}]$,
then the hypothesis of Lemma~\ref{redc} is satisfied.
\end{lemma}
\begin{proof} To simplify notation, set $G_{n,m}:= \frac{\Phi_n-F_m}{f_m}$, and let $k\le m\le n$. Then by \eqref{eq:betaPhi},
\begin{align*} \derdelta_n^k\left(\frac{\beta_n\Phi_n + \beta_{n+1}\Phi_{n+1} - F_m}{f_m}\right)\ &=\ \sum_{j=0}^k\binom{k}{j}\derdelta_n^{j}(\beta_{n+1})\derdelta_n^{k-j}(G_{n+1,m}-G_{n,m})\\
&\qquad + \derdelta_n^k(G_{n,m})
\end{align*}
on $[a_{n+1}, b_{n+1}]$. Suppose $B\in \R^{>}$ and $|\derdelta_n^j(\beta_{n+1})|\le B$ on $[a_{n+1}, b_{n+1}]$
for $j=0,\dots, k$. Then the above identity and (II) gives that on
$[a_{n+1}, b_{n+1}]$,
$$ \left|\derdelta_n^k\left(\frac{\beta_n\Phi_n + \beta_{n+1}\Phi_{n+1} - F_m}{f_m}\right)\right|\ \le\ \left[\sum_{j=0}^k\binom{k}{j}\cdot B\cdot 2\right] + 1\ =\ 2^{k+1}B+1,$$
which gives the desired result. 
\end{proof}

\subsection*{Using composition}
In this subsection we explore how we might arrange
that our admissible sequence $\big((a_n, b_n, \beta_n)\big)$
satisfies the hypothesis of Lemma~\ref{redcc}, and thus of Lemma~\ref{redc}.  In the next subsection we then construct such a sequence. 

By (iv) there is no problem for $k=0$, since $0 \le \beta_{n+1}\le 1$. Assume~${1\le k \le n}$ and 
set $\phi:= \phi_n$, so $\derdelta:= \phi^{-1}\der=\derdelta_n$, and set $a:= a_{n+1}$, $b:= b_{n+1}$, $\beta\ :=\ \beta_{n+1}$.
We wish to
bound $|\derdelta^k(\beta)|$ on $[a, b]$  by a positive constant that may depend on~$k$ but not on~${n\ge k}$. To achieve this goal we introduce the
strictly increasing bijection~$g\colon\R^{\ge 1} \to \R^{\ge 0}$ given by $g(r)=\int_1^r\phi(t)\,dt$, so
$g\in \Cc_1^{n+1}$, $g'=\phi$, and~$g$ has as compositional inverse
the strictly increasing bijection $g^{\inv}\colon \R^{\ge 0} \to \R^{\ge 1}$ of class~$\Cc^{n+1}$.
Induction on $j\le n+1$ gives $\derdelta^j\beta=(\beta\circ g^{\inv})^{(j)}\circ g$ on $\R^{\ge 1}$. 
For $j=k$ this identity gives  for any $B\in \R^{>}$ the equivalence
\begin{equation}\label{eq:delta&k(beta)}
|\derdelta^k(\beta)|\le B \text{ on $[a, b]$}\quad \Longleftrightarrow\quad |(\beta\circ g^{\inv})^{(k)}|\ \le\ B \text{ on $\big[g(a), g(b)\big]$.}
\end{equation}
We shall arrange below that $b$ is given in terms of $a$ by $g(b)=g(a)+1$ and that on $\big[g(a), g(b)\big]$
the function $\beta\circ g^{\inv}$ equals $\alpha_{g(a), g(b)}$ with $\alpha$ the bump function from Section~\ref{psh1}.
In the next subsection we show that then for sufficiently fast growing~$(a_n)$ all our constraints are satisfied.

\subsection*{The construction} In view of the dependence of the function
$g$ on $n$ in the story above we restore here indices, defining the strictly increasing bijection 
$$g_n\ :\,\R^{\ge 1} \to \R^{\ge 0}, \qquad  g_n(r)\ :=\ \int_1^r\phi_n(t)\,dt,$$ so
$g_n\in \Cc^{n+1}_1$, $g_n'=\phi_n$, and $g_n$ has as compositional inverse
the strictly increasing bijection $g_n^{\inv}\colon \R^{\ge 0} \to \R^{\ge 1}$ of class $\Cc^{n+1}$. 

Next we take a strictly increasing sequence $(a_n)$ of real numbers $\ge 1$ such that~$a_n\to \infty$ and (I), (II), (III) hold, and such that for every $n$ we have
$b_{n+1} < a_{n+2}$ where the real number
$b_{n+1}$ is defined by $g_n(b_{n+1})=g_n(a_{n+1})+1$. 
(It will be clear that there are such sequences.) 
With $b_0$ any real number satisfying~$a_0 < b_0 < a_1$, we now have $a_n < b_n < a_{n+1}$ for all $n$. 
We define for each $n$ the $\Cc^{\infty}$-function
$$\alpha_n  := \alpha_{g_n(a_{n+1}), g_n(b_{n+1})} \colon\R \to \R .$$
The bump function $\alpha$ came with   constants~$C_k>0$ such that $|\alpha^{(k)}|\le C_k$ on $\R$,   thus~$|\alpha_n^{(k)}|\le C_k$ on $\R$
for all~$k$,~$n$ in view of $g_n(b_{n+1})-g_n(a_{n+1})=1$ and \eqref{eq:alphaab(m)}.   
Since $\alpha_n\big(g_n(b_{n+1})\big)=1=1-\alpha_{n+1}\big(g_{n+1}(b_{n+1})\big)$ we can define $\beta_{n+1}\colon \R^{\ge 1} \to \R$ by 
$$\beta_{n+1}(t)\ =\  \begin{cases}
\ \alpha_n\big(g_n(t)\big)			& \text{ for $t \le b_{n+1}$,}\\
\ 1-\alpha_{n+1}\big(g_{n+1}(t)\big)	& \text{ for $t \ge b_{n+1}$.} 
\end{cases}$$
Then $\beta_{n+1}$ is continuous. We also take a continuous function $\beta_0\colon \R^{\ge 1} \to \R$ 
such that (iii), (iv), (v) hold for $n=0$.
We now have constructed a sequence~$\big((a_n, b_n, \beta_n)\big)$ that satisfies
conditions (i), (ii), (iv), and (v) (and (iii) for $n=0$). In fact, it fulfills all our wishes:

\begin{prop}\label{goodseq} The sequence $\big((a_n, b_n, \beta_n)\big)$ is admissible, and for all $k$ and~$n\ge k$ we have
$|\derdelta_n^k(\beta_{n+1})|\le C_k$ on $[a_{n+1}, b_{n+1}]$. 
\end{prop} 
\begin{proof} 
Clearly $\beta_{n+1}$ is of class $\mathcal C^{n+1}$ on $\R^{\geq 1}\setminus\{b_{n+1}\}$.
Now $\alpha_n\circ g_n=1$ on~$[b_{n+1},\infty)$, so ${(\alpha_n\circ g_n)^{(j)}(b_{n+1})}=0$ for $j=1,\dots, n+1$. Moreover, $\alpha_{n+1}\circ g_{n+1}=0$ on~$[1,a_{n+2}]$, so 
$\beta_{n+1}$ is  $\mathcal C^{n+1}$ on all of~$\R^{\ge 1}$.  Therefore condition (iii) is satisfied, and so~$\big((a_n, b_n, \beta_n)\big)$ is admissible.
 The  bound  $|\derdelta_n^k(\beta_{n+1})|\le C_k$ on~$[a_{n+1}, b_{n+1}]$  for~$n\ge k$ is clear from~$\beta_{n+1}\circ g_n^{\inv}=\alpha_n$  on~$\big[g_n(a_{n+1}), g_n(b_{n+1})\big]$ and the equivalence \eqref{eq:delta&k(beta)}.
\end{proof}

\subsection*{Finishing the proof of Theorem~\ref{immom}} As already mentioned we can use \cite[Theorem~11.19]{ADH2} to pass to an extension and arrange that
$H\supseteq \R$ is closed.
If~$H^{>\R}$ has uncountable coinitiality, we are
done by Corollary~\ref{unccof}. Suppose $H^{>\R}$ has countable coinitiality. Then we have an admissible sequence as in Proposition~\ref{goodseq}, and so
by Lemma~\ref{redcc}  and Corollary~\ref{corredc}, if $(F_i)$ has no
pseudolimit in $H$, then $\Phi$ is a
pseudolimit of $(F_i)$ in a Hardy field extension of $H$.
This concludes the proof. 

\begin{cor}\label{cormhci} Suppose $H$ is a maximal Hardy field. Then $\ci(H^{>\R})>\omega$.
\end{cor} 
\begin{proof} If $\ci(H^{>\R})=\omega$, then $H$ being $\upl$-free yields a divergent pc-sequence $(\upl_{\rho})$ in~$H$ whose well-ordered index set has cofinality
$\omega$, contradicting Corollary~\ref{corimmom}. 
\end{proof}

\section{Constructing Overhardian Germs}\label{bh}

\noindent
Our goal in this section is the following:

\begin{theorem}\label{sjo} If $H\supseteq \R$ is a Liouville closed Hardy field
and $\phi\in \Cc$, $\phi>_{\ex} H$, then some
$y\in \Cc^{\infty}$ 
with $y>_{\ex}\phi$ generates a Hardy field $H\<y\>$ over $H$. 
\end{theorem}

\noindent
This is Sj\"odin's main result in \cite{S}, except that he considers only $\Cc^{\infty}$-Hardy fields. 
Our construction of $y$ follows that of Sj\"odin, with the material organized so that much of it will also be useful in the next section where we fill more general gaps.

Boshernitzan~\cite[Theorems~1.1, 1.2]{Boshernitzan86} (see also \cite[proof of Corollary~5.23]{ADH5})
showed that $y$ in 
Theorem~\ref{sjo} can  be taken in $\Cc^\omega$,   using a result of Kneser~\cite{Kneser} on   solutions~$E\in\Cc^\omega$ to the 
functional equation ${\exp}\circ E=E\circ (x+1)$. (Our follow-up paper will have  a different argument that yields a $y\in\Cc^\omega$ in Theorem~\ref{sjo}.)

\medskip
\noindent
We state here an easy consequence of Theorem~\ref{sjo}:

\begin{cor}\label{corsjo} If $H$ is a maximal Hardy field, then $\cf(H) > \omega$, and thus
$$\ci(H)\ =\ \cf(H^{<a})\ =\ \ci(H^{>a})\ >\ \omega\  \text{ for all $a\in H$.}$$
\end{cor}
\begin{proof} If  $H$ is a maximal Hardy field with a strictly increasing cofinal sequence $(h_n)$ in~$H$, then Lemma~\ref{sum1} yields a $\phi\in \Cc$ such that
$h_n <_{\ex} \phi$ for all~$n$, contradicting Theorem~\ref{sjo}.
For any Hausdorff field~$F$ and $a\in F$ we   have~$\cf(F)=\ci(F)=\cf(F^{<a})=\ci(F^{>a})$
(use fractional linear transformations). 
\end{proof}

\noindent
This corollary yields Theorem~\ref{mt} in the case where $A$ or $B$ is finite.  

\subsection*{Lemmas on logarithmic derivatives} Let $f\in \Cc_a^1$. Note that if $f(t)>0$ and~$f'(t)>0$ for all $t\ge a$, then $f$ is strictly increasing, and thus $f(t)\ge f(a)>0$ for all $t\ge a$.
It is convenient to replace here $f'$ by $f^\dagger$, noting that if $f(t)>0$ for all~$t\ge a$, then $f^\dagger(t)$ is defined for all~$t\ge a$. Thus if $f(t)>0$ and $f^\dagger(t)>0$ for all~$t\ge a$, then $f$ is strictly increasing, and thus $f(t)\ge f(a)>0$ for all~$t\ge a$. 

\begin{lemma}\label{ha1} Let $f\in \Cc_a^2$. Assume that $f(t), f^\dagger(t), f^{\dagger\dagger}(t)>0$ for all~$t\ge a$.
Then~$f(t)\to +\infty$ as $t\to +\infty$. 
\end{lemma}
\begin{proof}
Applying the remark preceding the lemma to $f^\dagger$ in the role of $f$ gives  
$f^\dagger(t)\ge f^\dagger(a)$ for $t\ge a$, so 
$f'(t)= f(t)f^\dagger(t)\ge f(a) f^\dagger(a)$ for $t\ge a$, where we apply that same remark also to $f$. Hence for $t\ge a$,
$$f(t)\ =\ f(a)+\int_{a}^t f'(s)\,ds\ \ge\ f(a) + (t-a)f(a) f^\dagger(a),$$
which gives the desired conclusion.  \end{proof}

\begin{lemma}\label{ha2} Let $f\in \Cc_a^3$, and suppose that for all $t\ge a$,
$$f(t)>0, \quad f^\dagger(t)>0,\quad f^{\dagger\dagger}(t)>0,\ \text{ and }
f^\dagger(t)\ >\ f^{\dagger\dagger}(t)\ >\ f^{\dagger\dagger\dagger}(t).$$
Then  $f(t)/f^\dagger(t)\to +\infty$ as $t\to +\infty$. 
\end{lemma}
\begin{proof} We have
$(f/f^\dagger)^\dagger=f^\dagger-f^{\dagger\dagger}$, so $(f/f^\dagger)^\dagger(t)>0$ for $t\ge a$. Also
\begin{align*} {(f/f^\dagger)^{\dagger}}'\ &=\ (f^\dagger-f^{\dagger\dagger})'\ =\ {f^{\dagger}}'-{f^{\dagger\dagger}}'\\
&=\ f^{\dagger\dagger}f^\dagger-f^{\dagger\dagger\dagger}f^{\dagger\dagger}\  =\ f^{\dagger\dagger}(f^\dagger-f^{\dagger\dagger\dagger})
\end{align*}
and thus $(f/f^\dagger)^{\dagger\dagger}(t)>0$ for all $t\ge a$.
Applying Lemma~\ref{ha1} to $f/f^\dagger$ in the role of~$f$ now gives the desired result.
\end{proof}

\begin{lemma}\label{ha3} Let $f\in \Cc_a^4$, and suppose that for all $t\ge a$,
\begin{align*} &f(t)>0, \quad f^\dagger(t)>0,\quad f^{\dagger\dagger}(t)>0,\quad f^{\dagger\dagger\dagger}(t)>0, \text{ and}\\
&f^\dagger(t)\ >\ f^{\dagger\dagger}(t)\ >\ f^{\dagger\dagger\dagger}(t)\ >\  f^{\dagger\dagger\dagger\dagger}(t).
\end{align*}
Then $f^\dagger(t)\to +\infty$ as $t\to +\infty$, and for every $n$, $f(t)> f^\dagger(t)^n$, eventually. 
\end{lemma}
\begin{proof} Applying Lemma~\ref{ha1} to $f^\dagger$ in the role of $f$ gives $f^\dagger(t)\to +\infty$ as $t\to +\infty$. Applying Lemma~\ref{ha2} to $f^\dagger$ in the role of $f$ gives $f^\dagger(t)/f^{\dagger\dagger}(t) \to +\infty$ as~${t\to +\infty}$. Let~$n\ge 1$ and take $a_n\ge a$ such that $f^\dagger(t)/n> f^{\dagger\dagger}(t)$ for all $t\ge a_n$. 
So the assumptions of Lemma~\ref{ha2} are satisfied for $a_n$ and
the restriction of
$f^{1/n}$ to~$[a_n,+\infty)$ in the role of $a$ and $f$, hence 
$f(t)^{1/n}/\big(f^\dagger(t)/n\big)\to +\infty$ as $t\to +\infty$, and thus~$f(t)/f^\dagger(t)^n\to +\infty$ as $t\to +\infty$.
\end{proof}

\subsection*{Hardian and overhardian germs} 
Let $y\in \Cc^{<\infty}$. 
Following the terminology of~\cite{S} we say that $y$ is {\bf hardian\/} if $y$ generates a Hardy field $\Q\<y\>$. 

\begin{lemma}\label{lemdaghar} If $y$ is hardian and $y>_{\ex} 0$, $y^\dagger>_{\ex} 0$, then
$y>_{\ex} (y^\dagger)^n$ for all $n$.
\end{lemma}
\begin{proof} Suppose $y$ is hardian and $y>_{\ex} 0$, $y^\dagger>_{\ex} 0$. The case $y\prec 1$ is impossible, since it would give $y^\dagger <_{\ex} 0$. If $y\asymp 1$, then $y^\dagger\prec 1$, and we are done. If $y\succ 1$ and~$y^\dagger\prec 1$, we are done.
If $y\succ 1$ and~$y^\dagger \succeq 1$, then $v(y^\dagger)=o(vy)$ by  [ADH, 9.2.10], which gives the desired conclusion.  
\end{proof}

\noindent
We set $y^{\<0\>}:=y$, and inductively,
if $y^{\<i\>}\in \Cc^{<\infty}$ is defined and  $y^{\<i\>}\in (\Cc^{<\infty})^\times$ (so either $y^{\<i\>}<_{\ex} 0$
or $y^{\<i\>}>_{\ex} 0$), then
$y^{\<i+1\>}:=(y^{\<i\>})^\dagger$, and otherwise $y^{\<i+1\>}$ is not defined.  
As in \cite{S} we call $y$ {\bf overhardian\/} if for all $i$,
$$y^{\<i\>} \text{ is defined},\quad y^{\<i\>}>_{\ex} 0, \text{ and } y^{\<i\>} >_{\ex} y^{\<i+1\>}.$$ 
If $y$ is overhardian, then so is $y^\dagger$. By Lemma~\ref{ha3}:

\begin{cor}\label{overh1} If $y$ is overhardian, then for all $i$, $n$ we have 
$$y^{\<i\>}>_{\ex}\R, \quad y^{\<i\>}>_{\ex} (y^{\<i+1\>})^n.$$
\end{cor} 

\noindent
Next we recall from [ADH, 4.3]  that
a differential polynomial $P(Y)\in K\{Y\}$ over a differential field $K$ has a unique logarithmic decomposition 
$$P(Y)\ =\ \sum_{\i}P_{\<\i\>} Y^{\<\i\>}\qquad (P_{\<\i\>} \in K).$$ 
If $K$ is a Hardy field and $y^{\<i\>}$ is defined for all $i$, then we can substitute $y$ for the indeterminate $Y$ to get
$P(y)=\sum_{\i}P_{\<i\>} y^{\<\i\>}$ in $\Cc^{<\infty}$, where of course
$$y^{\<\i\>}:=(y^{\<0\>})^{i_0}\cdots (y^{\<r\>})^{i_r}\quad\text{for $\i=(i_0,\dots, i_r)\in \N^{1+r}$.}$$ 
Such a substitution is in particular possible if $y$ is overhardian. Thus for over\-hardian~$y$ and $P\in \R\{Y\}^{\ne}$ we obtain 
$P(y)\in (\Cc^{<\infty})^\times$ from Corollary~\ref{overh1}. Therefore:

\begin{cor}\label{overh2} If $y$ is overhardian, then $y$ is hardian.
\end{cor}

\begin{lemma}\label{lemoverh} If $y$ is overhardian, then $\log y\prec y^\dagger$.
\end{lemma}
\begin{proof} More generally, let $y$ be hardian, $y>_{\ex} \R$,
$y^\dagger > _e \R$, and $y^{\dagger\dagger}>_{\ex} \R$; we claim that then
$\log y\prec y^\dagger$. To prove this, take a Liouville closed
Hardy field $H\supseteq \R$ with~$y\in H$. Applying 
 [ADH, 9.2.18]
to $\alpha=vy$ in the asymptotic couple $(\Gamma,\psi)$ of~$H$
gives~$\log y \asymp y^\dagger/y^{\dagger\dagger} \prec y^\dagger$. 
\end{proof}

\noindent
Given a Hardy field $H$, we say that a germ $y\in\Cc$ is {\bf $H$-hardian} if
$y$ is contained in a Hardy field extension of $H$; see also \cite[Section~4]{ADH5}.

\begin{cor}\label{overh3} Suppose $H\supseteq\R$ is a Liouville closed Hardy field and
$y>_{\ex} H$. Then the following are equivalent: \begin{enumerate}
\item[(i)] $y$ is overhardian;
\item[(ii)] $y$ is $H$-hardian;  
\item[(iii)] $y$ is hardian.
\end{enumerate}
\end{cor}
\begin{proof} From $\exp(H)\subseteq H$ and 
$y>_{\ex} H$ we obtain $\log y>_{\ex} H$.
If $y$ is overhardian, this gives by induction on $n$
and Lemma~\ref{lemoverh} that $y^{\<n\>}>_{\ex} H$ for all $n$, and
so~${P(y)<_{\ex} H}$ or $P(y)>_{\ex} H$ for all $P(Y)\in H\{Y\}\setminus H$, hence
$y$ is $H$-hardian. This proves (i)~$\Rightarrow$~(ii), and (ii)~$\Rightarrow$~(iii) is trivial. To show (iii)~$\Rightarrow$~(i), assume~(iii). From $\log y >_{\ex} H$ and $\exp (x^2)\in H$ we obtain $\log y \succ \exp(x^2)$. Working in  a Hardy field
containing $y$, $\log y$, $x$, and $\exp(x^2)$, we have $(\log y)'\succ \exp(x^2)'$, so  $$v(y^\dagger)\  <\  v\big(x\exp(x^2)\big)\  <\  v\big(\!\exp(x^2)\big)\ <\ 0,$$
hence $v(y^{\dagger\dagger}) \le v\big(\!\exp(x^2)^\dagger\big)=v(x)<0$,
and thus $y^{\dagger\dagger}>_{\ex} \R$. Hence $y^\dagger >_{\ex}  \log y >_{\ex} H$ by the proof of Lemma~\ref{lemoverh}. Since $y^\dagger$ is hardian, we can iterate this argument, which by induction shows that all $y^{\<n\>}$
are defined and $>_{\ex} H$. This yields (i) in view of Lemma~\ref{lemdaghar}.
\end{proof} 

\noindent
Corollary~\ref{overh3} combines \cite[Theorems~3 and 4]{S};
the implication (iii)~$\Rightarrow$~(ii)  also follows from \cite[Theorem~12.23]{Boshernitzan82}. 

\begin{cor}\label{overh4} If $y$ is overhardian, then so is $\log y$. Moreover,
$$  \text{$y$ is overhardian}\ \Longleftrightarrow\  \text{$y$ is hardian and $y>_{\ex} \exp_n(x)$ for all $n$.}$$ 
\end{cor}
\begin{proof} Suppose $y$ is overhardian. Then $y$ is hardian and hence so is $\log y$. Moreover, $\log y> _e \R$, and $\log y \asymp y^\dagger/y^{\dagger\dagger}$ by the proof of Lemma~\ref{lemoverh}. Hence
$$(\log y)^\dagger\  \sim\  \big(y^\dagger/y^{\dagger\dagger}\big)^\dagger\ =\ y^{\dagger\dagger} - y^{\dagger\dagger\dagger}\ \sim\ y^{\dagger\dagger}\ =\ y^{\<2\>}.$$ Now an easy induction shows that all
$(\log y)^{\<n\>}$ are defined, and that for $n\ge 1$ we have $(\log y)^{\<n\>}\sim y^{\<n+1\>}$. This proves the first claim of the corollary. Also $x\prec y$, since~$1\prec y \preceq x$ would give $y^\dagger\preceq x^\dagger=1/x$, contradicting $y^\dagger>_{\ex}\R$. Applying this to $\log_n y $ (which we now know to be overhardian), gives
$\log_n y >_{\ex} x$, and thus~$y >_{\ex} \exp_n(x)$, proving the direction $\Rightarrow$ of the equivalence.  

For the converse, assume $y$ is hardian and $y>_{\ex}\exp_n(x)$ for all $n$. 
Then $H:= \Li\!\big(\R(x)\big)$, the Liouville closure of $\R(x)$ as a Hardy field,  embeds as an $H$-field over $\R(x)$ into the Liouville closed
$H$-field extension $\T$ of $\R(x)$. Since the se\-quence~$\big(\!\exp_n(x)\big)$ is cofinal in $\T$, this is also the case in $H$, so $y>_{\ex} H$, and 
hence~$y$ is overhardian by Corollary~\ref{overh3}.
\end{proof}

\subsection*{Constructing overhardian germs} Our goal is the following:

\begin{theorem}\label{ha8} For any $\phi\in \Cc$ there is an overhardian
$y\in \Cc^{\infty}$ such that~${y^{\<m\>}>_{\ex} \phi}$ for all $m$. 
\end{theorem}

\noindent
Note that Theorem~\ref{sjo} follows from Corollary~\ref{overh3} and Theorem~\ref{ha8}. 
To get an idea of how to construct a $y$ as in Theorem~\ref{ha8}, consider an
 overhardian $y$ represented by a function
in $\Cc_a^{\infty}$, to be denoted also by $y$. Then we have a strictly increasing sequence $(a_m)$ of real numbers $\ge a$ tending to $+\infty$ such that 
$y^{\<m\>}(t)$ is defined for $t\ge a_m$, for every $m$, and thus
$$y^{\<m-1\>}(t)\ =\ y^{\<m-1\>}(a_{m})\cdot \exp \int_{a_{m}}^t y^{\<m\>}(s)\, ds\quad \text{for $m\ge 1$, $t\ge a_{m}$.}$$
It follows that  $y$ is determined as a function
on $[a_0,+\infty)$ by the family of restrictions~$\big(y^{\<m\>}|_{[a_m, a_{m+1}]}\big)$: $y$ on $[a_0, a_1]$ and
$y^{\<1\>}$ on $[a_1,a_2]$ determine $y$ on $[a_0, a_2]$; likewise, $y^{\<1\>}$ on $[a_1,a_2]$ and $y^{\<2\>}$ on $[a_2,a_3]$ determine $y^{\<1\>}$ on $[a_1, a_3]$, and thus~$y$ on~$[a_0, a_3]$, and so on. 
We use this as a clue to reverse engineer overhardian elements.  

We start with $a\in \R$ and a strictly increasing sequence $(a_m)$ in $\R^{\ge a}$ tending to~$+\infty$ and for each $m\ge 1$ a continuous function
$y_{m-1,m}\colon [a_{m-1}, a_{m}]\to \R$. Let~${m\ge 1}$. We define the continuous
function $y_{k,m}\colon [a_k, a_{m}]\to \R$ for $0\le k < m$ by downward
recursion on $k$: $y_{m-1,m}$ is already given to us, and for $1\le k < m$,
$$ y_{k-1,m}(t) \ :=\ \begin{cases} y_{k-1,k}(t) & \text{ for $a_{k-1}\le t\le a_k$,}\\
y_{k-1,k}(a_k)\cdot\exp \displaystyle\int_{a_k}^t y_{k,m}(s)\, ds &\text{ for $a_k \le t\le a_m$.}
\end{cases}$$
(See Figure~\ref{fig:ykm}.)

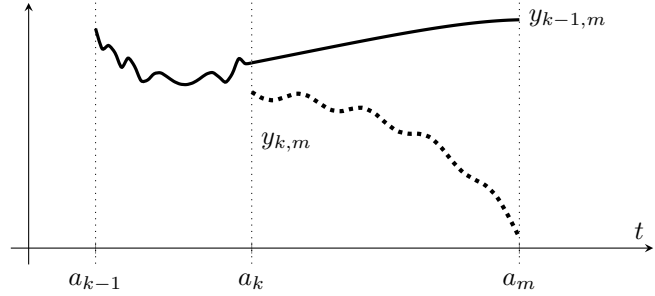
\begin{figure}[ht]
\begin{tikzpicture}
\def\akmone{0.75};\def\ak{2.5};\def\am{5.5};
  \begin{axis} [axis lines=center, xmin=-0.2, xmax=7, ymin = -0.01, ymax = 0.15, width=0.8\textwidth, height = 0.4\textwidth, xlabel={$t$}, xtick={\akmone,\ak,\am}, ytick=\empty, xticklabels={\strut $a_{k-1}$, \strut $a_k$, \strut $a_{m}$}]
    \draw[dotted] (axis cs: \akmone,-0.25) -- (axis cs: \akmone,2);
    \draw[dotted] (axis cs: \ak,-0.25) -- (axis cs: \ak,2);
    \draw[dotted] (axis cs: \am,-0.25) -- (axis cs: \am,2);

   	\addplot [domain=\akmone:\ak, smooth, very thick] { sin((x-1.75)^2*1200) * 0.005 + 0.1 + (x-1.75)^2*0.03};

   	\addplot [domain=\ak:\am, smooth, ultra thick, dotted] { ( (-exp(x*0.93)+1) + 6*sin(x*500) ) * 0.00055 + 0.1};
	\def\ykmonem { ( (-exp(x*0.93)/0.93+x) - (6*cos(x*500)/500) ) * 0.00055 + (0.1*x)  }; 
	\def\ykmonematak {  ( (-exp(\ak*0.93)/0.93+\ak) - (6*cos(\ak*500)/500) ) * 0.00055 + (0.1*\ak)  };
   	\addplot [domain=\ak:\am, smooth, very thick] { (sin((\ak-1.75)^2*1200) * 0.005 + 0.1 + (\ak-1.75)^2*0.03) * exp( (\ykmonem)  - (\ykmonematak) ) };

    \node[right] at (axis cs:\am,0.14) {$y_{k-1,m}$};
        \node[right] at (axis cs:\ak,0.065) {$y_{k,m}$};

  \end{axis}
\end{tikzpicture}
\caption{Passing from $y_{k,m}$ to $y_{k-1,m}$}\label{fig:ykm}
\end{figure}

\medskip
\noindent
Downward induction on $k$ gives $y_{k,m}=y_{k,m+1}$ on $[a_k, a_m]$ for~${k < m}$. This fact gives for each
$k$ a continuous function $y_k\colon [a_k, +\infty)\to \R$
such that $y_k=y_{k,m}$ on~$[a_k,a_m]$, for all $m>k$. Thus
for $k\ge 1$ we have 
$$ y_{k-1}(t)\ =\ y_{k-1}(a_{k})\cdot \exp \int_{a_{k}}^t y_{k}(s)\,ds\quad \text{for $t\ge a_{k}$.}$$
In the next lemma we use the notation $E(t)^{(r)}_{t=a}$ where the expression $E(t)$ defines a function
$t\mapsto E(t)$ in $\Cc^r(I)$, where $I=[b,c]$ ($b<c$ in $\R$) and $a\in I$. With $f$ this function,   
 $E(t)^{(r)}_{t=a}:= f^{(r)}(a)$. In connection with (ii) in that lemma we note that for~$r\in \N^{\ge 1}$ and
  $g\in \Cc^{r-1} [b,c]$ ($b<c$ in $\R$) and setting $G(t):=\int_b^t g(s)\, ds$ for~${t\in [b,c]}$ we have
  $(\exp G)^{(r)} = A_r\big(g,\dots, g^{(r-1)}\big)\cdot(\exp G)$ with $A_r\in \Z[X_0,\dots, X_{r-1}]$ depending only on $r$,
  and thus $\big(\!\exp G(t)\big)^{(r)}_{t=b}=A_r\big(g(b),\dots, g^{(r-1)}(b)\big)$.

\begin{lemma}\label{ha9} Assume the following holds for all $k\ge 1$:
\begin{enumerate}
\item[(i)] $y_{k-1,k}\in \Cc^{\infty}[a_{k-1},a_k]$ and $y_{k-1,k}(t)>0$ for $a_{k-1}\le t\le a_k$;
\item[(ii)]  $y_{k-1,k}^{(r)}(a_k)=y_{k-1,k}(a_k)\cdot\big(\!\exp \int_{a_k}^t y_{k,k+1}(s)\, ds\big)^{(r)}_{t=a_k}$ for all $r\in \N^{\ge 1}$.
\end{enumerate}
Then for all $k$ we have $y_k\in \Cc_{a_k}^{\infty}$, $y_k(t)>0$ for $t\ge a_k$, and $y_k^\dagger=y_{k+1}$ on $[a_{k+1},+\infty)$. Thus $y_0$ is overhardian if $y_k>_{\ex} y_{k+1}$ for all $k$. 
\end{lemma}
\begin{proof} Downward induction on $k$ shows that $y_{k,m}$ for $k<m$ has the corresponding properties. Note:  $y_k^{\<m\>}$ is defined in $\Cc^{<\infty}$ and equals $y_{k+m}$ in $\Cc^{<\infty}$ for all $k$, $m$. 
\end{proof} 

\noindent 
Towards proving Theorem~\ref{ha8} we may assume $\phi\in \Cc$ to be represented by a continuous function $\phi\colon [a,+\infty)\to \R^{>}$, so $\phi$ denotes the function and its germ.  

\begin{lemma}\label{ha10} There exists an increasing $\Cc^{\infty}$-function $f\colon [a,+\infty)\to \R$ such that~$\phi(t) < f(t)$ and 
$f(t) > f^{\dagger}(t)$ for all $t\ge a$. 
\end{lemma}
\begin{proof} Lemma~\ref{thetazeta} yields a decreasing $\Cc^{\infty}$-function
$\zeta\colon [a,+\infty)\to \R^{>}$ with~$1/\phi(t) > \zeta(t)$
and $\zeta'(t)> -1$ for all $t\ge a$. 
Then $f:= 1/\zeta$ works.
\end{proof}

\noindent
Replacing $\phi$ by $f$ and renaming,
we arrange that
$\phi\colon [a,+\infty)\to \R^{>}$ is increasing of class $\Cc^{\infty}$ and $\phi(t) > \phi^\dagger(t)$ for all $t \ge a$. With these assumptions:

\begin{lemma}\label{ha11} Suppose for all $k\ge 1$ we have $y_{k-1,k}(t)> \phi(t)$ for $a_{k-1}< t\le a_k$.
 Then for all $k$ we have $y_k(t) > \phi(t)$ for $t > a_k$. 
\end{lemma}
\begin{proof} Let $1\le k < m$, and assume as an inductive assumption that $y_{k,m}(t) > \phi(t)$ for $a_k < t \le a_m$. Our job is to show that then $y_{k-1,m}(t) > \phi(t)$ for $a_{k-1} < t \le a_m$, and this amounts to showing for $a_k \le t \le a_m$ that 
$$y_{k-1,k}(a_{k})\cdot \exp \int_{a_{k}}^t y_{k,m}(s)\, ds\ >\ \phi(t).$$ 
This holds for $t=a_k$, and for $a_k< t \le a_m$ we have
\begin{align*} y_{k-1,k}(a_{k})\cdot \exp \int_{a_{k}}^t y_{k,m}(s)\, ds\ &>\ \phi(a_k)\cdot \exp\int_{a_k}^t \phi(s)\,ds \\
&>\ \phi(a_k)\cdot \exp\int_{a_k}^t \phi^\dagger(s)\,ds\\
&=\ \phi(a_k)\cdot \exp\! \big(\log \phi(t)-\log \phi(a_k)\big)\ =\ \phi(t),
\end{align*}
which gives the desired result.
\end{proof} 

\noindent
For $b\ge a$ we define the $\Cc^{\infty}$-function $\phi_b\colon [a,+\infty)\to \R^{>}$ by
\begin{equation}\label{eq:phib}
\phi_b(t)\ =\ \phi(b)\cdot \exp \int_b^t \phi(s)\,ds,
\end{equation}
so $\phi(t) < \phi_b(t)$ for $t>b$, using again that $\phi(s)> \phi^\dagger(s)$ for $s>b$. 

\begin{lemma}\label{ha12} Suppose that for all $k\ge 1$ we have $\phi< y_{k-1,k} \le  \phi_{a_{k-1}}$ on $(a_{k-1}, a_k]$.
Then for $k+1 < m$ we have $y_{k,m} > y_{k+1,m}$ on $[a_{k+1},a_m]$.
\end{lemma}
\begin{proof} For $m=k+2$ and $a_{k+1}\le t \le a_m$ we have
\begin{align*} y_{k,m}(t)\ &=\ y_{k,k+1}(a_{k+1})\cdot\exp\int_{a_{k+1}}^t y_{k+1,m}(s)\,ds\\
&>\ \phi(a_{k+1})\exp \int_{a_{k+1}}^t \phi(s)\,ds\ =\ \phi_{a_{k+1}}(t)\ \ge\ y_{k+1,m}(t).
\end{align*} 
Let $1\le k < k+1 < m$ and assume inductively that $y_{k,m}(t) > y_{k+1,m}(t)$ whenever~$a_{k+1}\le t\le a_m$. Then for $a_k\le t\le a_{k+1}$ the special case above yields
$$y_{k-1,m}(t)\ =\ y_{k-1,k+1}(t)\ >\ y_{k,k+1}(t)\ =\ y_{k,m}(t),$$
and for $a_{k+1}\le t\le a_m$ the inductive assumption gives
\begin{align*} y_{k-1,m}(t)\ &=\ y_{k-1,k}(a_k)\cdot \exp\int _{a_k}^{t}y_{k,m}(s)\,ds\\
&=\ y_{k-1,k}(a_k)\cdot \exp\int _{a_k}^{a_{k+1}}y_{k,m}(s)\,ds\cdot \exp\int_{a_{k+1}}^t y_{k,m}(s)\,ds\\    &=\
y_{k-1,m}(a_{k+1})\cdot \exp \int_{a_{k+1}}^ty_{k,m}(s)\,ds\\ 
&>\ y_{k,m}(a_{k+1})\exp \int_{a_{k+1}}^ty_{k+1,m}(s)\,ds\ =\ y_{k,m}(t),
\end{align*}
which concludes the induction. 
\end{proof}

\begin{cor}\label{exyk} Suppose that for all $k\ge 1$ we have
\begin{enumerate}
\item[(i)] $y_{k-1,k}\in \Cc^{\infty}[a_{k-1},a_k]$;
\item[(ii)] $\phi < y_{k-1,k} \le \phi_{a_{k-1}}$ on $(a_{k-1}, a_k]$;
\item[(iii)] $y_{k-1,k}^{(r)}(a_{k-1})\ =\ \phi^{(r)}(a_{k-1})$ for all $r\in \N$;
\item[(iv)] $y_{k-1,k}^{(r)}(a_k)= y_{k-1,k}(a_k)\cdot\big(\!\exp \int_{a_k}^t \phi(s)\, ds\big)^{(r)}_{t=a_k}$ for all $r\in \N^{\ge 1}$. 
\end{enumerate}
Then $y:= y_0\in \Cc_{a_0}^{\infty}$, $y$ is overhardian, and $y^{\<k\>}>_{\ex} \phi$ for all $k$. 
\end{cor}
\begin{proof}  By (i), (iii), (iv), and the remark preceding Lemma~\ref{ha9}:
$$
y^{(r)}_{k-1,k}(a_k)\ =\   y_{k-1,k}(a_k)\cdot  \big(\!\exp\textstyle\int_{a_k}^t y_{k,k+1}(s)\, ds\big)^{(r)}_{t=a_k} \qquad (k\ge 1, r\in \N^{\ge 1}). $$
Then Lemma~\ref{ha9} yields $y_k\in \Cc_{a_k}^{\infty}$, and  $y_k>0$, $y_k^\dagger=y_{k+1}$ on $[a_{k+1},+\infty)$.
Also~${y_k>\phi}$ on $(a_{k},+\infty)$
by Lemma~\ref{ha11}, and $y_k>y_{k+1}$ on $[a_{k+1},+\infty)$
by (ii) and Lem\-ma~\ref{ha12}. It remains to appeal to the last sentence of Lemma~\ref{ha9}.
\end{proof}

\noindent
The phrase ``$y$ is overhardian'' in the corollary above is short for ``the germ of~$y$ at~$+\infty$ is overhardian''. Given the strictly increasing sequence $(a_k)$ of
real numbers~$\ge a$ tending to $+\infty$ and the increasing $\Cc^{\infty}$-function $\phi\colon [a,+\infty) \to \R^{>}$ such that~$\phi(t) > \phi^\dagger(t)$ for all $t\ge a$, it follows from Lemma~\ref{phizeta} that there exist functions~$y_{k-1,k}$ for~$k\ge 1$ satisfying conditions
(i)--(iv) of Corollary~\ref{exyk}, where each value $y_{k-1,k}(a_k)$ can be chosen arbitrarily in the interval $\big(\phi(a_k), \phi_{a_{k-1}}(a_k)\big)$. Now the conclusion of that corollary yields Theorem~\ref{ha8}, and thus Theorem~\ref{sjo}.

\section{Filling Wide Gaps}\label{fwg}

\noindent
We now adapt the material from the previous section to filling a {\em wide gap}.
To describe this situation, let $H\supseteq \R$ be a Liouville closed Hardy field. By a {\bf wide gap} in $H$
we mean a pair $A$, $B$ of nonempty subsets of $H^{>\R}$ such that $A< B$, 
there is no $h\in H$ with $A <  h < B$, and $A$ and $\exp A$ are cofinal; note that then~$A$ and~$\log A$ are  cofinal,  that $B$, $\exp B$, $\log B$ are coinitial, and 
that for any~$\phi\in \Cc$ with~$A<_{\ex} \phi <_{\ex} B$ we  have $A<_{\ex} \log \phi, \exp \phi <_{\ex} B$. 
Moreover, if $A,B$ is a wide gap in $H$, then it is  an additive gap in $H$, and
$A$, $\operatorname{sq}(A)$ are cofinal, and $B$, $2B$, $\sqrt{B}$ are coinitial, by Corollary~\ref{cor:mult gap} and Lemmas~\ref{lem:add gap 1}, \ref{lem:mult->add}, and~\ref{in2}(i). 
Let us also record the following, although we shall not explicitly use it:

\begin{lemma} Let $H\supseteq \R$ be a Liouville closed Hardy field and $A$, $B$ nonempty subsets of $H^{>\R}$ such that $A< B$ and there is no $h\in H$ with $A< h< B$. Suppose there exists $\phi\in \mathcal{C}$ such that $A <_{\ex}  \phi$ and $\ex^\phi <_{\ex} B$. Then $A$, $B$ is a wide gap in $H$.
\end{lemma} 
\begin{proof} For $\phi$ as above and $h\in A$ we have
$h<_{\ex} \phi$, so  $\ex^h<_{\ex} \ex^\phi<_{\ex} B$, and thus~$\ex^h\le f$ for some $f\in A$. 
\end{proof}

\noindent
Here is the main result of this section:

\begin{theorem}\label{sjo+} If $H\supseteq \R$ is a Liouville closed Hardy field and $A$,~$B$ is a wide gap in $H$ with $\cf(A)=\ci(B)=\omega$, then some
$y\in \Cc^{<\infty}$ 
with $A <_{\ex} y <_{\ex} B$ is $H$-hardian. 
\end{theorem}

\noindent
 Wide gaps as in Theorem~\ref{sjo+} do actually occur, as we show in the next subsection.
Towards proving Theorem~\ref{sjo+} and some variants  we begin with a result that is mainly an exercise in valuation theory:

\begin{lemma} \label{val} Let $H\supseteq \R$ be a Liouville closed Hardy field, let $A$, $B$ be a wide gap in $H$, and let $y\in \Cc^{<\infty}$ be overhardian with $A<_{\ex} y<_{\ex} B$. Then $y$ is $H$-hardian and   $\d$-transcendental over $H$. 
\end{lemma}
\begin{proof}  It will be convenient to work with the $y^{\<n\>}$. Note that
Lemma~\ref{lemoverh} and the cofinality of $A$ and $\exp(A)$ give $A <_{\ex} \log y <_{\ex} y^\dagger <_{\ex} y<_{\ex} B$. Using this inductively we obtain $A<_{\ex} y^{\<i\>} <_{\ex} B$ for all $i$. We  prove by induction on $n$ the claim that~$y, y',\dots, y^{(n)}$ generate a Hausdorff
field extension $H_n\ :=\ H(y,y',\dots, y^{(n)})$ of~$H$. 
For $n=0$ this claim follows by applying Lemma~\ref{lem:5.1.18}  to $$P\ :=\ \big\{vh:\,\text{$h\in H^{>}$, $h\preceq g$ for some $g\in A$}\big\}.$$ 
Assume the claim holds for a certain $n$.  It is easy to check that then $y^{\<0\>},\dots,y^{\<n\>}$ lie in $H_n$, that $H_n\ =\ H\big(y^{\<0\>},\dots,y^{\<n\>}\big)$, and
that $H_n$ has value group
$$v(H_n^\times)\ =\ v(H^\times) \oplus \Z vy^{\<0\>} \oplus\cdots \oplus \Z vy^{\<n\>}$$
with $vB < vy^{\<i\>} < vA$ for all $i\le n$ and $vy^{\<i+1\>} =o(vy^{\<i\>})$ for all $i<n$.  Note that~$vA<0$. Let $\Delta$ be the smallest convex subgroup of $v(H^\times)$
that includes $vA$. Then $vA$ is coinitial in $\Delta$, and $\Delta + \Z vy^{\<0\>} +\cdots + \Z vy^{\<n\>}$ is a convex subgroup of~$v(H_n^\times)$ with
$vB < \Delta+ \Z vy^{\<0\>} +\cdots + \Z vy^{\<n\>}$. Hence the real closure $H_n^{\rc}$ of $H_n$, taken as
a Hausdorff field extension of $H_n$, has value group   
$$v\big(H_n^{\rc,\times}\big)\ =\ v(H^\times) \oplus \Q vy^{\<0\>} \oplus\cdots \oplus \Q vy^{\<n\>}, $$
and $\Delta$ as well as $\Delta+\Q vy^{\<0\>} +\cdots + \Q vy^{\<n\>}$ are convex subgroups of $v(H_n^{\rc,\times})$ with~$vB <\Delta+ \Q vy^{\<0\>} +\cdots + \Q vy^{\<n\>}$. (See Figure~\ref{fig:val}.) In view of  $A <_{\ex} (y^{\<n+1\>})^i <_{\ex} y^{\<n\>}$ for all~$i\ge 1$ it now follows from Lemma~\ref{lem:5.1.18} (with $H_n^{\rc}$ in the role of $H$)
 that~$y^{(n+1)}$  generates a Hausdorff field over $H_n^{\rc}$. 
 
 The structure of the value group of $H_n$ yields that $y$ is $\d$-transcendental over $H$ by the Zariski-Abhyankar Inequality~[ADH, 3.1.11]. 
\end{proof}

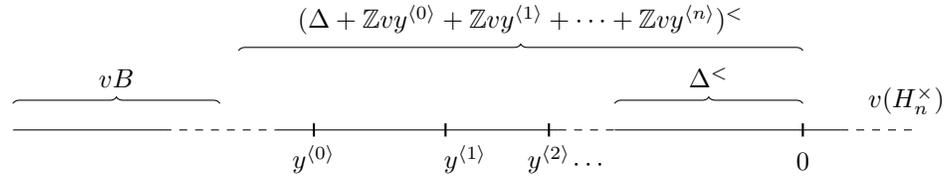
\begin{figure}[ht]
\begin{tikzpicture}

\draw[dashed] (2,0) -- (3.5,0);
 
\draw  (3.5,0) -- (7.25,0);
\draw[dashed] (7.25,0) -- (8,0);

\draw  (8,0) -- (11,0);

\draw (0,0) -- (2,0);
\draw[dashed] (11,0) -- (12,0);

\def\ypt{10.5};

\node at (\ypt+0.5, 0) [anchor = north,shift={(2.5em,2 em)}]{$v(H_n^\times)$};

\draw[thick] (\ypt, 0.25em)--(\ypt, -0.25em);
\node at (\ypt, 0) [anchor = north,shift={(0em,-.5 em)}]{$0$};

\draw[thick] (\ypt-3.375, 0.25em)--(\ypt-3.375, -0.25em);
\node at (\ypt-3.375, 0) [anchor = north,shift={(0em,-.2 em)}]{$y^{\<2\>}$};

\draw[thick] (\ypt-4.75, 0.25em)--(\ypt-4.75, -0.25em);
\node at (\ypt-4.75, 0) [anchor = north,shift={(0.75em,-.2 em)}]{$y^{\<1\>}$};

\draw[thick] (\ypt-6.5, 0.25em)--(\ypt-6.5, -0.25em);
\node at (\ypt-6.5, 0) [anchor = north,shift={(0em,-.2 em)}]{$y^{\<0\>}$};

\node at (\ypt-3, 0) [anchor = north,shift={(0.5em,-.75 em)}]{$\cdots$};

\draw[decoration={brace,raise=1em},decorate] (0,0) -- node [above=1.25em] {$vB$} (2.75,0);
\draw[decoration={brace,raise=1em},decorate] (8,0) -- node [above=1.25em] {$\Delta^<$} (10.5,0);
  
\draw[decoration={brace,raise=3em},decorate] (3,0) -- node [above=3.25em] {$(\Delta+\Z vy^{\<0\>}+\Z vy^{\<1\>}+\cdots+\Z vy^{\<n\>})^<$} (10.5,0);

\end{tikzpicture}
\caption{Value group of $H_n$}\label{fig:val}
\end{figure}

\noindent
Let us now consider the slightly different situation where $H\supseteq \R$ is a Liouville closed Hardy field and $y\in \Cc^{<\infty}$ is overhardian with
$y>_{\ex} H$. Then the proof of the lemma above goes through for $A:= H^{>\R}$ and $B=\emptyset$, although this pair $A$,~$B$ is not a wide gap. The proof not only shows in this situation that $y$ generates a Hardy field $H\<y\>$, but also that $v(H^\times)$ is a convex subgroup of $v\big(H\<y\>^\times\big)$, and that~$v\big(H\<y\>^\times\big)=v(H^\times) \oplus \bigoplus_i \Z vy^{\<i\>}$, and so   $y$ is $\d$-transcendental over $H$. (See also [ADH, 16.6.10].)

\subsection*{Constructing ``countable'' wide gaps} The Liouville closed Hardy field $\Li(\R)=\Li\!\big(\R(x)\big)$ is $\d$-algebraic over $\R$. Hence by \cite[Theorem~3.4]{AD} the sequence $\big(\!\exp_n(x)\big)$
is cofinal in $\Li(\R)$, so $\cf\!\big(\!\Li(\R)\big)=\omega$. More generally, let $H\supseteq \R$ be any Liouville closed Hardy field 
with $\cf(H)=\omega$. 
Then \cite[remarks after Lemma~5.17]{ADH5} yields a $\phi\in \Cc$ with $\phi>_{\ex} H$ and so Theorem~\ref{sjo} gives an $H$-hardian $y\in \Cc^{<\infty}$ such that~$y>_{\ex} H$. We now consider the Hardy-Liouville closure~$\Li\!\big(H\<y\>\big)$ of $H\<y\>$. We have a wide gap $A$,~$B$ in $\Li\!\big(H\<y\>\big)$ given by 
\begin{align*} A\ &:=\ \big\{f\in \Li\!\big(H\<y\>\big):\,  \text{$\R< f < h$  for some $h\in H$}\big\},\\
 B\ &:=\ \big\{g\in \Li\!\big(H\<y\>\big):\, g> H\big\}.
 \end{align*} 
Note that $\cf(H)=\omega$ gives $\cf(A)=\omega$. Moreover: 

\begin{lemma}  $B=\big\{g\in \Li\!\big(H\<y\>\big):\, \text{$g>\log_n y$ for some $n$}\big\}$. 
\end{lemma} 
\begin{proof} By the remarks preceding this subsection $y$ is $\d$-transcendental over $H$ and~$\{y^n:\,n=0,1,2,\dots\}$ is cofinal in $H\<y\>$. Now $\Li\!\big(H\<y\>\big)$ is $\d$-algebraic over~$H\<y\>$, so $\big\{\!\exp_n(y):\,n=0,1,2,\dots\big\}$ is
cofinal in $\Li\!\big(H\<y\>\big)$ by \cite[Theorem~3.4]{AD} applied to~$K=H\<y\>$.  In particular, $\Li\!\big(H\<y\>\big)$  has $\d$-transcendence degree~$1$ over~$H$. Towards a contradiction, suppose $g\in B$ and $g<\log_n y$ for all
$n$. With $g$ instead of~$y$ we conclude that $\big\{\!\exp_n(g):\,n=0,1,2,\dots\big\}$ is
cofinal in $\Li\!\big(H\<g\>\big)$ and $\Li\!\big(H\<g\>\big)$  has $\d$-transcendence degree $1$ over $H$. 
Hence $y>\Li\!\big(H\<g\>\big)$, and so with $\Li\!\big(H\<g\>\big)$ in the role of $H$ we conclude that
$\Li\!\big(H\<y\>\big)=\Li\!\big(\Li\!\big(H\<g\>\big)\<y\>\big)$ has $\d$-transcendence degree $1$ over
$\Li\!\big(H\<g\>\big)$ and thus $\d$-transcendence degree $2$ over $H$,  a contradiction. 
\end{proof}

\noindent
Thus $A$,~$B$ is a wide gap in $\Li\!\big(H\<y\>\big)$ with $\cf(A)=\ci(B)=\omega$.    
See also Figure~\ref{fig:widegap}.

\begin{figure}[ht]
\begin{tikzpicture}

\draw[densely dotted] (-1,0) -- (0,0);

\draw[dashed] (0,0) -- (0.5,0);
\draw (0.5,0) -- (2,0);
\draw[dashed] (2,0) -- (2.5,0);

\draw[dashed] (3.25,0) -- (5,0);

\draw (5,0) -- (10.75,0);
\draw[dashed] (10.75,0) -- (11.5,0);

\def\ypt{10.5};

\draw[thick] (\ypt, 0.25em)--(\ypt, -0.25em);
\node at (\ypt, 0) [anchor = north,shift={(0em,-.5 em)}]{$y$};

\draw[thick] (\ypt-3, 0.25em)--(\ypt-3, -0.25em);
\node at (\ypt-3, 0) [anchor = north,shift={(0em,-.2 em)}]{$\log y$};

\draw[thick] (\ypt-4.5, 0.25em)--(\ypt-4.5, -0.25em);
\node at (\ypt-4.5, 0) [anchor = north,shift={(0.75em,-.2 em)}]{$\log_2 y$};

\draw[thick] (\ypt-5.25, 0.25em)--(\ypt-5.25, -0.25em);
\node at (\ypt-5.25, 0) [anchor = north,shift={(0em,-.2 em)}]{$\log_3 y$};

\node at (\ypt-6.5, 0) [anchor = north,shift={(0em,-.5 em)}]{$\cdots$};

\node at (-0.5, 0) [anchor = north,shift={(0em,-.5 em)}]{$\R$};

\draw[decoration={brace,raise=1em},decorate] (0,0) -- node [above=1.25em] {$A$} (2.5,0);
\draw[decoration={brace,raise=1em},decorate] (3.25,0) -- node [above=1.25em] {$B$} (11.5,0);  
\end{tikzpicture}
\caption{The countable wide gap $A$, $B$}\label{fig:widegap}
\end{figure}
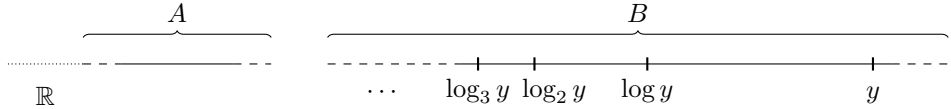

\subsection*{Upper bounds} 
Assume that $a\ge 1$, $\phi\colon [a,+\infty)\to \R^{>}$ is 
$\Cc^{\infty}$ and increasing, and~${\phi > \phi^\dagger}$ on $[a,+\infty)$. Let $(a_m)$ be a strictly increasing sequence of real 
numbers~${\ge a}$ tending to $+\infty$, and let for each $m\ge 1$ a continuous function
$$y_{m-1,m}\colon\ [a_{m-1}, a_{m}]\to \R$$ be given. As in the previous section this gives rise to functions $y_{k,m}$ for $k<m$ and
functions $y_m$ and $y:= y_0$. Finally, assume that $y_{k-1,k} \le \phi_{a_{k-1}}$ on $(a_{k-1}, a_k]$, for all~$k\ge 1$. 
(See \eqref{eq:phib} for the definition of $\phi_b$ for $b\geq a$.)
Our goal is to find an upper bound for $y$ on $[a_0, a_n]$ for $n\ge 1$ that depends only on $\phi$ and $n$, not on the sequence $(a_m)$
or the functions $y_{m-1,m}$.

For $n\ge 1$ and $a_{n-1}\le t \le a_n$, we have
\begin{align*} y_{n-1,n}(t)\ &\le\ \phi_{a_{n-1}}(t)\ =\ \phi(a_{n-1})\exp\int_{a_{n-1}}^t \phi(s)\,ds\\
&\le\  \phi(a_{n-1})\exp\!\big((t-a_{n-1})\phi(t)\big)\ =\ \frac{\phi(a_{n-1})}{\exp\!\big(a_{n-1}\phi(t)\big)}\exp\!\big(t\phi(t)\big)\\
 &\le\ \exp\!\big( t\phi(t)\big). 
\end{align*}
Let $1\le k < n$. Then $y_{k-1,n}(t)\le \exp\!\big(t\phi(t))\big)$ for $a_{k-1}\le t\le a_k$. We assume inductively that for $a_k\le t\le a_n$ we have $y_{k,n}(t)\le \exp_{n-k}\!\big(t\phi(t)+(n-k)t\big)$. 
Then for $a_{k}\le t\le a_n$,
\begin{align*} y_{k-1,n}(t)\ &=\ y_{k-1,k}(a_k)\exp\int_{a_{k}}^t y_{k,n}(s)\,ds\\
 &\le\ y_{k-1,k}(a_k)\exp\!\big[(t-a_{k})\exp_{n-k}
 \big(t\phi(t)+(n-k)t\big)\big]\\
 &\le \frac{y_{k-1,k}(a_k)}{\exp\!\big[a_k\exp_{n-k}\!
 \big(t\phi(t)+(n-k)t\big)\big]}\exp\!\big[t\exp_{n-k}\!\big(t\phi(t)+(n-k)t\big)\big]\\
 &\le\ \exp_{n-(k-1)}\!\big(t\phi(t)+(n-k+1)t\big), 
\end{align*}
where we use that for $t\ge a_k$ we have the inequalities
\begin{align*} y_{k-1,k}(a_k)\ &\le\ 
\exp\!\big(a_k\phi(a_k)\big)\ \le\ \exp\!\big[a_k\exp_{n-k}\!\big(t\phi(t)+(n-k)t\big)\big], \\
 t\exp_{n-k}\!\big(t\phi(t)+(n-k)t\big)\ &\le\  \exp_{n-k}\!\big(t\phi(t)+(n-k+1)t\big),
\end{align*}
the latter being a special case of the easily verified fact that 
$$t\exp_n\!\big(t\phi(t)+nt\big)\ \le\ 
\exp_n\!\big(t\phi(t)+(n+1)t\big) \qquad (n\ge 1,\ t\ge 1).$$ 
We have now proved by downward induction on $k$ that for all $k<n$,
$$y_{k,n}(t)\ \le\ \exp_{n-k}\!\big(t\phi(t)+(n-k)t\big)\qquad \text{ for $a_k\le t\le a_n$.}$$
For $y:= y_0$ this yields 
$$y(t)\ \le\  \exp_n\!\big(t\phi(t)+nt\big)\ \text{ for $n\ge 1$  and $a_0 \le t \le a_n$.} $$
To simplify notation, let $\phi_n\colon [a,+\infty)\to \R$ be the function given by $$\phi_n(t)\ :=\ \exp_n\!\big(t\phi(t)+nt\big),$$ so that that
$\phi < \phi_1 < \phi_2 < \phi_3 < \cdots$ on $[a,+\infty)$, and the bound above takes the form that for all $n\ge 1$ we have
$y\le \phi_n$ on $[a_0,a_n]$.

\subsection*{Back to wide gaps} {\em In the rest of this section $H$ is a Liouville closed Hardy field with~$\R\subseteq H$, and $A$,~$B$ is a wide gap in $H$}. We say that $\phi\in \Cc$ {\bf lies between\/} $A$ and $B$ if~$A<_{\ex} \phi <_{\ex} B$.  By an {\bf intermediary\/} for $A$,~$B$ we mean a $\phi\in \Cc^{\infty}$ lying between~$A$ and~$B$ such that
$0 <_{\ex}\phi^\dagger<_{\ex} \phi$; note that the condition $0 <_{\ex} \phi^\dagger$ implies that $\phi$ is eventually strictly increasing.

\begin{prop}\label{propintermed} Suppose $x\in A$, $\ci(B)=\omega$, and $\phi$ is an intermediary for~$A$,~$B$. Then there exists an overhardian $y\in \Cc^{\infty}$ such that
$\phi<_{\ex} y^{\<n\>} <_{\ex} B$ for all $n$, in particular $A<_{\ex} y <_{\ex} B$, and so $y$ is $H$-hardian, by Lemma~\ref{val}.
\end{prop}
\begin{proof} 
Take a strictly increasing $\Cc^{\infty}$-function~$\phi\colon [a,+\infty)\to \R^{>}$ representing the germ $\phi$
such that $a\ge 1$ and $0 <\phi^\dagger< \phi$ on $[a,+\infty)$. 
For $n\ge 1$, let $\phi_n\colon [a,+\infty) \to \R$ be the function
from the previous subsection given by $$\phi_n(t)\ =\ \exp_n\!\big(t\phi(t)+nt\big).$$  From $x\in A$ and $B$ and $\log B$ being cofinal we obtain $\phi_n<_{\ex} \phi_{n+1} <_{\ex} B$.  
 
Take a strictly decreasing
sequence $g_1 > g_2 > g_3> \cdots$ in $B$, coinitial in $B$. Let~$g_n$ also denote a continuous function $[a,+\infty)\to \R$ representing the germ $g_n$. 
Choose a strictly increasing sequence $b_1 < b_2 < b_3 < \cdots$ of real numbers $\ge a$ tending to~${+\infty}$ such that
$\phi_n< g_n$ and $g_{n+1} < g_n$, on $[b_n, +\infty)$. 
Next, set $a_n:= b_{n+1}$, and choose functions $y_{k-1,k}\in \Cc^{\infty}[a_{k-1},a_k]$ for $k\ge 1$ such that conditions
(i)--(iv) of Corollary~\ref{exyk} are satisfied. (The discussion following that corollary indicates how to construct such functions,
using Lemma~\ref{phizeta}.) 
This yields an overhardian~$y:= y_0\in \Cc_{a_0}^{\infty}$ as in that corollary, with
$y^{\<k\>}>_{\ex} \phi$ for all $k$. 

Let $n\ge 1$. The upper bound from the previous subsection gives
$y\le \phi_n$ on~$[a_0, a_n]$, so
$y< g_n$ on $[b_n, b_{n+1}]$.  With $n+1$ instead of $n$ this gives $y< g_{n+1}$ on~$[b_{n+1}, b_{n+2}]$, and as 
$g_{n+1} < g_{n}$ on $[b_n, +\infty)$, we get $y< g_n$ on $[b_n, b_{n+2}]$. Continuing this way we get $y< g_n$ on $[b_n, b_{n+3}]$, and so on, and thus $y < g_n$ on $[b_n,+\infty)$. Since this holds for all $n\ge 1$, this yields $y<_{\ex}B$. 
\end{proof}

\begin{lemma}\label{intermed} Suppose $x\in A$, and some element of $\Cc$ lies between $A$ and $B$. Then there exists an intermediary for $A$,~$B$.
\end{lemma}
\begin{proof} Let $f\colon [a,+\infty)\to \R^{>}$ be a continuous function whose germ at $+\infty$ lies between $A$ and $B$. Lemma~\ref{smooth} gives
a $\Cc^{\infty}$-function $f^*\colon [a,+\infty)\to \R^{>}$ such that~$f < f^* < f+1$ on $[a,+\infty)$. Then $f^*<_{\ex}B$, and so replacing
$f$ by $f^*$ we have arranged that $f\in \Cc_a^{\infty}$. Defining $F(t):=1 +\int_a^tf(s)\,ds$ we obtain a strictly increasing $F\in \Cc_a^{\infty}$ with $F'=f$. By Lemma~\ref{intineq} we have
$\int A<_{\ex} F <_{\ex} \int B$, and so $A<_{\ex} F<_{\ex} B$ by
Lemma~\ref{in2}(iii) and Lemma~\ref{in3}(ii). Thus we can replace $f$ by $F$ and arrange in this way that $f$ is also strictly increasing and $f\ge 1$. Next, consider the strictly decreasing $\Cc^{\infty}$-function $\theta\colon [a,+\infty) \to (0,1]$ given by 
$$\theta(t)\ :=\ \int_t^{t+1}f(s)^{-1}\,ds\ =\ \int_0^1 f^{-1}(s+t)\,ds, \qquad f^{-1}(s):= f(s)^{-1} \text{ for $s\ge a$.}$$ 

\claim{$\theta'> -1$ on $[a,+\infty)$, and $B^{-1} <_{\ex} \theta <_{\ex} A^{-1}$.}

\smallskip\noindent
That $\theta'>-1$ on $[a,+\infty)$ is clear from 
$$\theta'(t)\ =\ \int_0^1 (f^{-1})'(s+t)\,ds\  =\ f(t+1)^{-1}- f(t)^{-1}.$$
Also $\theta(t)< f(t)^{-1}$ for $t\ge a$, so $\theta <_{\ex}  f^{-1} <_{\ex} A^{-1}$. 

To establish the claim it remains to show that $B^{-1}<_{\ex} \theta$, and this is where we shall need Lemma~\ref{in3}(iii). 
Let $g\in \int B^{-1}$, so $g\in H^{\prec 1}$ and $g'=h^{-1}$ with~${h\in B}$.
We have $h\succ \ex^{x}$, so after increasing $a$ if necessary we can assume that the germ~$h$ is represented
by a continuous function $h\colon [a,+\infty) \to \R$ with~$h(t)> \ex^t$ and thus~${0 < h(t)^{-1} < \ex^{-t}}$, for all $t\ge a$. This yields a $\Cc^1$-function
$$t\mapsto \int_{+\infty}^t h(s)^{-1}\,ds\ :=\ -\int_t^{+\infty}h(s)^{-1}\,ds\ :\,[a,+\infty) \to \R$$
with derivative $h^{-1}$ and tending to $0$ as $t\to +\infty$, so
this function represents the germ $g$, and will be denoted below by $g$. Thus for $t\ge a$,
$$g(t+1)-g(t)\ =\ \int_t^{t+1}h(s)^{-1}\,ds. $$
Moreover, $h(s)^{-1} < f(s)^{-1}$ for all sufficiently large $s\ge a$, and thus 
\begin{equation}\label{eq:g(t+1)-g(t)}
g(t+1)-g(t)\ <\ \theta(t)\quad \text{ for all sufficiently large $t\ge a$.}
\end{equation} 
From $h^{-1}\prec \ex^{-x}$ we get $0 < -g\prec \ex^{-x}$, and so
Lemma~\ref{in1}(iii) applied to $-g$ and combined with \eqref{eq:g(t+1)-g(t)} gives $-g(t)/2 < \theta(t)$ for all sufficiently large
$t\ge a$. In view of Lemma~\ref{in3}(iii) and coinitiality of $B$, $2B$ this yields $B^{-1}<_{\ex} \theta$, as claimed.  

From the claim it follows that the germ of $\phi:= \theta^{-1}\colon [a,+\infty)\to \R$ is an intermediary for~$A$,~$B$.
\end{proof}

\begin{cor}\label{xo} Suppose $x\in A$ and $\cf(A)=\ci(B)=\omega$. Then
there exists an over\-hardian $y\in \Cc^{\infty}$ with 
$A<_{\ex} y  <_{\ex} B$, thus generating a Hardy field over $H$.
\end{cor}
\begin{proof} Using $\cf(A)=\ci(B)=\omega$, Lemmas~\ref{sum1} and~\ref{sum2} give an
element of $\Cc$ that lies between $A$ and $B$. Then Lemma~\ref{intermed} provides an intermediary for~$A$,~$B$, which in view of Proposition~\ref{propintermed} gives the desired result.
\end{proof}

\subsection*{Proof of Theorem~\ref{sjo+}} We assume $\cf(A)=\ci(B)=\omega$. Our job is to obtain a $y\in \Cc^{<\infty}$ such that $A <_{\ex} y <_{\ex} B$ and $y$ generates a Hardy field over $H$. Take any $g\in A$. Then $g>\R$, so $g'$ is active in $H$, and we pass to the compositional conjugate $H\circ g^{\text{inv}}$, which is again a Liouville closed Hardy field containing $\R$ as a subfield, and having $A\circ g^{\inv}, B\circ g^{\inv}$ as a wide gap with $x=g\circ g^{\inv}\in A\circ g^{\inv}$. Now Corollary~\ref{xo} yields a $y\in \Cc^{\infty}$ such that
$A\circ g^{\inv} <_{\ex} y <_{\ex} B\circ g^{\inv}$ and $y$ generates a Hardy field over $H\circ g^{\inv}$. It follows that
$y\circ g\in \Cc^{<\infty}$, $A<_{\ex} y\circ g <_{\ex} B$, and
$y\circ g$ generates a Hardy field over $H$. 
This concludes the proof. \qed

\medskip\noindent 
If $A$ in Theorem~\ref{sjo+} contains an element of $\Cc^{\infty}$, then we can take $y$ in the conclusion of that theorem to be in $\Cc^{\infty}$ as well: in the proof, take $g\in \Cc ^{\infty}$.

\section{The Number of Maximal Hardy Fields}\label{secnmh}

\noindent
Since $\Cc$ has cardinality $\frak{c}=2^{\aleph_0}$, the number of Hardy fields (and thus of maximal Hardy fields) is at most $2^{\frak{c}}$.  By
Proposition~3.7 in \cite{Boshernitzan87} there are $\geq\frak c$ many maximal Hardy fields. 
In this short section we show:

\begin{theorem}\label{th2c} The number of maximal Hardy fields is equal to $2^{\frak{c}}$. 
\end{theorem}

\noindent
This is mainly an application of the previous two sections.  Let  $S$ be an ordered set. Define a {\bf countable gap\/} in $S$ to be a pair~$P$,~$Q$ of countable subsets of $S$ such that $P < Q$ and there is no $s\in S$ with $P<s<Q$; for example, if $P$ is a countable cofinal subset of $S$, then~$P$,~$\emptyset$ is a countable gap in $S$. Also, $S$ is $\eta_1$ iff it has no countable gap. 
We thank Ilijas Farah for pointing out that the following well-known lemma might be useful in proving statements like 
Theorem~\ref{th2c} via a suitable binary tree construction: 

\begin{lemma}\label{lmif} If $S$ has cardinality $<\frak{c}$, then $S$ has a countable gap.
\end{lemma} 
\begin{proof} Suppose $S$ has no countable gap. Then $S$ is in particular dense: for any~${p < q}$ in $S$ there is an $s\in S$ with $p<s<q$. Thus we can embed
 the ordered set $(\Q;{<})$ of rational numbers into $S$. Identifying $\Q$ with its image under such an embedding, there is for every $r\in \R\setminus \Q$
 an $s\in S$ such that for all $t\in \Q$:  $s>t$ in $S$ iff $r>t$ in~$\R$. Thus the cardinality of $S$ is at least that of $\R$, which is $\frak{c}$.
\end{proof} 

\noindent
Below $H\supseteq \R$ is a 
Hardy field. We set
$$H^{\te}\ :=\ \{f\in H:\, \text{$f$ is overhardian}\},$$
 the {\bf transexponential} (or {\bf overhardian}) {\bf part} of $H$. 
By   Corollary~\ref{overh4} we have
$$H^{\te}\ =\ \big\{f\in H:\, \text{$f>\exp_n(x)$ for all $n$}\big\},$$
so $H^{\te}$ is closed upward in $H^{>\R}$. On $H^{\te}$ we define the equivalence relation
$\sim_{\exp}$ of {\bf exponential equivalence\/} by 
\begin{align*}
f\sim_{\exp} g &\quad :\Longleftrightarrow\quad \text{$f\le \exp_n(g)$ and $g\le \exp_n(f)$ for some $n$}\\
 &\quad \ \Longleftrightarrow\quad \text{$f\le \exp_m(g)$ and $g\le \exp_n(f)$ for some $m$, $n$.}
\end{align*} Let $*f$ be the exponential equivalence class of $f\in H^{\te}$, 
 a convex subset of $H^{\te}$. 
We linearly order the set~$*H^{\te}$ of exponential equivalence classes by:
$$*f < *g\quad :\Longleftrightarrow\quad \text{$\exp_n(f) < g$ for all $n$}\qquad (f,g\in H^{\te}).$$
For a Hardy field extension  $H_1$ of $H$ we have $H_1^{\te}\cap H= H^{\te}$, and we identify $*H^{\te}$ with a subset of $*H_1^{\te}$ via the order-preserving embedding 
$$*f \,(\text{in $*H^{\te}$} ) \ \mapsto\ *f \,(\text{in $*H_1^{\te}$})\quad \text{ for }f\in H^{\te}.$$
Note that $\R(x)^{\te}\ =\ \emptyset$. 
If $H$ is Liouville closed, then $\exp(*f)=\log(*f)=*f$ for~$f\in H^{\te}$.
We record a few other  properties of $\sim_{\exp}$ used later: 

\begin{lemma}\label{lem:simexp}
Let $f\in H^{\te}$. Then
\begin{enumerate}
\item[(i)]  $(*f)\cdot (*f) = *f$;
\item[(ii)] if $g\in H^{>\R}$ and $[vf]=[vg]$, then $g\in H^{\te}$ and $*f=*g$;
\item[(iii)]   $(*f)^\dagger = *f$ and $\der(*f) = *f$.
\end{enumerate}
\end{lemma}
\begin{proof}
Parts (i) and (ii) follow easily from the definitions.
For (iii),   note first that~$f^\dagger$ is overhardian by a remark before Corollary~\ref{overh1} and
$\log f< f^\dagger$ by
Lemma~\ref{lemoverh}, so~$f<\exp(f^\dagger)$, and $f^\dagger<f<\exp(f)$, and thus
$f^\dagger\sim_{\exp} f$. This yields $(*f)^\dagger = *f$, hence $\der(*f) = *f$ by~(i). 
\end{proof}

\noindent
Using results of the previous section we shall prove: 

\begin{prop}\label{prbin} Suppose~$P$,~$Q$ is a countable gap in $*H^{\te}$.  Then $H$ has Hardy field extensions
$H_0=H\<f_0\>$, $H_1=H\<f_1\>$ with $f_0\in H_0^{\te}$, $f_1\in H_1^{\te}$, such that 
$$P< *f_0 < Q,\qquad P < *f_1 < Q,$$ $H_0$ and $H_1$ have no common Hardy field extension, and 
$$*H_0^{\te}=*H^{\te}\cup\{*f_0\},\qquad   *H_1^{\te}=*H^{\te}\cup\{*f_1\}.$$
\end{prop}

\noindent
We accept this for the moment, and indicate how it enables a binary tree construction leading to Theorem~\ref{th2c}. Let $\abs{X}$ denote the cardinality of the set $X$, and
identify as usual a cardinal with the least ordinal of that cardinality, where an ordinal $\lambda$ is considered as the set of ordinals $<\lambda$. 
Let $\cH$ be the set of all Hardy fields~$H\supseteq \R$ such that $\abs{*H^{\te}} < \frak{c}$. We build by transfinite recursion a binary tree in $\cH$ by assigning to each ordinal
$\lambda < \frak{c}$ and function $s\colon \lambda\to \{0,1\}$ a Hardy field~$H_s\in \cH$ with~$\abs{*H_s^{\te}}\le \abs{\lambda}$.
For $\lambda=0$ the function $s$ has empty domain and we take $H_s=\R$. 
Suppose $s\colon \lambda\to \{0,1\}$ as above and $H_s\in \cH$ are given with $\abs{*H_s^{\te}}\le \abs{\lambda}$. Then Lemma~\ref{lmif}  provides a countable gap~$P$,~$Q$ in 
$*H_s^{\te}$. Let~$s0, s1\colon \lambda+1\to \{0,1\}$ be the obvious extensions of $s$, and  let $H_{s0}, H_{s1}\in \cH$ be obtained from $H_s$ as~$H_0$,~$H_1$ are obtained from $H$ in Proposition~\ref{prbin}.  Let~$\lambda< \frak{c}$ be an infinite limit ordinal and
$s\colon \lambda\to \{0,1\}$;  assume that for every $\alpha< \lambda$ there is given $H_{s|\alpha}\in \cH$ with~$H_{s|\alpha}\subseteq H_{s|\beta}$ whenever
$\alpha\le \beta < \lambda$.
Then we set~$H_s:= \bigcup_{\alpha<\lambda} H_{s|\alpha}$. Assuming also inductively that $\abs{*H_{s|\alpha}^{\te}}\le  |\alpha|$ for all $\alpha<\lambda$,
 we obtain $\abs{*H_s^{\te}}\le \abs{\lambda}\cdot \abs{\lambda}=\abs{\lambda}$, as desired.
This finishes the construction of our tree. It yields for any function~$s\colon \frak{c}\to \{0,1\}$ a Hardy field $H_s:=\bigcup_{\lambda<\frak{c}} H_{s|\lambda}$, and the way we constructed the tree guarantees that if $s, s'\colon \frak{c} \to \{0,1\}$ are different, then
$H_s$ and $H_{s'}$ have no common Hardy field extension. Thus there are $2^{\frak{c}}$ many
maximal Hardy fields. 

\medskip\noindent
It remains to prove Proposition~\ref{prbin}. This goes via some lemmas.

\begin{lemma}\label{lemza}
Let $K$ be an asymptotic field with value group $\Gamma$ and $\Psi:=\psi(\Gamma^{\ne})$ its $\Psi$-set.  Let $L$ be an asymptotic field extension of $K$ of finite transcendence degree over $K$. Then $\Psi_L\setminus \Psi$ is finite, where $\Psi_L$ is the $\Psi$-set of $L$.
\end{lemma}
\begin{proof}  By  [ADH, 3.1.11] (Zariski-Abhyankar), $\Gamma_L/\Gamma$ has finite rational rank. Now use that if $\gamma_1,\dots, \gamma_n\in \Gamma_L^{\ne}$ and $\gamma_1^\dagger,\dots, \gamma_n^\dagger $ are pairwise distinct and not in $\Psi$, then~$\gamma_1,\dots, \gamma_n$ are $\Z$-linearly independent modulo $\Gamma$. 
\end{proof}

\begin{lemma}\label{lemli} 
Let $L\supseteq H\supseteq \R$ be   Hardy fields where $L$ is $\d$-algebraic over $H$. 
Then~$*H^{\te}=*L^{\te}$. \textup{(}In particular, $*H^{\te}=*\Li(H)^{\te}$.\textup{)}
\end{lemma} 
\begin{proof} Let $y\in L^{\te}$;  we claim that $y\sim_{\exp} h$ for some $h\in H^{\te}$.
To prove this, set~$\gamma:=vy$. Then  
$\gamma < \psi(\gamma) <  \dots  < \psi^n(\gamma) < \psi^{n+1}(\gamma) <\dots < 0$, so Lemma~\ref{lemza} gives $n\ge 1$ with $\psi^n(\gamma)\in \Gamma:=v(H^\times)$, say
$\psi^n(\gamma)=vh$ with $h\in H^{>}$. Then~$h$ is overhardian and 
$y\sim_{\exp} h$ by Lemma~\ref{lem:simexp}.
\end{proof} 

\begin{lemma}\label{lemte} Let $H\supseteq \R$ be a Hardy field and let $P$,~$Q$ be a countable gap in~$*H^{\te}$.
Then~$H$ has a Hardy field extension $H\<y\>$ with overhardian $y\in\Cc^{<\infty}$ such that~${P< *y < Q}$. For any such $y$ we have $*H\<y\>^{\te}=*H^{\te}\cup \{*y\}$. 
\end{lemma} 
\begin{proof} Using Lemma~\ref{lemli} we arrange that $H$ is Liouville closed, in particular, $\exp_n(x)\in H$ for all $n$. 
Assume for now that $Q\ne \emptyset$.  Then~$P$,~$Q$ gives rise to a wide gap $A$, $B$ in $H$ by 
\begin{align*} A\ &:=\ \big\{\!\exp_n(x):\,n=0,1,2,\dots\big\} \cup \big\{\!\exp_n(a): n=0,1,2,\dots,\  a\in H^{\te},\ *a\in P\big\},\\
B\ &:=\ \big\{\!\log_n(b): n=0,1,2,\dots,\ b\in H^{\te},\ *b\in Q\big\},
\end{align*}     
with $\text{cf}(A)=\omega$ and $\text{ci}(B)=\omega$. Then Corollary~\ref{xo} 
yields an overhardian $y\in \Cc^{<\infty}$ with $A <_{\ex} y < _e B$. Given any such $y$ it generates a Hardy field $H\<y\>$ over $H$ by Lemma~\ref{val}, with
$P<*y < Q$.  Moreover, by the proof of that lemma,
 $$v\big(H\<y\>^\times\big)\ =\ v(H^\times) \oplus \bigoplus_n \Z v\big(y^{\<n\>}\big),$$
 with convex subgroups $\Delta$ of $v(H^\times)$ (as defined in that proof) and $\Delta + D$ of $v\big(H\<y\>^\times\big)$  with $D:=\bigoplus_n \Z v\big(y^{\<n\>}\big)$. 
Let $f\in H\<y\>^{\te}$.  There are three possibilities: \begin{enumerate}
\item $vf\in \Delta$. Then $*f\in *H^{\te}$. 
\item $vf\in \Delta+D$, $vf<\Delta$. Then $vf=mvy^{\<i\>}+o(vy^{\<i\>})$ for some $i$ and some~$m\ge 1$, hence $*f=*y$ by Lemma~\ref{lem:simexp}. 
\item $vf < \Delta+D$. Then an easy argument gives $b>A$ in $H$ with $vf=vb+o(vb)$, hence $*f=*b\in *H^{\te}$  by Lemma~\ref{lem:simexp}.
\end{enumerate} 
 If $Q=\emptyset$, then we set $A:= H^{>\R}$, $B:=\emptyset$, and proceed as before, using results from
Section~\ref{bh} instead of Corollary~\ref{xo} to obtain the existence of an overhardian~$y\in \Cc^{<\infty}$ with $A<_{\ex} y$, and
using instead of Lemma~\ref{val}  the remark following the proof of that lemma.  
\end{proof} 

\noindent
The following consequence of Lemmas~\ref{lemte} and~\ref{lmif} is worth recording. (It also uses the fact that any Hardy field, as a subset of $\Cc$, has cardinality $\le \frak{c}$.) 

\begin{cor} If $H$ is a maximal Hardy field, then  the ordered set $*H^{\te}$ is $\eta_1$, and $\abs{*H^{\te}}=\frak{c}$.
\end{cor}

\noindent
As to the $H$-field $\T$ of transseries that was studied extensively in  [ADH],
we usually think of $\T$ as rather large, but $H^{\te}=\emptyset$ for any Hardy field $H\supseteq \R$ which embeds into $\T$ (as $H$-fields); those $H$ are dwarfed by any maximal Hardy field. 

Next, given $\phi\in \Gi$, call $\phi$ {\bf hardy-small\/} if $\phi^{(n)}\prec 1$ for all $n$, and {\bf hardy-bounded\/} if $\phi^{(n)}\preceq 1$ for all $n$. 
For example, $\sin x$ is hardy-bounded. Here are some simple observations about these notions: 
If $\phi, \theta\in \Gi$ are hardy-small, then so is~$\phi+\theta$. If $\phi\in \Gi$ is hardian and $\phi\prec 1$, then $\phi$ is hardy-small.
If $\phi\in \Gi$ is hardian and $\phi\preceq 1$, then $\phi$ is hardy-bounded. 
If $\phi\in \Gi$ is hardy-bounded and~$\theta\in \Gi$ is hardy-small, then $\phi\theta$ is hardy-small.
If $\phi,\theta\in \Gi$ are hardy-bounded, then so are $\phi+\theta$ and $\phi\theta$. 
A routine computation gives:

\begin{lemma} \label{hshb} If $\phi\in \Gi$ is hardy-small, then $(1+\phi)^{-1}=1+\theta$ for some hardy-small $\theta\in \Gi$,
   and so $(1+\phi)^{-1}$ is hardy-bounded. 
\end{lemma}

\noindent
For the proof of Proposition~\ref{prbin} we shall use (see also \cite[Corollary~5.14]{ADH5}): 

\begin{lemma}[{Boshernitzan \cite[Theorem~13.6]{Boshernitzan82}}]\label{ohhb} Suppose $\phi\in \Gi$ is overhardian and 
$\theta\in \Gi$ is hardy-bounded. 
Then
$\phi+\theta$ is overhardian.
\end{lemma}
\begin{proof} Note that $\phi+\theta\in (\Gi)^\times$ by Corollary~\ref{overh1}. Moreover,
$$(\phi+\theta)^\dagger\ =\  \left[\phi\cdot \left(1+\frac{\theta}{\phi}\right)\right]^\dagger\ =\ \phi^\dagger + \left(1+\frac{\theta}{\phi}\right)^\dagger\ =\ \phi^\dagger + \frac{(\theta/\phi)'}{1+(\theta/\phi)}.$$
Now $\phi^{-1}$ is hardy-small, so $\theta/\phi$ and  $(\theta/\phi)'$ are hardy-small. Hence by Lemma~\ref{hshb},  $\big(1+(\theta/\phi\big)^{-1}$ is hardy-bounded, and so
$\frac{(\theta/\phi)'}{1+(\theta/\phi)}$ is hardy-small.  Therefore, as $\phi^\dagger$ is still overhardian, we can iterate the above to obtain
$$(\phi+\theta)^{\<n\>}\ =\ \phi^{\<n\>} + \theta_n,\quad \text{ with hardy-small $\theta_n$ for $n\ge 1$.}$$
Thus $\phi+\theta$ is overhardian in view of Corollary~\ref{overh1}. 
\end{proof} 

\noindent
In particular, if $\phi\in \Gi$ is overhardian, then so is $\phi + \sin x$. Thus for $H$, $P$, $Q$ as in the hypothesis of Proposition~\ref{prbin} and taking $y$ as in Lemma~\ref{lemte}, 
the conclusion of that proposition holds  for $f_0:=y$ and $f_1:= y + \sin x$ by the proof of Lemma~\ref{lemte}. 


\section{The $H$-couple of  a Maximal Hardy Field}\label{vg}

\noindent
Taking into account Lemma~\ref{eta}, proving Theorem~\ref{mt} has now been
reduced to showing that the value group of every maximal Hardy field is $\eta_1$. Let $H$ be a maximal Hardy field. Then $H$ is an asymptotic field in the sense of
 [ADH, 9.1], so it has an $H$-asymptotic couple
$(\Gamma, \psi)$ where $\Gamma$ is the value group of $H$.

\medskip\noindent
Let us consider more generally any asymptotic field $K$ with its asymptotic couple~$(\Gamma, \psi)$. Recall that
$\psi\colon \Gamma^{\ne}\to \Gamma$ is given by
$\psi(\gamma)=v(g^\dagger)$, with $g\in K^\times$ such that~$v(g)=\gamma$, and that $\psi(\gamma)$ is also written as $\gamma^\dagger$. Recall also that $\psi$ is a valuation on 
$\Gamma$. Let~$(\gamma_{\rho})$
be a pc-sequence in $\Gamma$ with respect to the valuation~$\psi$ on $\Gamma$. Take $g_{\rho}\in K^\times$ with~$v(g_\rho)=\gamma_{\rho}$. Then $(g_{\rho}^\dagger)$ is a pc-sequence
in $K$, since~$v(g_{\sigma}^\dagger-g_{\rho}^\dagger)=(\gamma_{\sigma}-\gamma_{\rho})^\dagger$ for $\sigma> \rho$, provided $\rho$ is sufficiently large. Suppose $g_{\rho}^\dagger\leadsto g^{\dagger}$ with nonzero $g$ in some asymptotic field extension
$L$ of $K$ (possibly~$L=K$). We claim that then for $\gamma=vg$
we have $\gamma_{\rho} \leadsto \gamma$. This is because eventually
$(\gamma-\gamma_{\rho})^\dagger=v(g^\dagger-g_{\rho}^\dagger)$,
and the latter is eventually strictly increasing, using also that
eventually $\gamma_{\rho} \ne \gamma$. 
In particular, if~$K=K^\dagger$ and
every pc-sequence in $K$ of length $\omega$
has a pseudolimit in $K$, then every pc-sequence in $\Gamma$ of length $\omega$ has a pseudolimit in $\Gamma$. Thus by Corollaries~\ref{corimmom},~\ref{cormhci}, and~\ref{corsjo}:

\begin{cor}\label{hco} If $H$ is a maximal Hardy field with asymptotic couple $(\Gamma,\psi)$, then every pc-sequence in $(\Gamma, \psi)$
of length $\omega$ has a pseudolimit in $(\Gamma,\psi)$, and 
$$\cf(\Gamma^{<})\ =\ \ci(\Gamma^{>})\ >\ \omega, \qquad \ci(\Gamma)\ =\ \cf(\Gamma)\ >\ \omega.$$
\end{cor}

\noindent
Corollary~\ref{hco} includes \cite[Proposition~8.1]{ADH5}:
every maximal Hardy field   contains a germ $\ell$ which is {\it translogarithmic,}\/ that is,
$\R <  \ell\leq\ell_n$ for all $n$, where $\ell_n$ is inductively defined by  $\ell_0:=x$ and   $\ell_{n+1}:=\log\ell_n$.

\subsection*{Ordered vector spaces and $H$-couples over an ordered field}
In the rest of this section we fix an ordered field $\k$ (only the case $\k=\R$ is really needed) and use notation, terminology, and results from \cite{ADH3}.  Let $\Gamma$ be an ordered vector space over~$\k$ (as defined there). For $\alpha\in \Gamma$ we defined its
{\it $\k$-archimedean class}\/ 
$$[\alpha]_{\k}\ :=\  \big\{\gamma\in \Gamma:\,\text{$|\gamma|\le c|\alpha|$ and $|\alpha| \le c|\gamma|$ for some $c\in \k^{>}$}\big\},$$
and we linearly ordered the set $[\Gamma]_{\k}$ of $\k$-archimedean classes. We defined $\Gamma$ to be a {\it Hahn space\/} if
for all $\alpha, \gamma\in \Gamma^{\ne}$ with $[\alpha]_{\k}=[\gamma]_{\k}$ 
 there is a scalar $c\in \k^{\times}$ such that $[\alpha-c\gamma]_{\k}<[\alpha]_{\k}$. 
(If $\k=\R$, then the $\k$-archimedean class $[\alpha]_{\k}$ of an element~$\alpha$ in an ordered vector space over $\k$ equals its archimedean class $[\alpha]$, and every ordered vector space over $\k$ is a Hahn space.)  For an ordered vector space $\Delta$ over $\k$ extending
$\Gamma$ we  identify $[\Gamma]_{\k}$ with a subset of $[\Delta]_{\k}$ via the order-preserving embedding $[\gamma]_{\k}\mapsto [\gamma]_{\k}\colon [\Gamma]_{\k} \to [\Delta]_{\k}$. 

Let now $(\Gamma, \psi)$ be an $H$-couple  over  $\k$, as defined in  \cite{ADH3}, so for all $\alpha,\beta\in \Gamma^{\ne}$, 
$$ [\alpha]_{\k}\le [\beta]_{\k}\ \Longrightarrow\  \psi(\alpha) \ge \psi(\beta).$$ We defined $(\Gamma, \psi)$ to be of {\it Hahn type\/} if for all  $\alpha,\beta\in \Gamma^{\ne}$ with
$\psi(\alpha)=\psi(\beta)$ there exists a scalar $c\in \k^\times$ such that $\psi(\alpha-c\beta)> \psi(\alpha)$; a consequence of ``Hahn type''
is that for all $\alpha,\beta\in \Gamma^{\ne}$, 
$$[\alpha]_\k\le [\beta]_\k\ \Longleftrightarrow\ \psi(\alpha)\ge \psi(\beta),$$
and so the underlying ordered vector space $\Gamma$ over $\k$ is a Hahn space. 
We defined~$(\Gamma,\psi)$  to be {\it closed\/}  if $\Psi:= \psi(\Gamma^{\ne})$ is downward closed in the ordered set $\Gamma$, and 
$(\Gamma,\psi)$ has asymptotic integration. 

Let now $K$ be a Liouville closed $H$-field. Recall from \cite{ADH3} that its value group~$\Gamma$ is then an ordered vector space over its (ordered) constant field $C$, with scalar multiplication given by 
$c\,vf=vg$ whenever $f,g\in K^\times$ and 
$cf^\dagger=g^\dagger$. Its asymptotic couple $(\Gamma,\psi)$ with this scalar multiplication is a closed $H$-couple over $C$ of Hahn type. 
For a Liouville closed Hardy field $H\supseteq \R$ its constant field is $\R$, and we construe its asymptotic couple as an $H$-couple over $\R$ as indicated.

\subsection*{Elements of countable type} Let $\Gamma$ be an ordered vector space over $\k$. Let $\beta$ be an element in an ordered vector space over $\k$ that 
extends $\Gamma$. Then we say that~{\bf $\beta$~has countable type over $\Gamma$\/} if $\beta\notin \Gamma$ and
$\cf(\Gamma^{<\beta}), \ci(\Gamma^{>\beta})\le \omega$;  in that case every element in 
$(\Gamma+\k\beta)\setminus \Gamma$ has countable type over 
$\Gamma$. See [ADH, 2.2] for immediate extensions of valued abelian groups and [ADH, 2.4]
for the $\k$-valuation of an ordered vector space over $\k$.

\begin{lemma}\label{lem:immext} Suppose $\beta$ has countable type
over $\Gamma$ and    the ordered vector space $\Gamma+\k\beta$
over $\k$ is an immediate extension of $\Gamma$ with respect to the
$\k$-valuation.
Then~$\beta$ is a pseudolimit of a divergent pc-sequence in 
$\Gamma$ of length~$\omega$.
\end{lemma} 
\begin{proof} The assumptions yield a countable (necessarily infinite) set $A\subseteq \Gamma$ such that for every $\gamma\in \Gamma$ there exists an $\alpha\in A$ with $[\alpha-\beta]_{\k} < [\gamma-\beta]_{\k}$. This easily yields a divergent pc-sequence $(\alpha_n)$ in $\Gamma$ with all $\alpha_n\in A$ such that $\alpha_n\leadsto \beta$.
\end{proof} 

\begin{lemma}\label{keyprep}
Suppose $\cf(\Gamma),\cf(\Gamma^{<})>\omega$, and   $\beta$ has countable type over $\Gamma$. Then  
$$\cf(\Gamma^{<\beta})\ =\ \ci(\Gamma^{>\beta})\ =\ \omega.$$
\end{lemma}
\begin{proof}
If $\beta<\Gamma$, then $\ci(\Gamma^{>\beta})=\ci(\Gamma)=\cf(\Gamma)>\omega$,
contradicting $\ci(\Gamma^{>\beta})\leq\omega$. Thus~$\Gamma^{<\beta}\neq\emptyset$.
If $\cf(\Gamma^{<\beta})\ne \omega$, then
$\Gamma^{<\beta}$ has a largest element $\gamma$,  so
$\Gamma^{>\beta}=\Gamma^{>\gamma}$, contradicting~$\ci(\Gamma^{>\gamma})=\cf(\Gamma^{<})>\omega$. Thus 
$\cf(\Gamma^{<\beta})= \omega$; likewise, $\ci(\Gamma^{>\beta})= \omega$.
\end{proof}

\noindent
For us the relevant fact relating  ``countable type'' to the $\eta_1$-property is as follows:  given an $H$-couple $(\Gamma, \psi)$ over $\k$, 
$$\text{$\Gamma$ is $\eta_1$}\quad\Longleftrightarrow\quad
\begin{cases}
&\parbox{20em}{there is no $H$-couple over $\k$ extending~$(\Gamma, \psi)$ with an
element of countable type over $\Gamma$.}\end{cases}$$
(For ``$\Leftarrow$'' use model-theoretic compactness.) 

\begin{lemma}\label{cps} Let $(\Gamma, \psi)$ be a closed $H$-couple
over $\k$, and suppose $\beta$ in an $H$-couple over $\k$ extending $(\Gamma,\psi)$ has countable type over $\Gamma$
and $\beta^\dagger\notin \Gamma$. 
Then $\beta^\dagger$ has countable
type over $\Gamma$.
\end{lemma}
\begin{proof} Without loss of generality we assume $\beta>0$. Consider first the case
where we have a strictly increasing sequence $(\alpha_m)$ in $\Gamma^{>}$ 
and a strictly decreasing sequence~$(\gamma_n)$ in $\Gamma^{>}$, such that $\alpha_m < \beta < \gamma_n$ for all~$m$,~$n$, and
$(\alpha_m)$ is cofinal in~$\Gamma^{< \beta}$, and $(\gamma_n)$ is coinitial in $ \Gamma^{>\beta}$.
Then $(\alpha_m^\dagger)$ is decreasing, $(\gamma_n^\dagger)$ is increasing, $\alpha_m^\dagger > \beta^\dagger > \gamma_n^\dagger $ for all~$m$,~$n$. 
Using  that the $H$-couple $(\Gamma,\psi)$ is closed we
also obtain that $(\alpha_m^\dagger)$ is coinitial in $\Gamma^{>\beta^\dagger}$ and
that $(\gamma_n^\dagger)$ is cofinal in
$ \Gamma^{<\beta^\dagger}$. Thus
$\beta^\dagger$ has countable type over $\Gamma$. Next consider the case $\beta> \Gamma$. Then  the cofinality of $\Gamma$ is $\omega$,
$\beta^\dagger < \Gamma$, and  so $\beta^\dagger$ has countable type over $\Gamma$, since the coinitiality of $\Gamma$ is also $\omega$. 

The case that there are $\alpha, \gamma\in \Gamma^{>}$ with $\alpha < \beta < \gamma$ and there is a largest $\alpha\in \Gamma^{>}$ with $\alpha < \beta$ or a least
$\gamma\in \Gamma^{>}$ with $\beta < \gamma$ cannot occur, since for such a  largest $\alpha$ we would have $\alpha < \beta < 2\alpha$, so $\alpha^\dagger=\beta^\dagger$, contradicting $\beta^\dagger\notin \Gamma$ (and a least such $\gamma$ yields the same contradiction).  

It remains to consider the case
$0 <  \beta < \Gamma^{>}$. Then $\beta$ being of countable type over~$\Gamma$ yields a strictly decreasing sequence $(\gamma_n)$ in $\Gamma^{>}$ that is coinitial in $\Gamma^{>}$.
Then~$(\gamma_m^\dagger)$ is increasing, $(\gamma'_n)$
is decreasing, $\gamma_m^\dagger < \beta^\dagger < \gamma'_n$ for all~$m$,~$n$, and $(\gamma_m^\dagger)$ is cofinal in 
$\Gamma^{<\beta^\dagger}$ and $(\gamma'_n)$ is coinitial in $\Gamma^{>\beta^\dagger}$. So here
$\beta^\dagger$ is also of countable type over $\Gamma$. 
\end{proof}

\subsection*{Good approximations} Let $\Gamma$ be an ordered vector space over $\k$, and let $\alpha$,~$\gamma$ range over
$\Gamma$. An {\it extension}\/ of $\Gamma$ is an ordered vector space over $\k$ extending $\Gamma$.

\begin{lemma}\label{bapp} Let $\beta\notin \Gamma$  be an element in an extension of $\Gamma$. Then for any $\alpha$,
$$[\beta-\alpha]_{\k} \notin [\Gamma]_{\k}\ \Longrightarrow\  [\beta-\alpha]_{\k} =\min_{\gamma}\  [\beta-\gamma]_{\k}.$$
If $\Gamma+\k\beta$ is a Hahn space, then this implication turns into an equivalence. 
\end{lemma} 
\begin{proof}  If $[\beta-\gamma]_{\k}<[\beta-\alpha]_{\k}$, then $(\beta-\alpha)-(\beta-\gamma)=\gamma-\alpha$ yields
$[\beta-\alpha]_{\k}=[\gamma-\alpha]_{\k}\in [\Gamma]_{\k}$; this gives (the contrapositive of)  ``$\Rightarrow$''. Suppose $\Gamma+\k\beta$ is a Hahn space.
If $[\beta-\alpha]\in [\Gamma]_{\k}$, say $[\beta-\alpha]_{\k}=[\gamma]_{\k}$, then
$\big[\beta- (\alpha+c\gamma)\big]_{\k}<[\beta-\alpha]_{\k}$ for some~$c\in \k^\times$; this proves (the contrapositive of) ``$\Leftarrow$''.
\end{proof} 

\noindent
Suppose $\beta\notin \Gamma$ lies in an extension of $\Gamma$. Then a {\bf good approximation} of $\beta$ in $\Gamma$ is by definition an $\alpha$ such that $[\beta-\alpha]_{\k} \notin [\Gamma]_{\k}$. 
Note that a good approximation of~$\beta$
 in $\Gamma$ exists iff $[\Gamma+\k\beta]_{\k}\neq[\Gamma]_{\k}$.
Together with Lemma~\ref{lem:immext} this yields:

\begin{cor}\label{cor:immext}
Suppose $\beta$ lies in an extension $\Gamma^*$ of $\Gamma$ and $\Gamma^*$ is a Hahn space. Assume also that
there is no divergent pc-sequence of length $\omega$ in $\Gamma$ and that
$\beta$ has
 countable type over~$\Gamma$. Then $\beta$
has a good approximation in $\Gamma$. 
\end{cor}

\begin{lemma}\label{sides}  Suppose $\beta\notin\Gamma$ in an extension of $\Gamma$ has a
good approximation $\alpha$ in~$\Gamma$.  Then the following holds:  \begin{enumerate}
\item[(i)] 
if $[\beta]_{\k}\in [\Gamma]_{\k}$, then
$\alpha\ne 0$, $[\beta-\alpha]_{\k} <[\beta]_{\k}=[\alpha]_{\k}$; and
\item[(ii)] for all $\gamma$, if $\sign(\beta-\gamma)\ne \sign(\beta-\alpha)$, then 
$[\alpha-\gamma]_{\k}=[\beta-\gamma]_{\k}>[\beta-\alpha]_{\k}$.
\end{enumerate}
\end{lemma} 
\begin{proof} 
Part (i) is clear. For (ii), assume  
$\alpha<\beta<\gamma$; the case $\gamma< \beta<\alpha$ reduces to this case by taking negatives. 
Then $\gamma-\alpha>\beta-\alpha>0$, so~$[\gamma-\alpha]_{\k}>[\beta-\alpha]_{\k}$, since
$[\beta-\alpha]_{\k}\notin[\Gamma]_{\k}$. Thus $[\beta-\gamma]_{\k}=
\big[(\beta-\alpha)+(\alpha-\gamma)\big]_{\k}=[\alpha-\gamma]_{\k}$. 
\end{proof} 

\noindent
{\it In the rest of this section $(\Gamma,\psi)$ is an $H$-couple over $\k$, and  $\alpha$, $\gamma$ range over $\Gamma$. By an \emph{extension}\/ of $(\Gamma,\psi)$ we mean an $H$-couple over $\k$ that extends $(\Gamma,\psi)$. We consider $(\Gamma,\psi)$ as a valued ordered vector space over $\k$ with the valuation on $\Gamma$ given by $\psi$, so $\alpha\sim \gamma$ means $(\alpha-\gamma)^\dagger > \alpha^\dagger$.}\/ 
For $\alpha\neq 0$ we  set $\alpha^{\sim}:=\{\gamma:\,\alpha\sim\gamma\}$.

\begin{lemma}\label{lem:alphasim} 
Suppose $(\Gamma,\psi)$ is   closed and $\alpha\neq 0$.  Then
$$\cf(\alpha^\sim)\ =\  \ci(\alpha^\sim)\  =\ \cf(\Gamma^<)\ =\ \ci(\Gamma^>).$$
\end{lemma}
\begin{proof}
We have
$\alpha^{\sim} 
= \big\{\alpha+\gamma:\, \gamma^\dagger>\alpha^\dagger\big\}$.
The map $\alpha+\gamma\mapsto\alpha-\gamma$ is a decreasing permutation of $\alpha^{\sim}$, so
$\cf(\alpha^\sim) = \ci(\alpha^\sim)$. We also have the decreasing map
$$\alpha+\gamma\mapsto \gamma^\dagger\colon\alpha^{\sim}\cap\Gamma^{>\alpha}\to\Gamma^{>\alpha^\dagger}$$ whose image is coinitial in $\Gamma^{>\alpha^\dagger}$, since $(\Gamma,\psi)$ is closed.
Hence~$\cf(\alpha^{\sim})=\ci(\Gamma^{>\alpha^\dagger})=\ci(\Gamma^>)$ by [ADH, 2.1.4].
\end{proof}

\begin{lemma}\label{key} Suppose   $(\Gamma,\psi)$ is of Hahn type, closed, and
$\cf(\Gamma), \cf(\Gamma^{<})>\omega$. Let $\beta$ in an extension of $(\Gamma,\psi)$ have countable type over $\Gamma$ with $[\beta]_{\k}\notin [\Gamma]_{\k}$. Then~$\beta^\dagger\notin \Gamma$, and so $\beta^\dagger$ has countable type over $\Gamma$ by Lemma~\ref{cps}.
\end{lemma}
\begin{proof}  We may replace $\beta$ by $-\beta$, and so we arrange $\beta>0$. For~$0<\alpha < \beta < \gamma$ we have $[\alpha]_{\k} < [\beta]_{\k}<[\gamma]_{\k}$, and so $\alpha^\dagger \ge \beta^\dagger \ge \gamma^\dagger$, but~$\alpha^\dagger > \gamma^\dagger$ by the Hahn type assumption. Suppose towards a contradiction that $\beta^\dagger\in \Gamma$. We distinguish two cases.  First case:
$\alpha^\dagger=\beta^\dagger$ for some $\alpha$ with $0<\alpha < \beta$.  Then~$\beta^\dagger > \gamma^\dagger$ for all~$\gamma>\beta$, but then
$\ci(\Gamma^{>\beta})=\omega$ (by Lemma~\ref{keyprep}) and $(\Gamma,\psi)$ being closed gives for such $\alpha$ that~$\cf(\Gamma^{<\alpha^\dagger})=\cf(\Gamma^{<\beta^\dagger})\le \omega$, contradicting 
$\cf(\Gamma^{<\alpha^\dagger})>\omega$. 
Second case: $\beta^\dagger=\gamma^\dagger$ for some $\gamma>\beta$. This leads to a contradiction in a similar way. 
\end{proof} 

\noindent
We say that $(\Gamma, \psi)$ is {\bf countably spherically complete\/} if every pc-sequence in it of
length $\omega$ pseudoconverges in it. In particular, if $(\Gamma,\psi)$ is
the $H$-couple of a maximal Hardy field (with $\k=\R$), then $(\Gamma,\psi)$ is of Hahn type, closed, countably spherically complete,  and $\cf(\Gamma),  \cf(\Gamma^{<})>\omega$. 
(See Corollary~\ref{hco}.)

If $(\Gamma,\psi)$ is of Hahn type, then the valuation $\psi$ on  $\Gamma$
is equivalent to the $\k$-valuation of~$\Gamma$~[ADH, p.~82]. If
in addition $(\Gamma,\psi)$ is countably spherically complete, then by Corollary~\ref{cor:immext}, any
$\beta$ in an extension of $(\Gamma,\psi)$ and of countable type over $\Gamma$
and such that $\Gamma+\k\beta$ is a Hahn space
 has a good approximation   in $\Gamma$.
 
 In the next lemma only part (i) of the conclusion is needed later. The other parts are included for their independent interest.

\begin{lemma}\label{corkey}
Suppose $(\Gamma,\psi)$ is of Hahn type, closed, and $\cf(\Gamma), \cf(\Gamma^{<})>\omega$.  Let~$\beta$ in an extension of
$(\Gamma,\psi)$ have countable type over $\Gamma$, with $[\beta]_{\k}\in [\Gamma]_{\k}$, and let $\alpha_0$ be a good approximation of $\beta$ in $\Gamma$. 
Then  
\begin{enumerate}
\item[(i)] $\beta_*:=(\beta-\alpha_0)^\dagger\notin \Gamma$, and $\beta_*$ has countable type over $\Gamma$;
\item[(ii)] if $\alpha_0 < \beta$, then there is a sequence $(\gamma_n)$ in
$\Gamma^{>\beta}$ such that 
$$[\beta-\alpha_0] _{\k}< [\gamma_n-\beta]_{\k}< [\beta]_{\k},\  \text{ for all $n$,}$$ 
and $\big([\gamma_n-\beta]_{\k}\big)$ is strictly decreasing and coinitial in $[\Gamma]_{\k}^{>[\beta-\alpha_0]_{\k}}$; 
\item[(iii)] if $\beta< \alpha_0$, then there is a sequence $(\gamma_n)$ in
$\Gamma^{<\beta}$ such that
$$[\beta-\alpha_0] _{\k}< [\beta-\gamma_n]_{\k}< [\beta]_{\k},\ \text{ for all $n$,}$$ and
$\big([\beta-\gamma_n]_{\k}\big)$ is strictly decreasing and coinitial in $[\Gamma]_{\k}^{>[\beta-\alpha_0]_{\k}}$; and
\item[(iv)] $\alpha_0 \sim \beta$, that is, $\beta_* > \alpha_0^\dagger=\beta^\dagger$.
\end{enumerate}
\end{lemma}
\begin{proof}
Applying Lemma~\ref{key} to~$\beta-\alpha_0$ in the role of $\beta$ gives (i).
As to (ii), let~$\alpha_0 < \beta$ and  suppose $[\alpha]_{\k}> [\beta-\alpha_0]_{\k}$; then $[\alpha]_{\k}=[\gamma-\beta]_{\k}$ for some $\gamma>\beta$: taking $\alpha>0$, this holds with
$\gamma:= \alpha_0 + \alpha$. Hence $$\big\{[\gamma-\beta]_{\k}:\,\gamma>\beta\big\}\ =\ [\Gamma]_{\k}^{>[\beta-\alpha_0]_{\k}}$$ by Lemma~\ref{sides}(ii).
Using also (i) we have a decreasing bijection
$$[\gamma-\beta]_{\k}\mapsto (\gamma-\beta)^\dagger\ :\, [\Gamma]_{\k}^{>[\beta-\alpha_0]_{\k}}\to \Gamma^{<\beta_*} \qquad (\gamma>\beta).$$  Thus
$\ci\!\big([\Gamma]_{\k}^{>[\beta-\alpha_0]_{\k}}\big) =\cf(\Gamma^{<\beta_*})=\omega$ by (i) and   Lemma~\ref{keyprep} applied to $\beta_*$ in the role of $\beta$, and
$[\beta]_{\k} >[\beta-\alpha_0]_{\k}$ by Lemma~\ref{sides}(i).
This proves (ii), and taking negatives we obtain (iii). 
For (iv) first note that $\cf(\alpha_0^\sim)=\cf(\Gamma^{<})>\omega$ by Lemma~\ref{lem:alphasim}.
 If~$\alpha_0<\gamma < \beta$, then $\alpha_0 \sim \gamma$:  otherwise  $\alpha_0 < \gamma < \beta$ and 
$[\gamma-\alpha_0]_{\k}\ge[\alpha_0]_{\k}> [\beta-\alpha_0]_{\k}$, which is impossible. The set $\alpha_0^{\sim}$ must contain elements $>\beta$, since otherwise $\alpha_0^{\sim}$ would be a cofinal subset of $\Gamma^{<\beta}$, contradicting $\cf(\Gamma^{<\beta})=\omega$. Thus
$\alpha_0\sim \beta$. 
\end{proof}

\subsection*{Case (b) extensions}
{\em In this subsection $(\Gamma,\psi)$ is an $H$-couple over $\k$ with asymp\-tot\-ic
 integration, and $\beta\notin \Gamma$ is in an $H$-couple
$(\Gamma^*, \psi^*)$ over $\k$ that extends~$(\Gamma,\psi)$}.  
Let $(\Gamma\<\beta\>, \psi_{\beta})$ be the $H$-couple over
$\k$ generated by~$\beta$ over $(\Gamma, \psi)$ in~$(\Gamma^*, \psi^*)$. The structure of the extension 
$\big(\Gamma\<\beta\>, \psi_{\beta}\big)$ of $(\Gamma,\psi)$  is described in detail in
\cite[Section~4]{ADH3}: the possibilities
are listed in \cite[Proposition~4.1]{ADH3} as (a), (b), ${\rm(c)}_n$, and  ${\rm(d)}_n$. 
Case~(b) is as follows:

\begin{itemize}\item[(b)] We have a sequence $(\alpha_i)$ in 
$\Gamma$ and a sequence $(\beta_i)$ in $\Gamma^*$ that is $\k$-linearly independent
over $\Gamma$, such that $\beta_0=\beta-\alpha_0$ and $\beta_{i+1}=\beta_i^\dagger-\alpha_{i+1}$ for all $i$, and such that
$\Gamma\<\beta\>=\Gamma \oplus \bigoplus_{i=0}^\infty \k\beta_i$. 
\end{itemize}

\begin{lemma}\label{caseb} Suppose 
$(\Gamma,\psi)$ is of Hahn type,  closed, countably spherically complete, and $\cf(\Gamma), \cf(\Gamma^{<})>\omega$.
Assume also that $\Gamma^*$ is a Hahn space,
and  $\beta$  has countable type over $\Gamma$. 
Then $\beta$ falls under Case~\textup{(b)}. 
\end{lemma} 
\begin{proof} Suppose $\beta$ falls under Case (a). This means
$(\Gamma+ \k\beta)^\dagger=\Gamma^\dagger$.  In particular, $\beta^\dagger\in \Gamma$, hence $[\beta]_{\k}\in [\Gamma]_{\k}$ by Lemma~\ref{key}, so 
$(\beta-\alpha)^\dagger\notin \Gamma$ for some $\alpha$, by Lem\-ma~\ref{corkey}(i) and the remark preceding that lemma, contradicting $(\Gamma+ \k\beta)^\dagger=\Gamma^\dagger$.  
 
Next, assume $\beta$ falls under Case ${\rm(c)}_n$.   
Then we have $\alpha_0,\dots,\alpha_n\in \Gamma$, and non\-zero~$\beta_0,\dots, \beta_n\in \Gamma^*$ such that 
$\beta_0=\beta-\alpha_0$, $\beta_{i+1}=\beta_i^\dagger-\alpha_{i+1}$ for $0\le i<n$, the vectors $\beta_0,\dots, \beta_n, \beta_n^\dagger$ are $\k$-linearly independent over 
$\Gamma$, and $(\Gamma+\k\beta_n^\dagger)^\dagger=\Gamma^\dagger$. As $\beta$ has countable type over $\Gamma$,
an induction using Lemma~\ref{cps} gives that
$\beta_0,\dots, \beta_n, \beta_n^\dagger$ have countable type over $\Gamma$. But then Case (a) would apply to $\beta_n^\dagger$ in the role of $\beta$, and we already excluded that possibility. 

The cases ${\rm(d)}_n$ are excluded because $(\Gamma,\psi)$ is closed, as noted after the proof  of Proposition 4.1 in \cite{ADH3}. \end{proof}

\noindent
Here is more information about Case (b):

\begin{lemma}\label{lemb} Let $(\alpha_i)$ and $(\beta_i)$ be as in \textup{(b)}. Then:
\begin{enumerate}
\item[(i)] $\beta_i^\dagger\notin \Gamma$ for all $i$, and thus $[\beta_i]_{\k}\notin [\Gamma]_{\k}$ for all $i$; 
\item[(ii)] $\alpha_0$ is a good approximation of $\beta$ in $\Gamma$;
\item[(iii)]  $\alpha_{i+1}$ is a good approximation of $\beta_i^\dagger$ in $\Gamma$, for all $i$;
\item[(iv)] $\beta_i^{\dagger\dagger}\le \beta_{i+1}^\dagger$ for all $i$;
\item[(v)] $(\beta_i^\dagger)$ is strictly increasing, and thus $\big([\beta_i]_{\k}\big)$ is strictly decreasing;
\item[(vi)]$\big[\Gamma\<\beta\>\big]_{\k}=[\Gamma]_{\k}\cup\big\{[\beta_i]_{\k}:\,i\in \N\big\}$, and thus $\Psi_{\beta} = \Psi\cup \big\{\beta_i^\dagger:\
 i\in \N\big\}$;
\item[(vii)] there is no $\delta\in \Gamma\<\beta\>$ with $\Psi < \delta < (\Gamma^{>})'$;
\item[(viii)] $\Gamma^{<}$ is cofinal in $\Gamma\<\beta\>^{<}$.
\end{enumerate}
If $(\Gamma,\psi)$ is closed and $\eta$ in an extension of $(\Gamma,\psi)$ realizes the same cut in $\Gamma$ as $\beta$, then there is an isomorphism
$\big(\Gamma\<\beta\>, \psi_{\beta}\big) \to \big(\Gamma\<\eta\>, \psi_{\eta}\big)$ of $H$-couples over $\k$ that is the identity on $\Gamma$ and sends
$\beta$ to $\eta$.  If $(\Gamma,\psi)$ is of Hahn type, then so is $(\Gamma\<\beta\>, \psi_{\beta})$. 
\end{lemma}
\begin{proof} Except for (ii), (iii), (iv), and the isomorphism claim this is in~\cite[Lemma~4.2]{ADH3}. Now~(ii) holds by $[\beta-\alpha_0]_{\k}=[\beta_0]_{\k}\notin [\Gamma]_{\k}$, 
and (iii) by $[\beta_i^\dagger-\alpha_{i+1}]_{\k}=[\beta_{i+1}]_{\k}\notin [\Gamma]_{\k}$. As to (iv), this is because
$[\beta_i^\dagger]_{\k}\ge [\beta_i^\dagger-\alpha_{i+1}]_{\k}=[\beta_{i+1}]_{\k}$ by (iii) and Lemma~\ref{bapp}. 

Now assume $(\Gamma, \psi)$ is closed and $\eta$ in an extension $(\Gamma_1, \psi_1)$ of $(\Gamma,\psi)$ realizes the same cut in $\Gamma$ as $\beta$, in particular, $\eta\notin \Gamma$. The case
$(\Gamma_1,\psi_1)=(\Gamma^*, \psi^*)$ is actually part of~\cite[Lemma~4.2]{ADH3}, and one can reduce to that case: the theory of closed $H$-couples
over $\k$ has QE in the language specified in~\cite[Section~3]{ADH3}, and so there is an $H$-couple $(\Gamma_1^*, \psi_1^*)$ extending
$(\Gamma,\psi)$ with embeddings~$(\Gamma^*,\psi^*)\to (\Gamma_1^*, \psi_1^*)$ and~$(\Gamma_1,\psi_1) \to (\Gamma_1^*,\psi_1^*)$
over~$\Gamma$. 
\end{proof} 

\noindent
We add the following observations:

\begin{cor}\label{ishift} Suppose $(\alpha_i)=(\alpha_0, \alpha_1, \alpha_2,\dots)$ and $(\beta_i)=(\beta_0, \beta_1,\beta_2,\dots)$ are as in \textup{(b)}. Then $-\beta$ falls under Case~\textup{(b)}
with associated sequences $(-\alpha_0, \alpha_1, \alpha_2,\dots)$ and $(-\beta_0, \beta_1,\beta_2,\dots)$.
Also, for any $i$,  $\beta_i^\dagger$ falls under Case~\textup{(b)}
with associated sequences $(\alpha_{i+1}, \alpha_{i+2},\alpha_{i+3},\dots)$ and $(\beta_{i+1}, \beta_{i+2},\beta_{i+3},\dots)$.
\end{cor}

\begin{cor}\label{cfomega} Suppose $(\Gamma,\psi)$ is closed, $\beta$ has countable type over $\Gamma$, $\alpha < \beta < \gamma$ for some $\alpha$, $\gamma$, and 
$(\alpha_i)$, $(\beta_i)$ are as in~\textup{(b)}. Then $\cf(\Gamma^{<\beta_i})=\ci(\Gamma^{>\beta_i})=\omega$ for all~$i$.
\end{cor}
\begin{proof} Induction using Lemma~\ref{cps} shows that every $\beta_i$ has countable type over $\Gamma$ and for every $i$ there are $\alpha$, $\gamma$ with $\alpha< \beta_i < \gamma$. It follows from  Lemma~\ref{lemb}(viii) that for any $\eta\in \Gamma\<\beta\>\setminus \Gamma$ the ordered set $\Gamma^{<\eta}$ has no largest element and the ordered set  $\Gamma^{>\eta}$ has no least element. Applying this to the $\beta_i$ gives the desired result.
\end{proof} 

\noindent
In the next corollary we let $\k_0$ be an ordered subfield of $\k$. Then $(\Gamma,\psi)$, $(\Gamma^*,\psi^*)$ are also $H$-couples
over $\k_0$.

\begin{cor}\label{cor:lemb} 
Let $(\alpha_i)$ be a sequence in $\Gamma$ and $(\beta_i)$ be a sequence in $\Gamma^*$.
Then~$\beta$  falls under Case~\textup{(b)} with respect to~$(\alpha_i)$,~$(\beta_i)$
iff $\beta$ falls under Case~\textup{(b)} with respect to~$(\alpha_i)$,~$(\beta_i)$ when~$(\Gamma,\psi)$ and~$(\Gamma^*,\psi^*)$ are viewed as $H$-couples
over $\k_0$.
\end{cor}
\begin{proof}
Use Lemma~\ref{lemb}(i),(v) and $\beta_i^\dagger=\beta_{i+1}+\alpha_{i+1}$.
\end{proof}

\noindent
Although the element $\beta$ of $(\Gamma^*, \psi^*)$ does not determine uniquely  the sequence $(\beta_i)$ in Case (b), it follows from Lemma~\ref{lemb}(i),(v),(vi) that $\beta$ does determine uniquely the sequences $(\beta_i^\dagger)$ and $\big([\beta_i]_{\k}\big)$. 
Without changing $\beta$ we still have considerable flexibility in choosing the $\alpha_i$ and $\beta_i$:

\begin{lemma} Let $(\alpha_i)$, $(\beta_i)$ be as in \textup{(b)}. 
Let $\alpha_0^*$ be a good approximation of~$\beta$ in $\Gamma$, and $\alpha_{i+1}^*$
a good approximation of $\beta_i^\dagger$ in $\Gamma$, for all $i$. Set $\beta_0^*:= \beta-\alpha_0^*$ and~$\beta_{i+1}^*:=\beta_i^\dagger - \alpha_{i+1}^*$. 
Then $(\alpha_i^*)$ and $(\beta_i^*)$ are also as in \textup{(b)}, with $[\beta_i^*]_{\k}=[\beta_i]_{\k}$  and 
$\beta_i^*-\beta_i\in \Gamma$ for all $i$. 
\end{lemma} 
\begin{proof} 
We have $\beta_i^*-\beta_i=\alpha_i-\alpha_i^*\in\Gamma$ for each $i$ and so $\Gamma\langle\beta\rangle=\Gamma\oplus\bigoplus_{i=0}^\infty\k\beta_i^*$. From Lemma~\ref{bapp} and Lemma~\ref{lemb}(ii),(iii) we get $[\beta_i^*]_{\k}=[\beta_i]_{\k}$ for all $i$, and so~$\beta_{i+1}^*=\beta_i^\dagger - \alpha_{i+1}^*=
(\beta_i^*)^\dagger-\alpha_{i+1}^*$ as required.
\end{proof}

\noindent
Next we consider a shift $(\Gamma, \psi-\gamma)$ of $(\Gamma,\psi)$ and replace $\beta$ by $\beta-\gamma$, viewed as an element of the extension $(\Gamma^*, \psi^*-\gamma)$ of $(\Gamma, \psi-\gamma)$:

\begin{lemma}  Let $(\alpha_i)$, $(\beta_i)$ be as in \textup{(b)}. Then $\beta-\gamma$ falls under \textup{(b)} with respect to the indicated shifts, as witnessed by the sequences $(\alpha_i-\gamma)$, $(\beta_i)$. 
\end{lemma}

\noindent
At the end of the introduction we defined $\alpha^{\<n\>}$. This comes into play now. 

\begin{lemma}\label{negbeta} Let $(\alpha_i)$ and $(\beta_i)$ be as in \textup{(b)}, and suppose that 
$\beta_i^\dagger < 0$ for all $i$. Then $\beta_i^{\<n+1\>}\le \beta_{i+n}^\dagger < 0$ for all $i$ and all $n$. 
\end{lemma}
\begin{proof} This is trivial for $n=0$. Suppose 
$\beta_i^{\<n+1\>}\le \beta_{i+n}^\dagger$.  Then by Lemma~\ref{lemb}(iv),
\[\beta_i^{\<n+2\>}\ \le\ \beta_{i+n}^{\dagger\dagger}\ \le\  \beta_{i+n+1}^\dagger.    \qedhere\]
 \end{proof} 

\noindent
We next discuss a situation where we can arrange that $\beta_i^\dagger < 0$ for all $i$.
 
\begin{remarkNumbered}\label{rem:shift}
Suppose $\cf(\Gamma^{<})>\omega$ and $(\alpha_i)$, $(\beta_i)$ are as in \textup{(b)}. Then  $\cf(\Psi)=\cf(\Gamma^{<})>\omega$, so we have $\gamma\in \Psi$ with $\beta_i^\dagger < \gamma$ for all $i$, hence
$\beta_i^\dagger-\gamma<0$ for all $i$. Thus~$\beta-\gamma$ falls under Case~(b) with respect to the
shifts $(\Gamma, \psi-\gamma)$ and $(\Gamma^*, \psi^*-\gamma)$ and for the associated sequences
$(\alpha_i-\gamma)$, $(\beta_i)$ we have $(\psi^*-\gamma)(\beta_i) <0$ for all $i$, so that the hypothesis
of Lemma~\ref{negbeta} is satisfied for this shifted situation. 
\end{remarkNumbered}

\subsection*{Constructing a case~\textup{(b)}-extension} Let $K$ be a Liouville closed $H$-field; below we view its asymptotic couple $(\Gamma,\psi)$ as an $H$-couple over $\k:=\Q$. 
Assume $\beta\notin \Gamma$ in an extension
$(\Gamma^*, \psi^*)$ of $(\Gamma,\psi)$ falls under Case~(b). We show:

\begin{prop}\label{prbcon} There exists an $H$-field extension $K\<y\>$ of $K$ such that: \begin{enumerate}
\item[(i)] $y>0$ and $vy\notin \Gamma$ realizes the same cut in $\Gamma$ as $\beta$;
\item[(ii)] for any $H$-field extension $M$ of $K$ and  any $z\in M^{>}$ such that~$vz\notin \Gamma$ and $vz$ realizes the same cut in $\Gamma$ as $\beta$, there is an $H$-field embedding~$K\<y\> \to M$ over $K$ sending $y$ to $z$.
\end{enumerate}
\end{prop} 
\begin{proof}
Model-theoretic compactness gives a Liouville closed $H$-field extension $L$ of~$K$ with 
$y\in L^>$ such that $vy\notin \Gamma$ realizes the same cut in $\Gamma$ as $\beta$.  Lemma~\ref{lemb} then yields an isomorphism 
$\big(\Gamma\<\beta\>, \psi_{\beta}\big) \to \big(\Gamma\<vy\>, \psi_y\big)$
of $H$-couples over $\Q$ that is the identity on $\Gamma$ and sends $\beta$ to $vy$. (Here $\big(\Gamma\<vy\>, \psi_y\big)$ is the $H$-couple over
$\Q$ generated by $\Gamma\cup \{vy\}$ in the $H$-couple of $L$ over $\Q$.) It follows that
 $\Gamma\<vy\>/\Gamma$ has infinite dimension as a vector space over $\Q$, so $y$ is differentially transcendental over~$K$ in view of 
 $\Gamma\<vy\>\subseteq v(K_y^\times)$ where $K_y$ is the real closure of $K\<y\>$ in $L$.  We claim that $K\<y\>$ has the properties stated in the proposition;
in particular, we show that $K\<y\>$ is an $H$-subfield of $L$, not just an asymptotic (ordered) subfield of $L$. 

Let $(\alpha_i)$ and $(\beta_i)$ be as in (b); for each $i$, take $f_i\in K^{>}$ such that $vf_i=\alpha_i$. 
We define
$y_i\in K\<y\>$  by recursion:  $y_0:=y/f_0$, and $y_{i+1}=y_i^\dagger/f_{i+1}$; to make this recursion possible we simultaneously show by induction on $i$ that $y_i\ne 0$ and $vy_i\notin \Gamma$ realizes the same cut in $\Gamma$ as $\beta_i$, and
$v(y_i^\dagger)\notin \Gamma$ realizes the same cut in $\Gamma$ as $\beta_i^\dagger$. This is all straightforward using the above isomorphism
\begin{equation}\label{eq:beta->y}\big(\Gamma\<\beta\>, \psi_{\beta}\big) \to \big(\Gamma\<vy\>, \psi_y\big),\end{equation}
which sends $\beta_i$ to $vy_i$ for all $i$. 
Likewise we obtain that for all $n$,
\begin{align*} K_n\ &:=\ K\big(y, y',\dots, y^{(n)}\big)\ =\ K(y_0,\dots, y_{n})\ =\ K\big(y,\dots, y^{\<n\>}\big), \text{ and}\\ 
v(K_n^\times)\ &=\ \Gamma \oplus \Z vy_0\oplus\cdots\oplus \Z vy_{n}\ \subseteq\ \Gamma\<vy\>,
\end{align*}  with the above isomorphism \eqref{eq:beta->y} restricting to an isomorphism
$$\Gamma\oplus \Z\beta_0\oplus \cdots \oplus \Z \beta_{n}\ \to\  \Gamma \oplus \Z vy_0\oplus\cdots\oplus \Z vy_{n}, \quad\ \beta_i\mapsto 
vy_i \quad  (i=0,\dots,n)$$
of ordered abelian groups. Hence the residue field $\res(K_n)$ of the valued subfield~$K_n$ of~$L$ is algebraic over $\res(K)$ by
 [ADH, 3.1.11] (Zariski-Abhyankar), and so $\res(K)$ being real closed gives $\res(K_n)=\res(K)$. Then from $K\<y\>=\bigcup_n K_n$
we obtain~$\res\!\big(K\<y\>\big)=\res(K)$, so $K\<y\>$ is an $H$-subfield of $L$ with the same constant field as~$K$, by  [ADH, 9.1.2]. 
So far we only used $y\ne 0$ rather than $y>0$. 

Next, let $M$ be any $H$-field extension of $K$ and $z\in M^\times$ such that $vz\notin \Gamma$ realizes the same cut as $\beta$ in $\Gamma$.
By increasing $M$ we can assume $M$ is Liouville closed, and then all the above goes through with $z$ instead of $y$.
In particular, setting~$z_0:= z/f_0$ and $z_{i+1}:= z_i^\dagger/f_{i+1}$, we obtain for each $n$ an isomorphism
 of the valued subfield~$K_n$ of $L$ onto the valued subfield $K(z_0,\dots, z_n)$ of $M$ over $K$, sending~$y_i$ to $z_i$ for~${i=0,\dots, n}$.  These have a common extension to a valued differential field iso\-mor\-phism~$K\<y\> \to K\<z\>$ over $K$ sending $y$ to $z$.  For this isomorphism to preserve the ordering, we now assume besides $y>0$ that also $z>0$.
Induction on~$i$ then shows that $y_i$ and $z_i$ are both positive, or both negative, for each $i$: use that  all~$f_i>0$
and that for any $g$ in any $H$-field we have: $$g\succ 1\ \Rightarrow\ g^\dagger >0, \qquad g\prec 1\ \Rightarrow\ g^\dagger < 0.$$ 
The valuation determines for every polynomial $P(Y_0,\dots, Y_n)\in K[Y_0,\dots, Y_n]^{\ne}$ the unique dominant term in $P(y_0,\dots, y_n)$ and in $P(z_0,\dots, z_n)$ in the same way, so this
isomorphism $K\<y\>\to K\<z\>$ is also order-preserving. 
\end{proof} 

\begin{remarkNumbered}\label{rem:prbcon} 
Let $K\langle y\rangle$ be an $H$-field extension of $K$ with $y>0$ such that $vy$ realizes the same cut in $\Gamma$ as $\beta$,  with real closure  $F:=K\langle y\rangle^{\operatorname{rc}}$. By the proof above~$F$ has the same constant field as $K$, and the $H$-couple of 
$F$ over $\Q$ is generated over~$(\Gamma,\psi)$ by~$vy$, as witnessed by an isomorphism $\big(\Gamma\langle\beta\rangle,\psi_\beta) \to (\Gamma_F,\psi_F)$
over $\Gamma$ sending $\beta$ to $vy$.
\end{remarkNumbered}

\noindent
With $f_i$, $y_i$ as in the proof above (so $y_i^\dagger=f_{i+1}y_{i+1}$ for all $i$), we think informally of the element $y$ in Proposition~\ref{prbcon}
as given in terms of the $f_i$ by
$$y\	=\ f_0y_0\ =\ f_0\ex^{{}^{\textstyle\int\! f_1y_1}}\ =\ f_0\ex^{{}^{\textstyle\int\! f_1\ex^{{}^{\textstyle\int\! f_2y_2}}}} 
	\ =\cdots\ =\ \ f_0\ex^{{}^{\textstyle\int\! f_1\ex^{{}^{\textstyle\int\! f_2\ex^{{}^{\textstyle\int {}^{\textstyle\iddots} }}}}}}$$
In the next section we show how to construct such a $y$ analytically when $K$ is a Liouville closed Hardy field containing $\R$, under   additional hypotheses on~$\beta$. 

\section{Filling Gaps of Type \textup{(b)}}\label{fbg} 

\noindent
In Section~\ref{vg}---see in particular the remark at the beginning of that section and the remark preceding Lemma~\ref{cps}---we showed that Theorem~\ref{mt} reduces to:

\begin{lemma}\label{lemth01}  Let $H$ be a maximal Hardy field with $H$-couple $(\Gamma,\psi)$ over $\R$. Then no element in any extension of $(\Gamma,\psi)$ has countable type over $\Gamma$.
\end{lemma}
\begin{proof} Suppose towards a contradiction that  $\beta$  in some extension of $(\Gamma,\psi)$ has countable type
over $\Gamma$. Then $\beta$ falls under Case \textup{(b)} by 
the remarks that precede Lem\-ma~\ref{corkey} and by
Lem\-ma~\ref{caseb}.
Let $(\alpha_i)$, $(\beta_i)$ be as in (b). 
Then $(\beta_i^\dagger)$ is strictly increasing by Lemma~\ref{lemb}(v). 
Since $\cf(\Gamma^{<})=\cf(\Psi) >\omega$, we   can take~$\gamma\in \Psi$ such that $\beta_i^\dagger < \gamma$ for all $i$. 
Take $g\in H^{>}$ with $vg=\gamma$, and $\ell\in H$ with $\ell'=g$, so~$\ell>\R$. Composing with~$\ell^{\inv}$ yields a
maximal Hardy field~$H\circ \ell^{\inv}$ whose $H$-couple over~$\R$ we identify with the shift $(\Gamma, \psi-\gamma)$
of $(\Gamma, \psi)$. As indicated in Remark~\ref{rem:shift} this allows us to replace $H$ by~$H\circ \ell^{\inv}$ and $\beta$ by $\beta-\gamma$. By renaming we thus arrange that $\beta_i^\dagger<0$ for all
$i$. This situation is impossible by Theorem~\ref{bconstruction} below. 
\end{proof}  

\noindent
Theorem~\ref{bconstruction} is of interest independent of Theorem~\ref{mt} and Lemma~\ref{lemth01}, since
it involves a new way of constructing certain Hardy field extensions. 

\begin{theorem}\label{bconstruction} Let $H\supseteq \R$ be a Liouville closed Hardy field with $H$-couple $(\Gamma,\psi)$
over $\R$. Suppose $\beta$ in an extension of $(\Gamma,\psi)$ and of countable type over $\Gamma$ falls under Case \textup{(b)}, and $\beta_i^\dagger < 0$ for all $i$, where
$(\alpha_i)$, $(\beta_i)$ are as in \textup{(b)}. Then there exists~${y\ne 0}$  in a Hardy field extension of $H$ such that $vy$ realizes the same cut in~$\Gamma$ as $\beta$. 
\end{theorem} 

\noindent
The special cases $\beta< \Gamma$ and $\beta>\Gamma$ of Theorem~\ref{bconstruction} are taken care of by Section~\ref{bh}: say $\beta< \Gamma$; then $\ci(\Gamma)=\cf(H)=\omega$, and so there are overhardian
$y>_{\ex} H$, and any such $y$ has the desired property by Corollary~\ref{overh3}. 

The rest of this section proves Theorem~\ref{bconstruction} in the case where $\alpha < \beta < \gamma$ for some
$\alpha, \gamma\in \Gamma$. As we saw, this is also the final step in proving Theorem~\ref{mt}.

\subsection*{Some useful inclusions} Let $H\supseteq \R$ be a Liouville closed Hardy field with $H$-couple $(\Gamma,\psi)$ over $\R$, and let $\alpha$, $\gamma$ range over $\Gamma$. 
Let $\beta$ in an extension of $(\Gamma,\psi)$ of Hahn type be such that $\alpha < \beta< \gamma$ for some $\alpha$, $\gamma$, $[\beta]\notin [\Gamma]$ (so $\beta^\dagger\notin \Gamma$), and~$\beta^{\<n\>}<0$ for all $n\ge 1$. (We do allow $\beta>0$, but use $-|\beta|$ below to arrange a value~$<0$, with~$(-|\beta|)^{\<n\>}=\beta^{\<n\>}$ for $n\ge 1$.) Set  
$$A\ :=\ \big\{h\in H^{>\R}:\, -|\beta|< vh\big\},\qquad B :=\ \big\{h\in H^{>\R}: vh < -|\beta|\big\}.$$
Then $A\cup B=H^{>\R}$, $A< B$, and so there is no $h\in H$ with $A<h<B$. Also  
$$ vA\cup vB\ =\ \Gamma^{<},\quad vB\  <\  -|\beta|\  <\  vA\ <\ 0,$$ and so there is no $\alpha$ with $vB < \alpha < vA$. 

\begin{lemma}\label{AB}  The sets $A$ and $B$ have the following properties:\begin{enumerate}
\item[(i)] $\ex_n:= \exp_n(x)\in A$ for all $n$, and $B\ne \emptyset$;
\item[(ii)] $A=\sq(A)$ and $B=\sqrt{B}$.
 \end{enumerate}
 \end{lemma} 
 \begin{proof} 
As to (i), an easy induction shows that 
$\ex_n^\dagger=\ex_1\cdots \ex_{n-1}$ for $n\ge 1$ and~$\ex_n^{\<m\>}\sim \ex_{n-m+1}^\dagger$ for $n\ge m\ge 1$. In particular, $\ex_n^{\<n\>}\sim 1$ for $n\ge 1$. Since $\beta^{\<n\>}<0$ for all~$n\ge 1$, this gives $v(\ex_n)>\beta$ for all $n$. Item (ii) follows from $[\beta]\notin [\Gamma]$. 
\end{proof}

\noindent
We now set 
\begin{align*} A^\dagger\ &:=\ \big\{a^\dagger:\,a\in A,\  a^\dagger\succ 1\big\}, \quad
B^\dagger\ :=\ \big\{b^\dagger:\,b\in B\big\}, \text{ so in view of $\beta^\dagger\notin \Gamma$:}\\
A^\dagger\ &=\ \big\{h\in H^{>\R}:\,\beta^\dagger < vh\big\}, \quad
B^\dagger\ =\ \big\{h\in H^{>\R}:\,vh < \beta^\dagger\big\}.
\end{align*}
Thus $A^\dagger\cup B^\dagger=H^{>\R}$, $A^\dagger < B^\dagger$, and there is no $h\in H$ with $A^\dagger <h<B^\dagger$. Also
 $$v(A^\dagger) \cup v(B^\dagger)\ =\ \Gamma^{<}, \qquad v(B^\dagger)\  <\  \beta^\dagger\  <\  v(A^\dagger)\  <\ 0,$$ and there is no $\alpha$ with $v(B^\dagger) < \alpha < v(A^\dagger)$. Note also that $\ex_n^\dagger\in A^\dagger$ for all $n\ge 2$. 
 
 \begin{cor}\label{corAB} $\log A\subseteq A^\dagger\subseteq A$ and $\log B\supseteq B^\dagger\supseteq B$.
 \end{cor}
 \begin{proof} If $h\in H$, $h\ge \ex_2$, then $\log h \preceq (\log h)'=h^\dagger$. 
 Then by Lemma~\ref{AB}(i) we have~$\log A \subseteq A^\dagger$.
 Now use $\log A< \log B$ and $\log A \cup \log B=A^\dagger\cup B^\dagger=H^{>\R}$. As to $A^\dagger\subseteq A$: if $h\in H^{>\R}$ and $h^\dagger\succ 1$, then $vh^\dagger=o(vh)$ by [ADH, 9.2.10(iv)]. 
 \end{proof}
  
 \noindent
 To indicate the dependence of $A$, $B$, $A^\dagger$, $B^\dagger$ on $\beta$ we may denote these sets by 
 $$A(\beta),\quad B(\beta),\quad A^\dagger(\beta),\quad  B^\dagger(\beta). $$ 
 In fact, these four sets depend only on $[\beta]$ rather than $\beta$, in view of $[\beta]\notin [\Gamma]$. 

Recall that $\beta^\dagger\notin \Gamma$ and $\beta^\dagger < 0$, so if $\beta^\dagger$ has a good approximation in $\Gamma$, it
has a good approximation $\le 0$ in $\Gamma$.  Note: if $[\beta^\dagger]\notin [\Gamma]$, then $0$ is a good approximation
of $\beta$ in $\Gamma$ and any good approximation $\alpha\le 0$ to $\beta^\dagger$ in $\Gamma$ satisfies $\beta^\dagger < \alpha$. 

Suppose now that $\alpha\le 0$ is a good approximation of $\beta^\dagger$ in $\Gamma$, so $[\beta^\dagger-\alpha]\notin [\Gamma]$. 
Set $\beta_{\nx}:= \beta^\dagger-\alpha$, and assume also that $\beta_{\nx}^{\<n\>}<0$ for all $n\ge 1$. This means that the conditions we imposed earlier on $\beta$ are now also satisfied by $\beta_{\nx}$. Since $[\beta_{\nx}]$ does not depend on the particular
good approximation $\alpha\le 0$ of $\beta^\dagger$ in $\Gamma$, 
$$A(\beta_{\nx})\ =\ \big\{h\in H^{>\R}: vh >-|\beta_{\nx}|\big\}$$ 
doesn't either, and the assumption  that $\beta_{\nx}^{\<n\>}<0$ for all $n\ge 1$ will still be satisfied for any such $\alpha$.

\begin{lemma}\label{nextlog} $A(\beta_{\nx})\subseteq \log A(\beta)$. 
\end{lemma}
\begin{proof} Let $h\in A(\beta_{\nx})$; it suffices to show that then  $\ex^h\in A(\beta)$. Suppose towards a contradiction that $\ex^h\in B$. Then $v \ex^h< -|\beta|$, so $v(\ex^h)^\dagger =vh'<\beta^\dagger <0$.
If $[\beta^\dagger]\in \Gamma$, then $[\beta^\dagger] > [\beta_{\nx}]$, so $h'\in B(\beta_{\nx})$, and thus
$h\in B(\beta_{\nx})$ by Lemma~\ref{in3}(ii)  applied to $B(\beta_{\nx})$ in the role of $B$. If $[\beta^\dagger]\notin \Gamma$, then 
$[\beta^\dagger] = [\beta_{\nx}]$, and again~$h'\in B(\beta_{\nx})$, so
$h\in B(\beta_{\nx})$. In both cases we contradict  $h\in A(\beta_{\nx})$.
\end{proof}

\noindent
The diagram in Figure~\ref{fig:nextlog} depicts the gaps 
$$(A,B)=\big(A(\beta),B(\beta)\big),\quad (A^\dagger,B^\dagger),\quad (\log A,\log B), \quad \big(A(\beta_{\nx}),B(\beta_{\nx})\big)$$ in $H$ and   hypothetical $H$-hardian germs $y$, $y_{\nx}$ with $A<y<B$
and~$A(\beta_{\nx})<y_{\nx}<B(\beta_{\nx})$, as well as $y^\dagger$ and $\log y$.

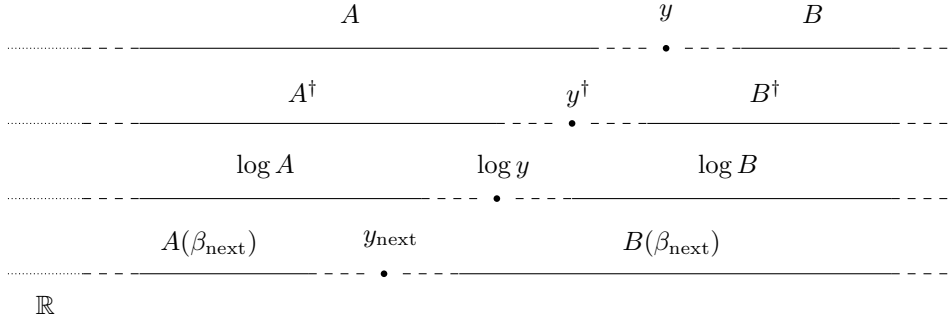
\begin{figure}[ht]
\begin{tikzpicture}

\def\betapt{7.75};
\def\betadaggerpt{6.5};
\def\logbetapt{5.5};
\def\betanxpt{4};

\draw[densely dotted] (-1,0) -- (0,0);
\draw[dashed] (0,0) -- (0.75,0);
\draw (0.75,0) -- (\betapt-1,0);
\draw[dashed] (\betapt-1,0) -- (\betapt-0.25,0);
\draw[dashed] (\betapt+0.25,0) -- (\betapt+1,0);
\draw (\betapt+1,0) -- (10.75,0);
\draw[dashed] (10.75,0) -- (11.5,0);
\node at ( 0.5*\betapt-0.325, 0) [anchor = north,shift={(0em,2em)}]{$A$};
\node at ( 5.825+0.5*\betapt, 0) [anchor = north,shift={(0em,2em)}]{$B$};

\node[circle,fill=black,inner sep=0pt,minimum size=0.25em]   at (\betapt,0) {};
\node at ( \betapt, 0) [anchor = north,shift={(0em,2em)}]{$y$};

\draw[densely dotted] (-1,-1) -- (0,-1);
\draw[dashed] (0,-1) -- (0.75,-1);
\draw (0.75,-1) -- (\betadaggerpt-1,-1);
\draw[dashed] (\betadaggerpt-1,-1) -- (\betadaggerpt-0.25,-1);
\draw[dashed] (\betadaggerpt+0.25,-1) -- (\betadaggerpt+1,-1);
\draw (\betadaggerpt+1,-1) -- (10.75,-1);
\draw[dashed] (10.75,-1) -- (11.5,-1);
\node at ( 0.5*\betadaggerpt-0.325,-1) [anchor = north,shift={(0em,2em)}]{$A^\dagger$};
\node at ( 5.825+0.5*\betadaggerpt, -1) [anchor = north,shift={(0em,2em)}]{$B^\dagger$};

\node[circle,fill=black,inner sep=0pt,minimum size=0.25em]   at (\betadaggerpt,-1) {};
\node at ( \betadaggerpt, -1) [anchor = north,shift={(0.25em,2em)}]{$y^\dagger$};

\draw[densely dotted] (-1,-2) -- (0,-2);
\draw[dashed] (0,-2) -- (0.75,-2);
\draw (0.75,-2) -- (\logbetapt-1,-2);
\draw[dashed] (\logbetapt-1,-2) -- (\logbetapt-0.25,-2);
\draw[dashed] (\logbetapt+0.25,-2) -- (\logbetapt+1,-2);
\draw (\logbetapt+1,-2) -- (10.75,-2);
\draw[dashed] (10.75,-2) -- (11.5,-2);
\node at ( 0.5*\logbetapt-0.325,-2) [anchor = north,shift={(0em,2em)}]{$\log A$};
\node at ( 5.825+0.5*\logbetapt, -2) [anchor = north,shift={(0em,2em)}]{$\log B$};

\node[circle,fill=black,inner sep=0pt,minimum size=0.25em]   at (\logbetapt,-2) {};
\node at ( \logbetapt, -2) [anchor = north,shift={(0.25em,2em)}]{$\log y$};

\draw[densely dotted] (-1,-3) -- (0,-3);
\draw[dashed] (0,-3) -- (0.75,-3);
\draw (0.75,-3) -- (\betanxpt-1,-3);
\draw[dashed] (\betanxpt-1,-3) -- (\betanxpt-0.25,-3);
\draw[dashed] (\betanxpt+0.25,-3) -- (\betanxpt+1,-3);
\draw (\betanxpt+1,-3) -- (10.75,-3);
\draw[dashed] (10.75,-3) -- (11.5,-3);
\node at ( 0.5*\betanxpt-0.325,-3) [anchor = north,shift={(0em,2em)}]{$A(\beta_{\nx})$};
\node at ( 5.825+0.5*\betanxpt, -3) [anchor = north,shift={(0em,2em)}]{$B(\beta_{\nx})$};

\node[circle,fill=black,inner sep=0pt,minimum size=0.25em]   at (\betanxpt,-3) {};

\node at (-0.5, -3) [anchor = north,shift={(0em,-.5 em)}]{$\R$};
\node at ( \betanxpt, -3) [anchor = north,shift={(0.25em,2em)}]{$y_{\nx}$};
\end{tikzpicture}
\caption{Various gaps in $H$ associated to  $(A,B)$}\label{fig:nextlog}
\end{figure}

\begin{lemma}\label{ABwide} We have $[\beta]> [\beta_{\nx}]$. If there is no $\gamma$ such that
$[\beta] > [\gamma] > [\beta_{\nx}]$, then $A(\beta)=A(\beta_{\nx})$, $B(\beta)=B(\beta_{\nx})$, and $A$, $B$ is a wide gap. 
\end{lemma}
\begin{proof} From  [ADH, 9.2.10(iv)] and $\beta^\dagger < 0$ we get $\beta^\dagger=o(\beta)$, so
$[\beta]>[\beta^\dagger]\ge[\beta_{\nx}]$. 
Suppose there is no $\gamma$ with
$[\beta] > [\gamma] > [\beta_{\nx}]$. Then clearly $A(\beta)=A(\beta_{\nx})$ and~$B(\beta)=B(\beta_{\nx})$,   so
$A\subseteq \log A$ by Lemma~\ref{nextlog}. Thus $A$, $B$ is a wide gap. 
\end{proof}

\noindent
By Lemma~\ref{AB} we have $\ex_n\in A(\beta_{\nx})$ for all $n$. In combination with the next result this gives further information about the behavior of
$A$ and $B$ and of the gap between them.  For $p, \phi \in \Cc$ with $\phi>_{\ex} 0$ we have the germ
$\phi^p\in \Cc$. Let  $p\in H$; then $h\in H^{>}$ gives $h^p=\exp(p\log h)\in H^{>}$, and for $S\subseteq H^{>}$ we set $S^p:= \{h^p:\,h\in S\}\subseteq H^{>}$.  

\begin{prop}\label{powergap} The sets $A$, $A^\dagger$, $A(\beta_{\nx})$, $B$ have the following properties: \begin{enumerate}
\item[(i)]$A(\beta_{\nx})\cdot A^\dagger\subseteq A^\dagger$;
\item[(ii)] if $p\in A(\beta_{\nx})$, then $A^p\subseteq A$ and $B^{1/p}\subseteq B$;
\item[(iii)]  if $p\in A(\beta_{\nx})$, $\phi\in \Cc$, and
$A <_{\ex} \phi <_{\ex} B$, then $A<_{\ex} \phi^{1/p}<_{\ex} \phi <_{\ex} \phi^p <_{\ex} B$. 
\end{enumerate}
\end{prop} 
\begin{proof} For (i) we distinguish two cases. Suppose first that $\beta^\dagger < \alpha$. Let $p\in A(\beta_{\nx})$, $h\in A^\dagger$; we need to show $ph\in A^\dagger$, that is, 
$v(ph) >\beta^\dagger$, equivalently, $vp > \beta^\dagger - vh$. Since $vp> \beta_{\nx}=\beta^\dagger-\alpha$, we do have $vp > \beta^\dagger - vh$ if $vh\ge \alpha$. If $vh< \alpha$, then~$\beta^\dagger < vh < \alpha$, so $vh$ is also a good approximation of $\beta^\dagger$ in $\Gamma$, and then replacing~$\alpha$ by $vh$ yields $vp > \beta^\dagger - vh$ in view of remarks made earlier about~$A(\beta_{\nx})$.

Next, suppose $\alpha < \beta^\dagger$, so $[\beta^\dagger]\in [\Gamma]$ by
an earlier remark. Let $p\in A(\beta_{\nx})$, $h\in A^\dagger$; as before we need to show $vp > \beta^\dagger - vh$. Now $\alpha < \beta^\dagger < vh$ gives~$[\alpha-\beta^\dagger] < [\beta^\dagger-vh]$ by Lemma~\ref{sides}(ii). Since $\alpha-\beta^\dagger$ and $\beta^\dagger - vh$ are both negative, this yields~$\alpha-\beta^\dagger > \beta^\dagger-vh$, which together with $vp> \alpha-\beta^\dagger$ gives 
$vp > \beta^\dagger - vh$.

 As to (ii), let $p\in A(\beta_{\nx})$ and $h\in A$.  We have $h^p>\R$ and $$(h^p)^\dagger\ =\ (p\log h)'\ =\ p'\log h +ph^\dagger,$$ 
and $ph^\dagger\preceq 1$ or $ph^\dagger \in A^\dagger$ by (i). Also $p'\in A(\beta_{\nx})$ or $0<p'\preceq 1$, by Lemma~\ref{in2}, and $\log h\in A^\dagger$ by Corollary~\ref{corAB}, so $p'\log h\preceq 1$ or $p'\log h\in A^\dagger$ by (i). Hence~$(h^p)^\dagger \in A^\dagger$, and thus $h^p\in A$. 
Next, let  $p\in A(\beta_{\nx})$ and $h\in B$. Then $h^{1/p}\notin B$ would mean $h^{1/p}\in A$ or $0<h^{1/p}\preceq 1$, and in either case  $h=(h^{1/p})^p$ would give $h \in A$ or~$h\preceq 1$, contradicting $h\in B$. This concludes the proof of (ii).  

 Property (iii) is a routine consequence of (ii).
\end{proof} 

\noindent
Part (iii) of Proposition~\ref{powergap} is only relevant if there is any $\phi\in \Cc$ with $A<_{\ex} \phi <_{\ex} B$. There are indeed such $\phi$ if $\cf(A)=\ci(B)=\omega$, by Corollary~\ref{corsmooth}.

\medskip\noindent
To describe $A(\beta_{\nx})$ directly in terms of $A^\dagger$, take $f\in H^{>}$ with $vf=\alpha$. Then:

\begin{lemma}\label{lemdir} If $\beta^\dagger < \alpha$, then $f\in A^\dagger$ or $f\asymp 1$, and
 $$A(\beta_{\nx})\ =\ H^{>\R}\cap f^{-1}A^\dagger, \qquad B(\beta_{\nx})\ =\ f^{-1}B^\dagger. $$
 If $\alpha < \beta^\dagger$, then $f\in B^\dagger$, $[\beta^\dagger - \alpha] < [\beta^\dagger]=[\alpha]\in [\Gamma]$, and
 $$  A(\beta_{\nx})\ =\ H^{>\R}\cap f (B^\dagger)^{-1},\quad
  f(A^\dagger)^{-1} \text{ is a coinitial subset of }B(\beta_{\nx}). $$
 \end{lemma} 
 \begin{proof} If $\beta^\dagger < \alpha$, then the inclusion $A(\beta_{\nx})\supseteq H^{>\R}\cap  f^{-1}A^\dagger$ and the equality~$B(\beta_{\nx})=f^{-1}B^\dagger$ are almost obvious, and one can use $\alpha\le 0$ to prove the inclusion $A(\beta_{\nx})\subseteq H^{>\R}\cap  f^{-1}A^\dagger$. 
 
 Next, suppose $\alpha < \beta^\dagger$. Then for the inclusion $A(\beta_{\nx})\supseteq H^{>\R}\cap f (B^\dagger)^{-1}$, use~$\beta^\dagger < \alpha-\beta^\dagger$, and for the statement about $B(\beta_{\nx})$, note that  
 \begin{align*} B(\beta_{\nx})\ &=\  \big\{h\in H^{>\R}:\,vh < \alpha-\beta^\dagger\big\}, \text{ and }\\
f(A^\dagger)^{-1}\ &=\ \big\{h\in H^{>\R}:\,\alpha < vh < \alpha-\beta^\dagger\big\}. \qedhere
\end{align*} 
 \end{proof} 
 
 \noindent
 Using the first part of Lemma~\ref{lemdir} we obtain: 
 
 \begin{cor}\label{az0z1}Suppose $\beta_{\nx}<0$ and $z_0, z_1\in \Cc^{<\infty}$ are such that
 $$z_0 >_{\ex} 0,\qquad z_0^\dagger\ =\ fz_1, \qquad A(\beta_{\nx}) <_{\ex} z_1 <_{\ex} B(\beta_{\nx}).$$
 Then $A<_{\ex} z_0 <_{\ex} B$. 
 \end{cor} 
 \begin{proof}  Let $h\in A$. Then $f^{-1}h^\dagger\in A(\beta_{\nx})$ or $f^{-1}h^\dagger \preceq 1$, so
 $f^{-1}h^\dagger <_{\ex} z_1=f^{-1}z_0^\dagger$, and thus $h^\dagger  <_{\ex} z_0^\dagger$.
 Then Lemma~\ref{logintineq} gives $c\in \R^>$ with~$ch <_{\ex} z_0$. 
 Applying this argument to $h^2$ instead of $h$ gives 
 $d\in \R^{>}$ with~$dh^2<_{\ex} z_0$, which in view of~$h<_{\ex} dh^2$ gives $h<_{\ex} z_0$. 
 In the same way one shows that if $h\in B$, then~$z_0 <_{\ex} h$. 
 \end{proof} 

\noindent
 Likewise, using the second part of Lemma~\ref{lemdir}:

\begin{cor} \label{bz0z1} If $\beta_{\nx}>0$ and $z_0, z_1\in \Cc^{<\infty}$ are such that
 $$z_0, z_1 >_{\ex} 0,\qquad z_0^\dagger\ =\ f/z_1, \qquad A(\beta_{\nx}) <_{\ex} z_1 <_{\ex} B(\beta_{\nx}),$$
 then $A<_{\ex} z_0 <_{\ex} B$. 
 \end{cor}

\subsection*{What remains to be done} 
Let $H\supseteq \R$ be a Liouville closed Hardy field with $H$-couple $(\Gamma,\psi)$ over $\R$, and let $\alpha$, $\gamma$ range over $\Gamma$. Suppose
$\beta$ in an extension of~$(\Gamma,\psi)$ is of countable type over $\Gamma$ and falls under Case 
\textup{(b)}, with $(\alpha_i)$, $(\beta_i)$ as in~\textup{(b)}, and $\beta_i^\dagger < 0$, $\alpha_{i+1}\le 0$ for all $i$.  Assume also that $\alpha < \beta < \gamma$ for some~$\alpha$,~$\gamma$. 
Then $\big(\Gamma\<\beta\>, \psi_{\beta}\big)$ is of Hahn type, by the last claim in Lemma~\ref{lemb}, and so 
 all $\beta_i$ lie in this extension of $(\Gamma,\psi)$ of Hahn type; this is significant because of the initial assumption on $\beta$ in the previous subsection. Note also that for all $i$ there are~$\alpha$,~$\gamma$
 with $\alpha < \beta_i < \gamma$, and that by Lemma~\ref{negbeta} we have $\beta_i^{\<n\>}<0$ for all $i$ and all $n\ge 1$. Thus we can apply the previous subsection to each
 $\beta_i$ in the role of $\beta$ there.  Set
$$A_i\ :=\ \big\{h\in H^{>\R}:\,vh> -|\beta_i|\big\},\qquad B_i\ :=\ \big\{h\in H: vh < -|\beta_i|\big\},$$
so $A_i=A(\beta_i)$, $B_i=B(\beta_i)$, and $\beta_{i+1}=(\beta_i)_{\nx}$
in the notation of the previous subsection. Thus by lemmas in that subsection: \begin{enumerate}
\item[(i)] $\ex_n\in A_i$ for all $i$, $n$; 
\item[(ii)] $A_i$ and $\sq(A_i)$ are cofinal, and $B_i$ and $\sqrt{B_i}$ are coinitial;
\item[(iii)] $h\in A_i\Rightarrow h^{\ex_n}\in A_i$, and $h\in B_i\Rightarrow h^{1/\!\ex_n}\in B_i$;
\end{enumerate}
By Corollary~\ref{cfomega} we have $\cf(\Gamma^{<\beta_i})=\ci(\Gamma^{>\beta_i})=\omega$ for all $i$.  Hence $\cf(A_i)=\ci(B_i)=\omega$ for all $i$. What remains to be done is to show the existence of a $y>0$ in a Hardy field extension of $H$ such that $vy$ realizes the same cut in $\Gamma$ as $\beta$.

\medskip\noindent
For each $i$, take $f_i\in H^{>}$ with $vf_i=\alpha_i$, and $f_i\ge 1$ for $i\ge 1$. To get the right idea for our reverse engineering,
suppose $y>0$ is $H$-hardian and
$vy$ realizes the same cut in $\Gamma$ as $\beta$. 
As in the proof of Proposition~\ref{prbcon}, let $y_i\in H\<y\>$ be given by~$y_0:= y/f_0$, and $y_{i+1} = y_i^\dagger/f_{i+1}$. Then
$vy_i$ realizes the same cut in $\Gamma$ as $\beta_i$, so~$y_i\succ 1$ if $\beta_i<0$ and
$y_i \prec 1$ if $\beta_i>0$. To have only positive infinite germs, set 
$$ z_i\ :=\  |y_i|\  \text{ if }\ \beta_i < 0, \qquad z_i\ :=\  |y_i|^{-1}\ \text{ if }\ \beta_i >0.$$
One verifies easily that then $A_i <_{\ex} z_i <_{\ex} B_i$, and
$$ \beta_{i+1} < 0\ \Longrightarrow\  z_i^\dagger\, =\, f_{i+1}z_{i+1}, \qquad \beta_{i+1} > 0\ \Longrightarrow\ z_i^\dagger\, =\, f_{i+1}/z_{i+1}.    $$
We first deal with a ``wide gap'' case:

\begin{lemma} \label{bwg} Suppose for some $n$ there is no $\gamma$ with $[\beta_n] > [\gamma] > [\beta_{n+1}]$.
Then there exists $H$-hardian $y>0$ such that
$vy$ realizes the same cut in $\Gamma$ as $\beta$. 
\end{lemma} 
\begin{proof} Let $n$ be as in the hypothesis. Then $A_n$, $B_n$ is a wide gap by Lemma~\ref{ABwide}, hence Section~\ref{fwg} gives an $H$-hardian
$z_n$ such that~${A_n <_{\ex} z_n <_{\ex}  B_n}$. 
Let $L:=\Li\!\big(H\<z_n\>\big)$. Then we have~$z_{n-1},\dots, z_0\in L^{>}$ such that for all $i<n$: 
$$ \beta_{i+1} < 0\ \Longrightarrow \ z_i^\dagger\, =\, f_{i+1}z_{i+1}, \qquad \beta_{i+1} > 0\ \Longrightarrow\ z_i^\dagger\, =\, f_{i+1}/z_{i+1}.$$
Downward induction on $i$ using Corollaries~\ref{az0z1}  and~\ref{bz0z1} then gives $A_i <_{\ex} z_i <_{\ex} B_i$ for all $i\le n$.
Thus if $\beta_0<0$, then $v(z_0)$ realizes the same gap in $\Gamma$ as $\beta_0$, and so~$y:= f_0z_0$ has the desired property. If $\beta_0 >0$, then $v(z_0)$ realizes the same gap in $\Gamma$ as
$-\beta_0$, and so  $y:= f_0/z_0$ has the desired property.
\end{proof}

\noindent
{\em It remains to consider the case that for all $i$ there exists $\gamma$ with 
$[\beta_i] > [\gamma] > [\beta_{i+1}]$. We assume this for the rest of this section}. 
The goal of our reverse engineering will be to construct germs $z_i$ as in the next lemma:

\begin{lemma}\label{reveng1} Let the germs $z_i\in \Cc^{<\infty}$ be such that  for all $i$, $A_i <_{\ex} z_i <_{\ex} B_i$ and 
$$ \beta_{i+1} < 0\ \Longrightarrow\  z_i^\dagger\ =\ f_{i+1}z_{i+1}, \qquad \beta_{i+1} > 0\ \Longrightarrow\ z_i^\dagger\ =\ f_{i+1}/z_{i+1}.$$
Then there exists $H$-hardian $y>0$ such that
$vy$ realizes the same cut in $\Gamma$ as $\beta$. 
\end{lemma}
\begin{proof}  Note that for each $n$ we have the ordered subgroup 
$\Gamma\oplus \Z\beta_0 \oplus \cdots \oplus \Z\beta_n$ of~$\Gamma\<\beta\>$, and likewise with
$\Q$ instead of $\Z$.  We prove by induction on~$n$ that~$z_0,\dots, z_n$ generate a Hausdorff field
$H_n:= H(z_0,\dots, z_n)$ over $H$, with
$$v(H_n^\times)\ =\ \Gamma \oplus \Z vz_0 \oplus \cdots \oplus \Z vz_n,$$
and with an ordered abelian group isomorphism that is the identity on $\Gamma$:
$$\Gamma\oplus \Z\beta_0\oplus \cdots \oplus \Z \beta_{n}\ \to\  \Gamma \oplus \Z vz_0\oplus\cdots\oplus \Z vz_{n}, \quad\ -|\beta_i|\mapsto 
vz_i \quad  (i=0,\dots,n).$$
For $n=0$ this follows from Lemma~\ref{lem:5.1.18}. Assume that the above holds for a certain~$n$. Then for the real closure $H_n^{\rc}$ of $H_n$ as a Hausdorff field extension of $H_n$, $$v\big(H_n^{\rc,\times}\big)\ =\ \Gamma \oplus \Q vz_0 \oplus \cdots \oplus \Q vz_n,$$
with an ordered abelian group isomorphism that is the identity on $\Gamma$:
$$\Gamma\oplus \Q\beta_0\oplus \cdots \oplus \Q \beta_{n}\ \to\  \Gamma \oplus \Q vz_0\oplus\cdots\oplus \Q vz_{n}, \quad\ -|\beta_i|\mapsto 
vz_i \quad  (i=0,\dots,n).$$
Thus $\big[v\big(H_n^{\rc,\times}\big)\big]=[\Gamma]\cup \big\{[vz_0], \dots, [vz_n]\big\}$ by Lemma~\ref{lemb}. 

\claim{For each $f\in H_n^{\rc,>}$, either
$f\preceq h$ for some $h\in A_{n+1}$, or $f\succeq h$ for some~${h\in B_{n+1}}$.}

\noindent 
Otherwise we have
 $f\in H_n^{\rc}$ with $A_{n+1}< f < B_{n+1}$, so $[vf]\notin [\Gamma]$ and $vf$ realizes the same cut in $\Gamma$ as $-|\beta_{n+1}|$. Taking $\gamma$ with
$[\beta_n] > [\gamma] > [\beta_{n+1}]$ we ob\-tain~${[vz_0]> \cdots > [vz_n] >[\gamma]>[vf]}$, contradicting
$vf\in v\big(H_n^{\rc,\times}\big)$.

\medskip
\noindent
The claim and Lemma~\ref{lem:5.1.18}
 give a Hausdorff field extension $H_n^{\rc}(z_{n+1})$ of~$H_n^{\rc}$, and the resulting Hausdorff  field extension $H_{n+1}=H_n(z_{n+1})$ of $H$ has the properties that the inductive step requires. This concludes the proof by induction. 

An easy induction on $n$ now shows that for $z:= z_0$ the elements $z, z',\dots, z^{(n)}$ of~$\Cc^{<\infty}$ generate the Hausdorff field $H(z, z',\dots, z^{(n)})=H_n$ over $H$,  and so we have a Hardy field
$H\<z\>$ over $H$. If $\beta_0<0$, then $y:= f_0z_0$ has the desired property, and if $\beta_0 >0$, then
$y:= f_0/z_0$ has the desired property. 
\end{proof}

\subsection*{First step in reverse engineering} To construct germs $z_i$ as in Lemma~\ref{reveng1} we first
take for each $i$ a continuous function $[0,+\infty)\to \R^{>}$ that represents the germ~$f_i\in H$ and to be denoted also by $f_i$, and with $f_i\ge 1$ on $[0,+\infty)$ for $i\ge 1$.  Next, let $(a_i)$ be a strictly increasing sequence of real
numbers $\ge 0$ tending to~$+\infty$ such that $f_0,\dots, f_m$  are of class $\Cc^m$ on $[a_m,+\infty)$. Let there also be given for each $m\ge1$ a continuous function
$z_{m-1,m}\colon [a_{m-1}, a_{m}]\to \R^{>}$.  Then we define the continuous
function $z_{k,m}\colon [a_k, a_{m}]\to \R^{>}$ for $0\le k < m$ by downward
recursion: $z_{m-1,m}$ for $m\ge 1$ is already given to us, and for $1\le k < m$,
$$z_{k-1,m}(t)\ :=\ \begin{cases}\ z_{k-1,k}(t)& \text{for $a_{k-1}\le t\le a_k$,}\\
 \ z_{k-1,k}(a_k)\cdot\exp \displaystyle\int_{a_k}^t f_k(s)z_{k,m}(s)\,ds& \text{for $a_k \le t\le a_m$, if $\beta_k<0$,}\\[1em]
 \ z_{k-1,k}(a_k)\cdot\exp \displaystyle\int_{a_k}^t \frac{f_k(s)}{z_{k,m}(s)}\, ds&\text{for $a_k \le t\le a_m$, if $\beta_k>0$.}\end{cases}
$$
Downward  induction on $k$ gives $z_{k,m}=z_{k,m+1}$ on $[a_k, a_m]$ for 
$k < m$. This fact gives for each
$k\in \N$ a continuous function $z_k\colon [a_k, +\infty)\to \R^{>}$
such that $z_k=z_{k,m}$ on $[a_k,a_m]$, for all $m>k$. Thus
for $k\ge 1$ and $t\ge a_k$ we have
\begin{align*} 
 \beta_k<0\ &\Longrightarrow\ z_{k-1}(t)\  =\ z_{k-1}(a_{k})\cdot \exp \int_{a_{k}}^t f_k(s)z_{k}(s)\, ds,\\
 \beta_k >0\ &\Longrightarrow\ z_{k-1}(t)\ =\ z_{k-1}(a_{k})\cdot \exp \int_{a_{k}}^t \frac{f_k(s)}{z_{k}(s)}\, ds,
 \end{align*} 
so  $z_{k-1}$ is of class $\Cc^1$ on $[a_k, +\infty)$, and:
\begin{align*}
 \beta_k<0\ &\Longrightarrow\  z_{k-1}^\dagger=f_k\,z_k \text{ on $[a_k, +\infty)$,}\\
\beta_k>0\  &\Longrightarrow\  z_{k-1}^\dagger=f_k/z_k \text{ on $[a_k, +\infty)$.}
\end{align*} 
Hence induction on $m$ gives that $z_k$ is of class $\Cc^m$ on $[a_{k+m},+\infty)$ (for all $k$, $m$),
and thus (the germ of)  each $z_k$ lies in $\Cc^{<\infty}$. 

The above is a general construction of functions whose germs satisfy the {\em equalities\/} in Lemma~\ref{reveng1}.  More work is needed to satisfy also the {\em inequalities\/} $A_i <_{\ex} z_i <_{\ex} B_i$ in that lemma. We now turn to this task. 

\subsection*{Second step in reverse engineering} 
Assume in this subsection that
$z_{k-1,k}\ge 1$ on $[a_{k-1},a_k]$, for all $k\ge 1$. Then $z_k\ge 1$ on $[a_k, +\infty)$ for all $k$. For $k\ge 1$ we have~$z_{k-1}(t), z_{k-1}^\dagger(t) >0$ for all $t\ge a_k$, so
$z_{k-1}$ is strictly increasing on $[a_{k}, +\infty)$.  
For each $k$, let $p_k, q_k\colon [a_k, +\infty) \to \R^{>}$ be continuous functions such that  $$p_{m-1}\ \le\ z_{m-1,m}\ \le\ q_{m-1}\ \text{ on  $[a_{m-1},a_m$],  for all $m\ge 1$.}$$
We try to find conditions on the families $(p_k)$ and $(q_k)$ so that these inequalities extend to 
$p_k \le z_{k,m} \le q_k$ on $[a_k, a_m]$ for all $k$, $m$ with $k<m$
(and thus $p_k \le z_k \le q_k$ on $[a_k, +\infty)$ for all $k$). Let $1\le k < m$ and assume inductively
that $p_k \le z_{k,m} \le q_k$ on $[a_k, a_m]$. On $[a_{k-1},a_k]$ we have
$p_{k-1}\le z_{k-1,k}\le q_{k-1}$, so $p_{k-1} \le z_{k-1,m}\le q_{k-1}$, as desired. 

\medskip
\noindent
First suppose $\beta_k<0$. Then for $a_k\le t\le a_m$ we have
$$z_{k-1,m}(t)\ =\  z_{k-1,k}(a_k)\exp \int_{a_k}^t f_k(s)z_{k,m}(s)ds,$$ hence
$$p_{k-1}(a_k)\exp\int_{a_k}^t f_k(s)p_k(s)ds\ \le\ z_{k-1,m}(t)\ \le\ q_{k-1}(a_k)\exp\int_{a_k}^t f_k(s)q_k(s)ds,$$
and so the desired $p_{k-1} \le z_{k-1,m} \le q_{k-1}$ on $[a_{k-1}, a_m]$ would follow if
\begin{align} 
\label{eq:Ik}\tag{${\rm{I}}_k$}
       p_{k-1}(t)\ &\le\ p_{k-1}(a_k)\exp\int_{a_k}^t f_k(s)p_k(s)\,ds&\text{for all $t\ge a_k$,}\\
\label{eq:IIk}\tag{${\rm{II}}_k$}
      q_{k-1}(t)\ &\ge\ q_{k-1}(a_k)\exp\int_{a_k}^t f_k(s)q_k(s)\,ds&\text{for all $t\ge a_k$.} 
\end{align}
Now assume $\beta_k >0$. Then for $a_k\le t\le a_m$ we have
$$z_{k-1,m}(t)\ =\ z_{k-1,k}(a_k)\exp \int_{a_k}^t \frac{f_k(s)}{z_{k,m}(s)}\,ds,$$ and so 
$$p_{k-1}(a_k)\exp\int_{a_k}^t \frac{f_k(s)}{q_k(s)}\,ds\ \le\ z_{k-1,m}(t)\ \le\ q_{k-1}(a_k)\exp\int_{a_k}^t \frac{f_k(s)}{p_k(s)}\,ds,$$
and so the desired $p_{k-1} \le z_{k-1,m} \le q_{k-1}$ on $[a_{k-1}, a_m]$ would follow if 
\begin{align} 
\label{eq:IIIk}\tag{${\rm{III}}_k$}
 p_{k-1}(t)\ &\le\ p_{k-1}(a_k)\exp\int_{a_k}^t \frac{f_k(s)}{q_k(s)}\,ds&\text{for all $t\ge a_k$,}\\
\label{eq:IVk}\tag{${\rm{IV}}_k$}
    q_{k-1}(t)\ &\ge\ q_{k-1}(a_k)\exp\int_{a_k}^t \frac{f_k(s)}{p_k(s)}\,ds&\text{for all $t\ge a_k$.}
\end{align}

\noindent
The above leads to the following:

\begin{lemma} \label{piqi} Let $p_i\in A_i$ and $q_i\in B_i$ for $i=0,1,2,\dots$ be given such that
\begin{align*} 
\beta_{i+1}\  <\  0\ &\Longrightarrow\ p_i^\dagger\  \le \  f_{i+1}p_{i+1},&& \hskip-4em  q_i^\dagger\ \ge\  f_{i+1}q_{i+1}\ & \hskip-3em\text{\textup{(}in $H$\textup{)},}\\
\beta_{i+1}\ >\ 0\ &\Longrightarrow\ p_i^\dagger\  \le\  f_{i+1}/q_{i+1},&&  \hskip-4em q_i^\dagger\  \ge\  f_{i+1}/p_{i+1}\ &\hskip-3em\text{\textup{(}in $H$\textup{)}.}
\end{align*}
Then there are germs $z_i\in \Cc^{<\infty}\ (i=0,1,2,\dots)$ such that for all $i$,  
$$p_i \le z_i \le q_i\  \text{ in }\Cc, \quad  \beta_{i+1} < 0\ \Longrightarrow\  z_i^\dagger = f_{i+1}z_{i+1}, \quad \beta_{i+1} > 0\ \Longrightarrow \ z_i^\dagger= f_{i+1}/z_{i+1}.$$
\end{lemma} 
\begin{proof} Take a strictly increasing sequence $(a_i)$ of real numbers $\ge 0$ tending to $+\infty$ and
represent $p_i$, $q_i$ for each $i$ by $\Cc^1$-functions $[a_i, +\infty)\to \R^{>}$, also to be denoted by $p_i$, $q_i$,
such that for all $m$,
\begin{itemize}
\item $f_0,\dots, f_m$ are of class $\Cc^m$ on $[a_m,+\infty)$;
\item $1\le p_m \le  q_m$ on  $[a_m,+\infty)$;
\item $\beta_{m+1}<0\ \Longrightarrow\ p_m^\dagger \le   f_{m+1}p_{m+1},\quad  q_m^\dagger\ge  f_{m+1}q_{m+1}$ on $[a_{m+1},+\infty)$; 
\item $\beta_{m+1}>0\ \Longrightarrow\ p_m^\dagger \le   f_{m+1}/q_{m+1},\ q_m^\dagger\ge  f_{m+1}/p_{m+1}$ on $[a_{m+1},+\infty)$.
\end{itemize} 
Upon replacing $(a_i)$ by a strictly increasing sequence $(b_i)$ of reals with $a_i \le b_i$ for all $i$ and the
$p_i$, $q_i $ by their restrictions to $[b_i,+\infty)$, for each $i$,  the conditions above are obviously still satisfied.
 For $t\ge a_m$ we have $$p_m(t)\  =\ p_m(a_{m+1})\exp\int_{a_{m+1}}^t p_m^\dagger(s)\,ds,\qquad q_m(t)\  =\ q_m(a_{m+1})\exp\int_{a_{m+1}}^t q_m^\dagger(s)\,ds,$$ and so
for all $k\ge 1$ conditions  \eqref{eq:Ik}  and  \eqref{eq:IIk}  are satisfied if $\beta_k<0$, and conditions~\eqref{eq:IIIk}  and  \eqref{eq:IVk}  are satisfied if $\beta_k >0$. Thus by the above we can take any continuous function
$z_{m-1,m}\colon [a_{m-1},a_m]\to \R$  with $p_{m-1}\le z_{m-1,m}\le q_{m-1}$ on~$[a_{m-1},a_m]$ for $m=1,2,\dots$
to give germs $z_i$ for $i=0,1,\dots$ as required. 
\end{proof}

\subsection*{Final step in reverse engineering}  This step involves a diagonalization. We take~$p_{i,n}\in A_i$ and $q_{i,n}\in B_i$ (for $i=0,1,2,\dots$, $n=0,1,2,\dots$) such that for all~$i$,~$n$: \begin{itemize}
\item $p_{i,n}\prec p_{i,n+1}$, and $\{p_{i,0}, p_{i,1}, p_{i,2},\dots\}$ is cofinal in $A_i$;
\item $q_{i,n}\succ q_{i,n+1}$, and $\{q_{i,0}, q_{i,1}, q_{i,2},\dots\}$ is coinitial in $B_i$;
\item if $\beta_{i+1}< 0$, then $p_{i+1,n}= p_{i,N}^\dagger/f_{i+1}$
and $q_{i+1,n}= q_{i,N}^\dagger/f_{i+1}$ for some~$N=N(i,n)>n$;
\item if $\beta_{i+1} >0$, then $p_{i+1,n}= f_{i+1}/q_{i,N}^\dagger$
and $q_{i+1,n}= f_{i+1}/p_{i,N}^\dagger$ for some~$N=N(i,n)>n$.
\end{itemize}
It follows from Lemma~\ref{lemdir} that there is such a family $\big((p_{i,n}, q_{i,n})\big)$. 
Setting $p_i:= p_{i,i}$ and $q_i:= q_{i,i}$ we note that the hypotheses of Lemma~\ref{piqi} are satisfied, and this gives us germs $z_i\in \Cc^{<\infty}$ for $i=0,1,2,\dots$ such that for all $i$, \begin{enumerate}
\item $p_{i} \le  z_i \le  q_{i}$ in $\Cc$; 
\item $ \beta_{i+1} < 0 \Longrightarrow  z_i^\dagger = f_{i+1}z_{i+1}, \qquad \beta_{i+1} > 0 \Longrightarrow z_i^\dagger= f_{i+1}/z_{i+1}$.
\end{enumerate}
We claim that then $A_i <_{\ex} z_i <_{\ex} B_i$ for all $i$. (Establishing this claim achieves our goal by Lemma~\ref{reveng1}.)
To prove this claim, suppose for a  certain pair $i$,~$n$ with~$i<n$ we have 
$p_{i+1,n} \le z_{i+1} \le  q_{i+1,n}$. (See Figure~\ref{fig:diag}.) As a subclaim we show that then~$p_{i,n} \le  z_i \le q_{i,n}$. 
Consider first the case $\beta_{i+1}<0$. Then for~$N:= N(i, n)>n$,
$$  f_{i+1}^{-1}p_{i, N}^\dagger\  \le\  z_{i+1}\ =\ f_{i+1}^{-1}z_i^{\dagger}\ \le\  f_{i+1}^{-1}q_{i,N}^\dagger,\ \text{ so }\
p_{i,N}^\dagger\ \le  z_i^\dagger\ \le \ q_{i,N}^\dagger,$$
which by Lemma~\ref{logintineq} gives constants $c_1, c_2>0$ with $c_1p_{i,N} \le  z_i \le c_2q_{i,N}$, and so~$p_{i,N-1} \le z_i \le  q_{i,N-1}$. Now $N-1\ge n$, and thus $p_{i,n} \le  z_i \le  q_{i,n}$ as promised. 
The case $\beta_{i+1}>0$ is handled in the same way. Given $i<n$ we have $p_{n,n}\le  z_n \le q_{n,n}$, and so
$p_{i,n} \le  z_i \le q_{i,n}$ by iterated application of the subclaim. For any fixed $i$ this yields
$A_i <_{\ex} z_i <_{\ex} B_i$ by the cofinality and coinitiality requirements we imposed on the $p_{i,n}$ and $q_{i,n}$.
This proves the claim, and concludes the proof of Theorem~\ref{bconstruction}, and thus of Theorem~\ref{mt}.  \qed

\begin{figure}[ht]
$${\setlength{\arraycolsep}{0.1em}
\begin{array}{ ccccccccccccccccccc }
  \fbox{$p_{0,0}$} & \prec & p_{0,1} & \prec &  \cdots & \prec  & p_{0,n} &\prec &\cdots & z_0 & \cdots&\prec &q_{0,n} &\prec &\cdots & \prec & q_{0,1} & \prec & \fbox{$q_{0,0}$}  \\
   p_{1,0} & \prec & \fbox{$p_{1,1}$} & \prec &  \cdots & \prec  & p_{1,n} &\prec &\cdots & z_1 & \cdots&\prec &q_{1,n} &\prec &\cdots & \prec & \fbox{$q_{1,1}$} & \prec & q_{1,0}  \\
     \vdots & & \vdots &\ddots &   &   & \vdots & & & \vdots & &   &  \vdots &   &  & \iddots  &  \vdots  & & \vdots  \\
  \vdots & &  \vdots & &   & \ddots  & \vdots & & & \vdots & &   &  \vdots &  \iddots &  &   &  \vdots  &  & \vdots  \\
   p_{n,0} & \prec & p_{n,1} & \prec &  \cdots & \prec  & \fbox{$p_{n,n}$} &\prec &\cdots & z_n & \cdots&\prec & \fbox{$q_{n,n}$} &\prec &\cdots & \prec & q_{n,1} & \prec & q_{n,0}  \\
  \vdots & &  \vdots & &   &   & \vdots  & \ddots & & \vdots & & \iddots  & \vdots   &   &  &   &  \vdots  &  & \vdots  
\end{array}}$$
\caption{}\label{fig:diag}
\end{figure}

\section{Isomorphism of Maximal Hardy Fields}\label{sec:iso} 

\noindent
The cardinality of any Hardy field extending $\R$ is $2^{\aleph_0}$. By Theorem~\ref{mt}, all maximal Hardy fields are $\eta_1$ and thus 
$\aleph_1$-saturated as real closed ordered fields; in particular,
under $\operatorname{CH}$
they are all isomorphic as ordered fields. However, they are not 
$\aleph_1$-saturated as ordered {\em differential\/} fields, 
since their constant field $\R$ isn't. Thus to show they are
isomorphic (under $\operatorname{CH}$), we need to argue in a different way, and this is what we do in this section. 

\begin{lemma}\label{iso1} 
Let $K$ be a countable closed $H$-field with
archimedean constant field~$C$, let
$L$ be a closed $H$-field with constant field $\R$, and   assume $L$ is~$\eta_1$.
Let $E$ be 
an $\upo$-free $H$-subfield of $K$, and let $i\colon E \to L$
be an $H$-field embedding. Then~$i$ extends to an $H$-field embedding $K \to L$. 
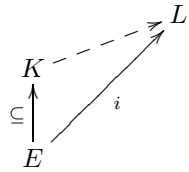
$$\xymatrix@C=4em@R=0.75em{ 
& L \\
K \ar@{-->}[ur] &   \\ \\
E \ar[uu]^{\subseteq} \ar[uuur]_i }$$
\end{lemma}
\begin{proof} We identify $C$ in the usual (and only possible) way
with a subfield of $\R$, and note that then
$i$ is the identity on $C_E\subseteq \R$. Then  [ADH, 10.5.15, 10.5.16]  yield an extension
of $i$ to an $H$-field embedding $E(C) \to L$ that is the
identity on~$C\subseteq \R$. The $H$-subfield $E(C)$ of $K$ is $\d$-algebraic over $E$, so is $\upo$-free. Replacing~$E$ by~$E(C)$
we reduce to the case that $C_E=C$. Recall that $\eta_1$-ordered sets 
are $\aleph_1$-saturated.
Hence  [ADH, 16.2.3]
applies and gives the desired conclusion.  
\end{proof}

\begin{lemma}\label{iso2} Let $L_1$, $L_2$ be closed $H$-fields with small
derivation and common constant field $\R$, and assume that $L_1$ and $L_2$ are $\eta_1$. Then the collection of $H$-field isomorphisms $K_1\to K_2$ between countable   closed $H$-subfields $K_1$ of $L_1$ and $K_2$ of~$L_2$ is nonempty and is a
back-and-forth system between $L_1$ and $L_2$. In particular, $L_1$ and $L_2$ are back-and-forth equivalent.
\end{lemma} 
\begin{proof} 
The theory of closed $H$-fields with small derivation has a (countable) prime model, by [ADH, p.~705], and  so there is an $H$-field isomorphism between copies of that prime model in $L_1$ and in $L_2$. 
Also, any countable subset of a closed $H$-field~$L$ is contained in a countable closed $H$-subfield of $L$, by downward L\"owenheim-Skolem~[ADH, B.5.10].
It remains to use Lemma~\ref{iso1}. 
\end{proof}

\noindent
A standard argument (cf.~proof of [ADH,   B.5.3]) using Lem\-ma~\ref{iso2} now yields: 

\begin{cor}\label{iso3} Let $L_1$, $L_2$ as in Lemma~\ref{iso2}
have cardinality $2^{\aleph_0}$. Assume $\operatorname{CH}$. Then  $L_1$ and $L_2$ are isomorphic as $H$-fields.
\end{cor}

\noindent
Next we recall that Berarducci and Mantova~\cite{BM} defined
a derivation $\der_{\BM}$ on the real closed 
field $\No$ of surreal numbers and proved that $\No$ with $\der_{\BM}$ is a Liouville closed $H$-field with $\R\subseteq \No$ as its field of constants. Below we consider $\No$ as an $H$-field
in this way, and recall also that its derivation $\der_{\BM}$ is small. We proved in
\cite[Theorems~1 and~2]{ADH1+} that $\No$ is even a closed $H$-field, that its real closed subfield~$\No(\omega_1)$ is closed under $\der_{\BM}$, and that $\No(\omega_1)$ as a differential
subfield of~$\No$ is a closed $H$-field as well. Moreover, $\No(\omega_1)$ has cardinality~$2^{\aleph_0}$, and is~$\eta_1$ as an ordered set. In combination with
Theorem~\ref{mt} and Corollary~\ref{iso3}, with~$L_1$ any maximal Hardy field and $L_2=\No(\omega_1)$, this yields
Corollary~\ref{cormt}  in the introduction; more precisely, also using \cite[Theorem~3]{Barwise}, we have:  

\begin{cor}\label{iso4}
Let $M$ be a maximal Hardy field. Then the ordered differential fields $M$ and  $\No(\omega_1)$ are back-and-forth equivalent. Hence $M$ and $\No(\omega_1)$  are $\infty\omega$-equivalent, and
assuming~$\operatorname{CH}$,   $M$ and $\No(\omega_1)$ are isomorphic.
\end{cor}

\noindent
We finish with a lemma on $\infty\omega$-elementary embeddings:

\begin{lemma}\label{iso5} Let $L_1$ and $L_2$ be as in Lemma~\ref{iso2}, with $L_1$ an $H$-subfield of $L_2$. 
Then $L_1 \preceq_{\infty\omega} L_2$.
\end{lemma}
\begin{proof} Let $\Phi$ be the back-and-forth system from  Lemma~\ref{iso2}.
By downward L\"owenheim-Skolem there is for all~$a_1,\dots,a_n\in L_1$ a countable closed $H$-subfield~$K_1$ of~$L_1$
containing~$a_1,\dots,a_n$, and then the identity map  $K_1\to K_1$ belongs to $\Phi$. 
This yields~${L_1 \preceq_{\infty\omega} L_2}$ by~\cite[Theorem~4]{Barwise}.
\end{proof}

\noindent
This  will be used in the follow-up paper on maximal analytic Hardy fields. 



\newlength\templinewidth
\setlength{\templinewidth}{\textwidth}
\addtolength{\templinewidth}{-2.25em}

\patchcmd{\thebibliography}{\list}{\printremarkbeforebib\list}{}{}

\let\oldaddcontentsline\addcontentsline
\renewcommand{\addcontentsline}[3]{\oldaddcontentsline{toc}{section}{References}}

\def\printremarkbeforebib{\bigskip\hskip1em The citation [ADH] refers to our book \\

\hskip1em\parbox{\templinewidth}{
M. Aschenbrenner, L. van den Dries, J. van der Hoeven,
\textit{Asymptotic Differential Algebra and Model Theory of Transseries,} Annals of Mathematics Studies, vol.~195, Princeton University Press, Princeton, NJ, 2017.
}

\bigskip

}

\end{document}